\documentclass[12pt]{article}
\usepackage{algorithm}
\usepackage{algpseudocode}
\usepackage{amsfonts}
\usepackage{amsmath}
\usepackage{amssymb} 
\usepackage{amsthm} 
\usepackage[title]{appendix} 
\usepackage{autobreak}
\usepackage{bigdelim} 
\usepackage{bm}
\usepackage{breqn}
\usepackage{cancel} 
\usepackage{cases}
\usepackage{chngcntr}
\usepackage{colortbl} 
\usepackage{comment}
\usepackage{empheq} 
\usepackage{enumerate}
\usepackage{fancyhdr} 
\usepackage{fancyvrb} 
\usepackage{framed} 
\usepackage{graphicx}
\usepackage{hyperref}
\hypersetup{
  colorlinks,
  linkcolor={red!50!black},
  citecolor={blue!50!black},
  urlcolor={blue!80!black}
}
\usepackage{mathrsfs} 
\usepackage{multirow} 
\usepackage{setspace} 
\usepackage{subfig} 
\usepackage{pdflscape} 
\usepackage{pgfplots} 
\usepackage{pxrubrica} 
\usepackage{ulem} 
\usepackage{url}

\usepackage[top=35truemm,bottom=35truemm,left=30truemm,right=30truemm]{geometry} 
\algtext*{EndWhile}
\algtext*{EndIf}
\algtext*{EndFor}
\newtheorem{theorem}{Theorem}[section]
\newtheorem{proposition}[theorem]{Proposition}%
\newtheorem{corollary}[theorem]{Corollary}
\newtheorem{lemma}[theorem]{Lemma}

\newtheorem{example}[theorem]{Example}%
\newtheorem{remark}[theorem]{Remark}%
\newtheorem{definition}[theorem]{Definition}%

\DeclareMathOperator*{\conv}{conv}
\DeclareMathOperator*{\cl}{cl}

\DeclareMathOperator*{\setint}{int}

\DeclareMathOperator*{\cone}{cone}
\DeclareMathOperator*{\diag}{diag}

\DeclareMathOperator*{\svec}{svec}

\DeclareMathOperator*{\Hom}{Hom}
\DeclareMathOperator*{\rk}{rk}
\DeclareMathOperator*{\smat}{smat}
\DeclareMathOperator{\subjectto}{subject~to}
\newcommand{\SOS}{\mathrm{SOS}}
\newcommand{\MOM}{\mathrm{MOM}}

\newcommand{\minimize}{\mathop{\rm minimize}\limits}
\newcommand{\maximize}{\mathop{\rm maximize}\limits}
\newcommand{\relmiddle}[1]{\mathrel{}\middle#1\mathrel{}}

\def\CP{\mathcal{CP}}
\def\COP{\mathcal{COP}}

\title{Approximation hierarchies for copositive cone over symmetric cone and their comparison}

\makeatletter
\let\@fnsymbol\@arabic
\makeatother

\author{
\normalsize
    Mitsuhiro Nishijima\thanks{Department of Industrial Engineering and Economics,
        Tokyo Institute of Technology, 2-12-1 Ookayama, Meguro-ku, 1528552, Tokyo, Japan.
        ({\tt nishijima.m.ae@m.titech.ac.jp, nakata.k.ac@m.titech.ac.jp}).}
\and
\normalsize
        Kazuhide Nakata\footnotemark[1]
        }

\begin{document}
\maketitle

\begin{abstract}\noindent
We first provide an inner-approximation hierarchy described by a sum-of-squares (SOS) constraint for the copositive (COP) cone over a general symmetric cone.
The hierarchy is a generalization of that proposed by Parrilo (2000) for the usual COP cone (over a nonnegative orthant).
We also discuss its dual.
Second, we characterize the COP cone over a symmetric cone using the usual COP cone.
By replacing the usual COP cone appearing in this characterization with the inner- or outer-approximation hierarchy provided by de Klerk and Pasechnik (2002) or Y{\i}ld{\i}r{\i}m (2012), we obtain an inner- or outer-approximation hierarchy described by semidefinite but not by SOS constraints for the COP matrix cone over the direct product of a nonnegative orthant and a second-order cone.
We then compare them with the existing hierarchies provided by Zuluaga et al. (2006) and Lasserre (2014).
Theoretical and numerical examinations imply that we can numerically increase a depth parameter, which determines an approximation accuracy, in the approximation hierarchies derived from de Klerk and Pasechnik (2002) and Y{\i}ld{\i}r{\i}m (2012), particularly when the nonnegative orthant is small.
In such a case, the approximation hierarchy derived from Y{\i}ld{\i}r{\i}m (2012) can yield nearly optimal values numerically.
Combining the proposed approximation hierarchies with existing ones, we can evaluate the optimal value of COP programming problems more accurately and efficiently.
\end{abstract}
\vspace{0.5cm}


\section{Introduction}\label{sec:intro}
In this study, we focus on the copositivity and its dual, the complete positivity of tensors, which include matrices and, more generally, linear transformations, over a symmetric cone.
Typically, the nonnegative orthant, second-order cone, and semidefinite cone and their direct product are symmetric cones.
A symmetric cone plays a significant role in optimization~\cite{Faybusovich1997} and often appears in the modeling of realistic problems~\cite{MT2015,NN2022}.
The copositivity and complete positivity over a nonnegative orthant, i.e., those in the typical sense, are of particular importance.
They have been deeply studied ~\cite{QL2017,SB2021} and exploited in the convex conic reformulation of many NP-hard problems~\cite{BDde2000,Burer2009,PVZ2015,PR2009}.
In addition, the complete positivity over other symmetric cones has been used in the convex conic reformulation of rank-constrained semidefinite programming (SDP)~\cite{BMP2016} and polynomial optimization problems over a symmetric cone~\cite{KKT2020}, whose applications include polynomial SDP~\cite{Kojima2003}, which often appears in system and control theory~\cite{HL2006,HS2004,HS2005}, and polynomial second-order cone programming~\cite{KK2010,LMNe2018}.
Moreover, Gowda~\cite{Gowda2019} discussed (weighted) linear complementarity problems over a symmetric cone, in which the copositivity over a symmetric cone of linear transformations was exploited.
For convenience, the cones of copositive (COP) and completely positive (CP) tensors over a closed cone $\mathbb{K}$ are hereafter called the COP and CP cones (over $\mathbb{K}$), respectively.

As  the copositivity and complete positivity appear in the reformulation of such formidable problems, the COP and CP cones are difficult to handle~\cite{DG2014}.
Thus, to guarantee copositivity or complete positivity, we must consider sufficient and necessary conditions that can be handled efficiently.

To achieve this objective, many types of approximation hierarchies have been proposed~\cite{AM2019,BD2009,dP2002,DP2013,Dong2013,GPS2020,GDFe2014,GL2007,IA2022,Lasserre2014,Parrilo2000,PVZ2007,SBD2012,Yildirim2012,ZF2019,ZVP2006}.
An approximation hierarchy, e.g., $\{\mathcal{K}_r\}_r$, gradually approaches the COP or CP cone from the inside or outside as the depth parameter $r$, which determines the approximation accuracy, increases and, in a sense, agrees with the cone in the limit.
By defining each $\mathcal{K}_r$ as a set represented by nonnegative, second-order cone, or semidefinite constraints, we can tentatively handle the copositivity or complete positivity on a computer using methods such as primal-dual interior point methods~\cite[Chap.~11]{Alizadeh2012}.

Most of the above mentioned works provided approximation hierarchies for the usual COP and CP cones, and few studies~\cite{DP2013,GDFe2014,Lasserre2014,ZVP2006} have considered those for the COP or CP cones over a closed cone $\mathbb{K}$ other than a nonnegative orthant.
Zuluaga et al.~\cite{ZVP2006} provided an inner-approximation hierarchy described by the sum-of-squares (SOS) constraints, which reduce to semidefinite constraints, for the COP cone over a pointed semialgebraic cone.
The term ``semialgebraic'' means that the set is defined by finitely many non-negativity constraints of homogeneous polynomials.
The class of pointed semialgebraic cones includes the nonnegative orthant, second-order cone, and semidefinite cone or their direct product, which are also symmetric cones.
The semidefiniteness of a matrix of size $n$ is characterized by the nonnegativity of all the $2^n - 1$ principal minors.
That is, the semidefinite cone is a pointed semialgebraic cone; however, the semialgebraic representation requires an exponential number of non-negativity constraints in its size.
In such a case, their approximation hierarchy is no longer tractable even if semidefinite constraints can describe it.
Lasserre~\cite{Lasserre2014} provided an outer-approximation hierarchy described by a semidefinite constraint for the COP cone over a general closed convex cone $\mathbb{K}$.
The hierarchy is implementable only for the case in which the moment of a finite Borel measure dependent on $\mathbb{K}$ is obtainable.
The following case of $\mathbb{K}$ being the direct product of a nonnegative orthant and a second-order cone is an example in which the moment can be theoretically obtained.
The study by Dickinson and Povh~\cite{DP2013} is a variant of that of Lasserre~\cite{Lasserre2014}, which considered the special case in which $\mathbb{K}$ is included in a nonnegative orthant to provide a tighter approximation than Lasserre~\cite{Lasserre2014}.

This study aims to provide approximation hierarchies for the COP cone over a symmetric cone and compare them with existing ones.
First, we provide an inner-approximation hierarchy described by an SOS constraint.
It is a generalization of the approximation hierarchy proposed by Parrilo~\cite{Parrilo2000} for the usual COP cone.
We call the proposed approximation hierarchy the NN-type inner-approximation hierarchy.
Moreover, we discuss its dual to provide an outer-approximation hierarchy for the CP cone over a symmetric cone and provide its more explicit expression for the case in which the symmetric cone is a nonnegative orthant.

Second, we characterize the COP cone over a symmetric cone using the usual COP cone.
The basic idea for providing an approximation hierarchy is to replace the usual COP cone appearing in this characterization with its approximation hierarchy.
In general, the induced sequence is defined by the intersection of \textit{infinitely} many sets and is not even guaranteed to converge to the COP cone over a symmetric cone.
However, by exploiting the inner-approximation hierarchy given by de Klerk and Pasechnik~\cite{dP2002} or outer-approximation hierarchy given by Y{\i}ld{\i}r{\i}m~\cite{Yildirim2012}, we obtain an inner- or outer-approximation hierarchy described by \textit{finitely} many semidefinite but not by SOS constraints for the cone of COP matrices (COP matrix cone) over the direct product of a nonnegative orthant and one second-order cone.
Hereafter, we call the proposed inner- and outer-approximation hierarchies the dP- and Y{\i}ld{\i}r{\i}m-type approximation hierarchies, respectively.

As mentioned, Zuluaga et al.'s (ZVP-type) inner-approximation hierarchy~\cite{ZVP2006} and Lasserre's (Lasserre-type) outer-approximation hierarchy~\cite{Lasserre2014} are applicable to the COP matrix cone over the direct product of a nonnegative orthant and second-order cone.
Then, we theoretically and numerically compare the proposed approximation hierarchies with existing ones.
We determined that we can numerically increase a depth parameter in the dP- and Y{\i}ld{\i}r{\i}m-type approximation hierarchies, particularly when the nonnegative orthant is small.
In particular, the Y{\i}ld{\i}r{\i}m-type outer-approximation hierarchy has a higher numerical stability than the Lasserre-type one and can approach nearly optimal values of COP programming (COPP) problems numerically.

The remainder of this paper is organized as follows.
In Sect.~\ref{sec:preliminaries}, we introduce the notation and concepts used in this study.
In Sect.~\ref{sec:inner_approx_COP_SOS}, we provide an SOS-based NN-type inner-approximation hierarchy for the COP cone over a general symmetric cone and discuss its dual.
In Sect.~\ref{sec:approx_COP_spectrum}, as generalizations of the approximation hierarchies given by de Klerk and Pasechnik~\cite{dP2002} and Y{\i}ld{\i}r{\i}m~\cite{Yildirim2012}, we provide dP- and Y{\i}ld{\i}r{\i}m-type approximation hierarchies described by finitely many semidefinite constraints for the COP matrix cone over the direct product of a nonnegative orthant and second-order cone.
We also discuss their concise expressions.
In Sect.~\ref{sec:compare_hierarchy}, we introduce the existing ZVP- and Lasserre-type approximation hierarchies that are applicable to the COP matrix cone over the direct product of a nonnegative orthant and second-order cone and compare them with the proposed approximation hierarchies.
In Sect.~\ref{sec:experiment}, we compare the approximation hierarchies numerically by solving optimization problems obtained by approximating the COP cone and investigate the effect of the concise expressions mentioned in Sect.~\ref{sec:approx_COP_spectrum}.
Finally, Sect.~\ref{sec:conclusion} provides concluding remarks.

\section{Preliminaries}\label{sec:preliminaries}
\subsection{Notation}\label{subsec:notation}
We use $\mathbb{N}$, $\mathbb{R}$, $\mathbb{R}^{n\times m}$, $\mathbb{S}^n$, and $\mathbb{S}_+^n$ to denote the set of nonnegative integers, set of real numbers, set of real $n\times m$ matrices, space of $n\times n$ symmetric matrices, and set of positive semidefinite matrices in $\mathbb{S}^n$, respectively.
For a finite set $I$, we use $\mathbb{R}^{I}$ and $\mathbb{S}^{I}$ to denote the $|I|$-dimensional Euclidean space with elements indexed by $I$ and space of $|I|\times |I|$ symmetric matrices with columns and rows indexed by $I$, respectively. Similarly, let $\mathbb{S}_+^{I}$ denote the set of positive semidefinite matrices in $\mathbb{S}^I$.
We use $\bm{e}_i$ to denote the vector with an $i$th element of 1 and the remaining elements of 0, whose size is determined from the context.
In addition, we use $\bm{0}$, $\bm{1}$, $\bm{O}$, $\bm{I}$, and $\bm{E}$ to denote the zero vector, vector with all elements 1, zero matrix, identity matrix, and matrix with all elements 1, respectively.
We sometimes use a subscript, such as $\bm{1}_n$ and $\bm{I}_n$, to specify the size.
Although all vectors that appear in this paper are column vectors, for notational convenience, the difference between a column and row may not be stated if it is clear from the context.
The Euclidean space $\mathbb{R}^n$ is endowed with the usual transpose inner product and $\|\cdot\|_2$ denotes the induced norm.
We use $S^n$ and $\Delta_=^n$ to denote the $n$-dimensional unit sphere and standard simplex in $\mathbb{R}^{n+1}$, i.e.,
\begin{align*}
S^n &= \{\bm{x}\in\mathbb{R}^{n+1}\mid \|\bm{x}\|_2 = 1\},\\ \Delta_=^n &= \{\bm{x}\in\mathbb{R}^{n+1} \mid \bm{x}^\top\bm{1} = 1 \text{ and } x_i\ge 0 \text{ for all $i = 1,\dots,n+1$}\},
\end{align*}
respectively.
For a set $\mathcal{X}$, we use $|\mathcal{X}|$, $\conv(\mathcal{X})$, $\cone(\mathcal{X})$, $\cl(\mathcal{X})$, $\setint(\mathcal{X})$, and $\partial(\mathcal{X})$ to denote the cardinality, convex hull, conical hull, closure, interior, and boundary of $\mathcal{X}$, respectively.
For two finite-dimensional real vector spaces $\mathbb{V}$ and $\mathbb{W}$, we use $\Hom(\mathbb{V},\mathbb{W})$ to denote the set of linear mappings from $\mathbb{V}$ to $\mathbb{W}$.
We use $\lfloor\cdot\rfloor$ and $\lceil\cdot\rceil$ to denote the floor and ceiling functions, respectively.

We call a nonempty set $\mathcal{K}$ in a finite-dimensional real vector space a cone if $\alpha x\in\mathcal{K}$ for all $\alpha>0$ and $x\in\mathcal{K}$.
For a cone $\mathcal{K}$ in a finite-dimensional real inner product space, $\mathcal{K}^*$ denotes its dual cone, i.e., the set of $x$ such that the inner product between $x$ and $y$ is greater than or equal to 0 for all $y\in\mathcal{K}$.
A cone $\mathcal{K}$ is said to be pointed if it contains no lines.
The following properties of a cone and its dual are well known:
\begin{theorem}[{\cite[Sect.~2.6.1]{BV2004}}]\label{thm:cone}
Let $\mathcal{K}$ be a cone.
\begin{enumerate}[(i)]
\item If $\mathcal{K}$ is pointed, closed, and convex, $\mathcal{K}^*$ has a nonempty interior.
\item If $\mathcal{K}$ is convex, $(\mathcal{K}^*)^* = \cl(\mathcal{K})$ holds.
If $\mathcal{K}$ is also closed, $(\mathcal{K}^*)^* = \mathcal{K}$ holds.
\end{enumerate}
\end{theorem}

For a polynomial $f$, we use $\deg(f)$ to denote the degree of $f$.
Let $H^{n,m}$ be the set of homogeneous polynomials in $n$ variables of degree $m$ with real coefficients.
We then define $\Sigma^{n,2m} \coloneqq \conv\{\theta^2\mid\theta\in H^{n,m}\}$.
$\Sigma^{n,2m}$ is known to be a closed convex cone~\cite[Proposition~3.6]{Reznick1992}.
For $\bm{\alpha}\in\mathbb{N}^n$ and $\bm{x}\in \mathbb{R}^n$, we define $\bm{\alpha}! \coloneqq \prod_{i=1}^n\alpha_i$, $|\bm{\alpha}|\coloneqq \sum_{i=1}^n\alpha_i$, and $\bm{x}^{\bm{\alpha}}\coloneqq \prod_{i=1}^n x_i^{\alpha_i}$.
In addition, we define
\begin{align*}
\mathbb{I}^n_{=m} &\coloneqq \{\bm{\alpha} \in \mathbb{N}^n\mid |\bm{\alpha}|=m\},\\
\mathbb{I}^n_{\le m} &\coloneqq \{\bm{\alpha} \in \mathbb{N}^n\mid |\bm{\alpha}|\le m\},\\
\mathbb{N}^m_n &\coloneqq \{(i_1,\dots,i_m)\in \mathbb{N}^m \mid 1\le i_k \le n\text{ for all $k = 1,\dots,m$}\}.
\end{align*}
Under this notation, $\mathbb{R}^{\mathbb{I}^n_{=m}}$ is linearly isomorphic to $H^{n,m}$ by the mapping $(\theta_{\bm{\alpha}})_{\bm{\alpha}\in \mathbb{I}^n_{=m}}\mapsto \sum_{\bm{\alpha}\in \mathbb{I}^n_{=m}}\theta_{\bm{\alpha}}\bm{x}^{\bm{\alpha}}$.
Let $\mathfrak{S}_m$ be the symmetric group of order $m$.
Then, the group $\mathfrak{S}_m$ acts on the set $\mathbb{N}_n^m$ by $\sigma\cdot (i_1,\dots,i_m) = (i_{\sigma(1)},\dots,i_{\sigma(m)})$.
As mentioned in \cite[Sect.~4]{Dong2013}, a bijection exists between the set $\mathbb{I}^n_{=m}$ and a complete set of the representatives of $\mathfrak{S}_m$-orbits in $\mathbb{N}_n^m$.
As the set
\begin{equation}
\{(i_1,\dots,i_m)\mid 1\le i_1\le \cdots \le i_m\le n\} \label{eq:complete_set_Nnm}
\end{equation}
is a complete set of the representatives of $\mathfrak{S}_m$-orbits in $\mathbb{N}_n^m$, we define $[\bm{\alpha}]\in\mathbb{N}^m_n$ as the element of \eqref{eq:complete_set_Nnm} corresponding to $\bm{\alpha}\in \mathbb{I}^n_{=m}$, i.e.,
\begin{equation*}
[\bm{\alpha}] = (\underbrace{1,\dots,1}_{\text{$\alpha_1$ factors}},\dots,\underbrace{n,\dots,n}_{\text{$\alpha_n$ factors}}).
\end{equation*}

\subsection{Euclidean Jordan algebra and symmetric cone}
\label{subsec:EJA}
A finite-dimensional real vector space $\mathbb{E}$ equipped with a bilinear mapping (product) $\circ\colon\mathbb{E}\times\mathbb{E}\to \mathbb{E}$ is said to be a Jordan algebra if the following two conditions hold for all $x,y\in \mathbb{E}$:
\begin{enumerate}[(J1)]
\item $x \circ y = y \circ x$
\item $x\circ ((x\circ x)\circ y) = (x\circ x)\circ(x\circ y)$
\end{enumerate}
In this study, we assume that a Jordan algebra has an identity element $e$ for the product.
A Jordan algebra $(\mathbb{E},\circ)$ is said to be Euclidean if there exists an associative inner product $\bullet$ on $\mathbb{E}$ such that
\begin{enumerate}[(J3)]
\setcounter{enumi}{2}
\item $(x\circ y)\bullet z = x\bullet (y\circ z)$
\end{enumerate}
for all $x,y,z\in\mathbb{E}$.
Throughout this study, we fix an associative inner product $\bullet$ on a Euclidean Jordan algebra $(\mathbb{E},\circ)$ and regard $(\mathbb{E},\circ,\bullet)$ as a finite-dimensional real inner product space.

Let $(\mathbb{E},\circ,\bullet)$ be a Euclidean Jordan algebra.
An element $c\in\mathbb{E}$ is called an idempotent if $c\circ c = c$.
In addition, an idempotent $c$ is called primitive if it is nonzero and cannot be written as the sum of two nonzero idempotents.
Two elements $c,d\in \mathbb{E}$ are called orthogonal if $c\circ d = 0$.
The system $c_1,\dots,c_k$ is called a complete system of orthogonal idempotents if each $c_i$ is an idempotent, $c_i\circ c_j = 0$ if $i\neq j$, and $\sum_{i=1}^kc_i = e$.
In addition, if each $c_i$ is also primitive, the system is called a Jordan frame.
Each Jordan frame is known to consist of exactly $\rk$ elements, where $\rk$ is the rank of the Euclidean Jordan algebra $(\mathbb{E},\circ,\bullet)$ and the rank depends only on the algebra~\cite[Sect.~I\hspace{-1pt}I\hspace{-1pt}I.1]{FK1994}.
Here, for the convenience of the proofs, we consider an ordered Jordan frame and let $\mathfrak{F}(\mathbb{E})$ be the set of ordered Jordan frames, i.e.,
\begin{equation*}
\mathfrak{F}(\mathbb{E}) = \{(c_1,\dots,c_{\rk})\mid \text{The system $c_1,\dots,c_{\rk}$ is a Jordan frame}\}.
\end{equation*}
Note that $\mathfrak{F}(\mathbb{E})$ is a compact subset in $\mathbb{E}^{\rk}$~\cite[Exercise~I\hspace{-1pt}V.5]{FK1994}.
Each element of $\mathbb{E}$ can be decomposed into a linear combination of a Jordan frame~\cite[Theorem~I\hspace{-1pt}I\hspace{-1pt}I.1.2]{FK1994}.
In particular, for each $x\in\mathbb{E}$, there exist $x_1,\dots,x_{\rk}\in\mathbb{R}$ and $(c_1,\dots,c_{\rk})\in \mathfrak{F}(\mathbb{E})$ such that
\begin{equation}
x = \sum_{i=1}^{\rk}x_ic_i.\label{eq:spectral}
\end{equation}

The symmetric cone $\mathbb{E}_+$ associated with the Euclidean Jordan algebra $(\mathbb{E},\circ,\bullet)$ is defined as $\mathbb{E}_+ \coloneqq \{x\circ x\mid x\in\mathbb{E}\}$.
Note that when each $x\in\mathbb{E}_+$ is decomposed into the form~\eqref{eq:spectral}, all coefficients $x_i$ are nonnegative.
Conversely, for any nonnegative scalars $x_1,\dots,x_{\rk}$ and $(c_1,\dots,c_{\rk})\in\mathfrak{F}(\mathbb{E})$, it follows that $\sum_{i=1}^{\rk}x_ic_i\in\mathbb{E}_+$.

We now show some examples of the symmetric cones frequently used in this paper.

\begin{example}[nonnegative orthant]\label{ex:nno}
Let $\mathbb{E}$ be an $n$-dimensional Euclidean space $\mathbb{R}^n$.
If we set $\bm{x}\circ \bm{y} \coloneqq (x_1y_1,\dots,x_ny_n)$ and $\bm{x}\bullet \bm{y} \coloneqq \bm{x}^\top\bm{y}$ for $\bm{x},\bm{y}\in\mathbb{E}$, then $(\mathbb{E},\circ,\bullet)$ is a Euclidean Jordan algebra, and the induced symmetric cone $\mathbb{E}_+$ is the nonnegative orthant $\mathbb{R}_+^n = \{\bm{x}\in\mathbb{R}^n \mid x_i\ge 0\text{ for all $i = 1,\dots,n$}\}$.
The set $\mathfrak{F}(\mathbb{E})$ of ordered Jordan frames is $\{(\bm{e}_{\sigma(1)},\dots,\bm{e}_{\sigma(n)}) \mid \sigma\in\mathfrak{S}_n\}$.
\end{example}

\begin{example}[second-order cone]\label{ex:soc}
Let $\mathbb{E}$ be an $n$-dimensional Euclidean space $\mathbb{R}^n$ with $n\ge 2$.
If we set $\bm{x}\circ \bm{y} \coloneqq (\bm{x}^\top\bm{y},x_1\bm{y}_{2:n} + y_1\bm{x}_{2:n})$ and $\bm{x}\bullet\bm{y} \coloneqq \bm{x}^\top\bm{y}$ for $\bm{x}=(x_1,\bm{x}_{2:n})$, $\bm{y}=(y_1,\bm{y}_{2:n})\in \mathbb{E}$, then $(\mathbb{E},\circ,\bullet)$ is a Euclidean Jordan algebra, and the induced symmetric cone $\mathbb{E}_+$ is the second-order cone $\mathbb{L}^n = \{(x_1,\bm{x}_{2:n})\in \mathbb{R}^n \mid x_1\ge \|\bm{x}_{2:n}\|_2\}$.
The set $\mathfrak{F}(\mathbb{E})$ of ordered Jordan frames is
\begin{equation*}
\left\{\left(\frac{1}{2}\begin{pmatrix}
1\\
\bm{v}
\end{pmatrix},\frac{1}{2}\begin{pmatrix}
1\\
-\bm{v}
\end{pmatrix}\right) \relmiddle| \bm{v}\in S^{n-2}\right\}.
\end{equation*}
\end{example}

\begin{example}[semidefinite cone]\label{ex:sdc}
Let $\mathbb{E}$ be the space $\mathbb{S}^n$ of $n\times n$ symmetric matrices.
If we set $\bm{X}\circ \bm{Y} \coloneqq (\bm{X}\bm{Y} + \bm{Y}\bm{X})/2$ and $\bm{X}\bullet \bm{Y} \coloneqq \sum_{i,j=1}^nX_{ij}Y_{ij}$ for $\bm{X},\bm{Y}\in \mathbb{E}$, then $(\mathbb{E},\circ,\bullet)$ is a Euclidean Jordan algebra, and the induced symmetric cone $\mathbb{E}_+$ is the semidefinite cone $\mathbb{S}_+^n$.
\end{example}

Consider the case in which a Euclidean Jordan algebra $(\mathbb{E},\circ,\bullet)$ can be written as the direct product (sum) of two Euclidean Jordan algebras  $(\mathbb{E}_i,\circ_i,\bullet_i)$ with rank $\rk_i$ and identity element $e_{\mathbb{E}_i}$ $(i = 1,2)$.
Note that the following discussion can be directly extended to the case of finitely many Euclidean Jordan algebras.
The product $\circ$ and associative inner product $\bullet$ are defined as follows:
\begin{align*}
(x_1,x_2)\circ (y_1,y_2) &\coloneqq (x_1\circ_1 y_1, x_2\circ_2 y_2),\\
(x_1,x_2)\bullet (y_1,y_2) &\coloneqq x_1\bullet_1 y_1 + x_2\bullet_2 y_2
\end{align*}
for $(x_1,x_2),(y_1,y_2) \in \mathbb{E} = \mathbb{E}_1\times \mathbb{E}_2$, the rank of $\mathbb{E}$ is $\rk_1 + \rk_2$, and the identity element $e$ of $\mathbb{E}$ is $(e_{\mathbb{E}_1},e_{\mathbb{E}_2})$.
In the following, we derive the set of ordered Jordan frames of $\mathbb{E}$.
This result will be exploited in Sect.~\ref{sec:approx_COP_spectrum} to obtain approximation hierarchies described by finitely many semidefinite constraints.

\begin{lemma}\label{lem:primitive_idempotent}
For a primitive idempotent $f = (f_1,f_2)\in \mathbb{E}$, exactly one of the following two statements holds:
\begin{enumerate}[(a)]
\item $f_{1}$ is a primitive idempotent of $\mathbb{E}_1$ and $f_{2} = 0$. \label{enum:f1_primitive}
\item $f_{1} = 0$ and $f_{2}$ is a primitive idempotent of $\mathbb{E}_2$. \label{enum:f2_primitive}
\end{enumerate}
\end{lemma}
\begin{proof}
We assume that $f_1 \neq 0$.
If $f_2 \neq 0$, then $f$ can be decomposed into the sum of two nonzero elements $(f_1,0)$ and $(0,f_2)$.
The two elements are actually idempotents of $\mathbb{E}$, which contradicts $f$ being primitive.
Thus, we have $f_2 = 0$.
Because $f = (f_1,0)$ is a primitive idempotent, $f_1$ is also a primitive idempotent.
This case falls under the case~\eqref{enum:f1_primitive}.

Next, we assume that $f_1 = 0$.
In this case, as in the above discussion, we note that $f_2$ is a primitive idempotent.
This case falls under the case~\eqref{enum:f2_primitive}.
\end{proof}

\begin{proposition}[The set of ordered Jordan frames of $\mathbb{E}$]\label{prop:primitive_idempotent}
$(f_1,\dots,f_{\rk_1 + \rk_2}) \in \mathfrak{F}(\mathbb{E})$ if and only if there exists a partition $(I_1,I_2)$ of $\{1,\dots,\rk_1+\rk_2\}$ such that the following two conditions hold:
\begin{enumerate}[(i)]
\item $f_i = (f_{1i},0)\in\mathbb{E}$ for all $i\in I_1$ and $(f_{1i})_{i\in I_1} \in \mathfrak{F}(\mathbb{E}_1)$.
\item $f_i = (0,f_{2i})\in \mathbb{E}$ for all $i\in I_2$ and $(f_{2i})_{i\in I_2} \in \mathfrak{F}(\mathbb{E}_2)$.
\end{enumerate}
\end{proposition}

\begin{proof}
We first prove the ``if'' part.
Because $(f_1,\dots,f_{\rk_1+\rk_2})$ satisfying the two conditions is evidently a complete system of orthogonal idempotents, we prove only that each $f_i$ is primitive.
To prove this, showing that $(f,0)\in\mathbb{E}$ is primitive if $f\in \mathbb{E}_1$ is primitive is sufficient.
Assume that $(f,0)$ can be written as the sum of two idempotents $(f_1,f_2)$ and $(g_1,g_2)$, i.e.,
\begin{equation}
(f,0) = (f_1,f_2) + (g_1,g_2). \label{eq:proof_primitive}
\end{equation}
First, \eqref{eq:proof_primitive} implies that $f_2 = -g_2$.
In addition, we note that $f_1$, $g_1$, $f_2$, and $g_2$ are idempotents because $(f_1,f_2)$ and $(g_1,g_2)$ are idempotents.
Therefore, $f_2 = g_2 = 0$.
Second, \eqref{eq:proof_primitive} implies again that $f = f_1 + g_1$.
As $f_1$ and $g_1$ are idempotents and $f$ is primitive, either $f_1$ or $g_1$ must be 0. We can assume that $f_1 = 0$ without loss of generality.
We then obtain $(f_1,f_2) = 0$, which implies that $(f,0)$ is primitive.

We now prove the ``only if'' part.
Let $(f_1,\dots,f_{\rk_1 + \rk_2}) \in \mathfrak{F}(\mathbb{E})$ and set $f_i = (f_{1i},f_{2i})$ for each $i$.
Then, each $f_i = (f_{1i},f_{2i})$ falls into exactly one of the two cases in Lemma~\ref{lem:primitive_idempotent}; thus, we define
\begin{align*}
I_1 &\coloneqq \{i \in \{1,\dots,{\rk}_1 + {\rk}_2\}\mid \text{$f_{1i}$ is a primitive idempotent and $f_{2i} = 0$}\},\\
I_2 &\coloneqq \{i \in \{1,\dots,{\rk}_1 + {\rk}_2\} \mid \text{$f_{1i} = 0$ and $f_{2i}$ is a primitive idempotent}\}.
\end{align*}
Evidently, $(I_1,I_2)$ is a partition of $\{1,\dots,\rk_1 + \rk_2\}$.
In the following, we show that $(f_{1i})_{i\in I_1}\in \mathfrak{F}(\mathbb{E}_1)$ and $(f_{2i})_{i\in I_2}\in \mathfrak{F}(\mathbb{E}_2)$.
From the assumption on $(f_1,\dots,f_{\rk_1 + \rk_2})$, it follows that
\begin{equation*}
e = (e_{\mathbb{E}_1},e_{\mathbb{E}_2}) = \sum_{i=1}^{\rk_1 + \rk_2}f_i = \left(\sum_{i\in I_1}f_{1i},\sum_{i\in I_2}f_{2i}\right),
\end{equation*}
from which we obtain $\sum_{i\in I_1}f_{1i} = e_{\mathbb{E}_1}$ and $\sum_{i\in I_2}f_{2i} = e_{\mathbb{E}_2}$.
In addition, it follows that $0 = f_i \circ f_j = (f_{1i}\circ_1 f_{1j},f_{2i}\circ_2 f_{2j})$
for any $i\neq j$.
In particular, we have $f_{1i}\circ_1 f_{1j} = 0$ for all $i\neq j\in I_1$ and $f_{2i}\circ_2 f_{2j} = 0$ for all $i\neq j\in I_2$.
\end{proof}

\subsection{Symmetric tensor space}\label{subsec:tensor}
Let $(\mathbb{V},(\cdot,\cdot))$ be an $n$-dimensional real inner product space.
Note that $\mathbb{V}$ can be identified with the dual space $\Hom(\mathbb{V},\mathbb{R})$ by the natural isomorphism $x\mapsto (x,\cdot)$.
We use
\begin{equation*}
\mathbb{V}^{\otimes m} \coloneqq \underbrace{\mathbb{V}\otimes\cdots\otimes\mathbb{V}}_{\text{$m$ factors}}
\end{equation*}
to denote the tensor space of order $m$ over $\mathbb{V}$.

Let $v_1,\dots,v_n$ be a basis for $\mathbb{V}$.
Then, $\tilde{v}_{i_1\cdots i_m} \coloneqq v_{i_1}\otimes\cdots \otimes v_{i_m}\ ((i_1,\dots,i_m)\in \mathbb{N}_n^m)$ form a basis for $\mathbb{V}^{\otimes m}$.
That is, each $\mathcal{A}\in\mathbb{V}^{\otimes m}$ can be written in the following form:
\begin{equation}
\mathcal{A} = \sum_{i_1,\dots,i_m=1}^n\mathcal{A}_{i_1\cdots i_m}\tilde{v}_{i_1\cdots i_m} \label{eq:tensor_representation}
\end{equation}
with coefficients $\mathcal{A}_{i_1\cdots i_m}\in\mathbb{R}$.
For $\sigma\in\mathfrak{S}_m$, the linear transformation $\pi_{\sigma}$ on $\mathbb{V}^{\otimes m}$ is defined by $\pi_{\sigma}(\tilde{v}_{i_1\cdots i_m}) \coloneqq \tilde{v}_{i_{\sigma(1)}\cdots i_{\sigma(m)}}$.
The definition of $\pi_{\sigma}$ does not depend on the choice of the basis for $\mathbb{V}$.
Then,
\begin{equation*}
\mathcal{S}^{n,m}(\mathbb{V}) \coloneqq \{\mathcal{A}\in\mathbb{V}^{\otimes m} \mid \pi_{\sigma}\mathcal{A} = \mathcal{A}\ \text{for all $\sigma \in\mathfrak{S}_m$}\}
\end{equation*}
denotes the symmetric tensor space of order $m$ over $\mathbb{V}$, which is a subspace of $\mathbb{V}^{\otimes m}$.
Note that the symmetric tensor $\mathcal{A}\in\mathcal{S}^{n,m}(\mathbb{V})$ with the form~\eqref{eq:tensor_representation} depends only on the coefficients in the form of $\mathcal{A}_{[\bm{\alpha}]}$ $(\bm{\alpha}\in \mathbb{I}_{=m}^n)$.
Let $\mathscr{S}$ be the linear transformation on $\mathbb{V}^{\otimes m}$ defined as
\begin{equation*}
\mathscr{S} \coloneqq \frac{1}{m!}\sum_{\sigma\in\mathfrak{S}_m}\pi_{\sigma}.
\end{equation*}
Then, $\mathscr{S}\mathcal{A} \in \mathcal{S}^{n,m}(\mathbb{V})$ for each $\mathcal{A}\in\mathbb{V}^{\otimes m}$ and $\mathscr{S}\tilde{v}_{[\bm{\alpha}]}$ $(\bm{\alpha}\in \mathbb{I}_{=m}^n)$ form a basis for $\mathcal{S}^{n,m}(\mathbb{V})$.
For $x\in\mathbb{V}$, let
\begin{equation*}
x^{\otimes m} \coloneqq \underbrace{x\otimes \cdots \otimes x}_{\text{$m$ factors}} \in \mathcal{S}^{n,m}(\mathbb{V}).
\end{equation*}

In particular, we consider the case of $\mathbb{V} = \mathbb{R}^n$ with the canonical basis $\bm{e}_1,\dots,\bm{e}_n$ and write $\mathcal{S}^{n,m}(\mathbb{R}^n)$ as $\mathcal{S}^{n,m}$.
Then, each element $\mathcal{A}\in\mathcal{S}^{n,m}$ can be considered a multi-dimensional array; thus, we write $\mathcal{A}_{i_1\cdots i_m}$ for the $(i_1,\dots,i_m)$th element of $\mathcal{A}$.
Note that the symmetric tensor space $\mathcal{S}^{n,2}$ of order two equals the space $\mathbb{S}^n$ of the symmetric matrices.

Let $\langle\cdot,\cdot\rangle$ be the inner product on $\mathbb{V}^{\otimes m}$ induced by that on $\mathbb{V}$.
That is, it satisfies $\langle x_1\otimes \cdots\otimes x_m,y_1\otimes \cdots\otimes y_m\rangle = \prod_{i=1}^m(x_i,y_i)$ for $x_1\otimes \cdots\otimes x_m$, $y_1\otimes \cdots\otimes y_m\in\mathbb{V}^{\otimes m}$.
We write $\|\cdot\|_{\rm F}$ for the norm on $\mathbb{V}^{\otimes m}$ induced by the inner product $\langle \cdot,\cdot\rangle$.

Using the inner product, we note that $\mathcal{S}^{n,m}(\mathbb{V})$ and $H^{n,m}$ are linearly isomorphic.
Indeed, let $\phi \in \Hom(\mathbb{V},\mathbb{R}^n)$ be the linear isomorphism induced by the basis $v_1,\dots,v_n$.
Then, the mapping $\psi\colon\mathcal{S}^{n,m}(\mathbb{V})\to H^{n,m}$ defined by
\begin{equation}
\psi(\mathcal{A}) \coloneqq \langle \mathcal{A},\phi^{-1}(\bm{x})^{\otimes m}\rangle \label{eq:isom_Snm(V)_Hnm}
\end{equation}
is a linear isomorphism.

In the following proofs, for convenience, we fix an orthonormal basis $v_1,\dots,v_n$ for $\mathbb{V}$ arbitrarily.
The following lemma describes a property of the inner product on the (symmetric) tensor space.

\begin{lemma}\label{lem:tensor_inner_product}
For $\mathcal{A}\in \mathbb{V}^{\otimes m}$ and $\mathcal{B}\in\mathcal{S}^{n,m}(\mathbb{V})$, it follows that $\langle \mathscr{S}\mathcal{A},\mathcal{B}\rangle = \langle \mathcal{A},\mathcal{B}\rangle$.
\end{lemma}
\begin{proof}
Using the orthonormal basis $v_1,\dots,v_n$ for $\mathbb{V}$, we write $\mathcal{A}$ and $\mathcal{B}$ in the form \eqref{eq:tensor_representation}.
It then follows from the symmetry of $\mathcal{B}$ that
\begin{align*}
\langle \mathscr{S}\mathcal{A},\mathcal{B}\rangle &= \left\langle \frac{1}{m!}\sum_{\sigma\in\mathfrak{S}_m}\pi_{\sigma}\mathcal{A},\mathcal{B}\right\rangle\\
&= \frac{1}{m!}\sum_{\sigma\in\mathfrak{S}_m}\sum_{i_1,\dots,i_m=1}^n \mathcal{A}_{i_1\cdots i_m}\mathcal{B}_{i_{\sigma(1)}\cdots i_{\sigma(m)}}\\
&= \frac{1}{m!}\sum_{\sigma\in\mathfrak{S}_m}\sum_{i_1,\dots,i_m=1}^n \mathcal{A}_{i_1\cdots i_m}\mathcal{B}_{i_1\cdots i_m}\\
&= \frac{1}{m!}\sum_{\sigma\in\mathfrak{S}_m} \langle \mathcal{A},\mathcal{B}\rangle\\
&= \langle \mathcal{A},\mathcal{B}\rangle.
\end{align*}
\end{proof}

Consider the case of $\mathbb{V} = \mathbb{R}^n$ with the canonical basis $\bm{e}_1,\dots,\bm{e}_n$.
The following lemma provides an orthonormal basis for the symmetric tensor space $\mathcal{S}^{n,m}$ and the representation of the elements of $\mathcal{S}^{n,m}$ with the orthonormal basis.
\begin{lemma}\label{lem:orth_basis}
We now define $\mathcal{F}_{\bm{\alpha}} \coloneqq \sqrt{m!/\bm{\alpha}!}\mathscr{S}\tilde{\bm{e}}_{[\bm{\alpha}]}$
for each $\bm{\alpha}\in \mathbb{I}_{=m}^n$.
Then, $(\mathcal{F}_{\bm{\alpha}})_{\bm{\alpha}\in \mathbb{I}_{=m}^n}$ is an orthonormal basis for $\mathcal{S}^{n,m}$.
In addition, using the orthonormal basis, we can represent each $\mathcal{A}\in\mathcal{S}^{n,m}$ as $\mathcal{A} = \sum_{\bm{\alpha}\in \mathbb{I}_{=m}^n}\sqrt{m !/\bm{\alpha}!}\mathcal{A}_{[\bm{\alpha}]}\mathcal{F}_{\bm{\alpha}}$.
\end{lemma}
\begin{proof}
Note that $\dim\mathcal{S}^{n,m} = |\mathbb{I}_{=m}^n|$.
Because the linear independence of $(\mathcal{F}_{\bm{\alpha}})_{\bm{\alpha}\in \mathbb{I}_{=m}^n}$ is clear, showing that it is orthonormal is sufficient.
Let $\bm{\alpha},\bm{\beta}\in \mathbb{I}_{=m}^n$.
We first consider the case of $\bm{\alpha} \neq \bm{\beta}$.
Let $(i_1,\dots,i_m) \coloneqq [\bm{\alpha}]$ and $(j_1,\dots,j_m)\coloneqq [\bm{\beta}]$.
Then, as $(i_1,\dots,i_m) \neq (j_1,\dots,j_m)$, there exists $k_0 \in \{1,\dots,m\}$ such that $i_{k_0} \neq j_{k_0}$, which implies that the number of values $i_{k_0}$ in the vector $(i_1,\dots,i_m)$ is not equal to that in the vector $(j_1,\dots,j_m)$.
Using Lemma~\ref{lem:tensor_inner_product}, we then have
\begin{align*}
\langle \mathcal{F}_{\bm{\alpha}},\mathcal{F}_{\bm{\beta}}\rangle &= \left\langle\sqrt{\frac{m!}{\bm{\alpha}!}}\mathscr{S}\tilde{\bm{e}}_{i_1\cdots i_m},\sqrt{\frac{m!}{\bm{\beta}!}}\mathscr{S}\tilde{\bm{e}}_{j_1\cdots j_m}\right\rangle\\
&= \sqrt{\frac{m!}{\bm{\alpha}!}}\sqrt{\frac{m!}{\bm{\beta}!}}\langle \mathscr{S}\tilde{\bm{e}}_{i_1\cdots i_m},\tilde{\bm{e}}_{j_1\cdots j_m}\rangle\\
&= \sqrt{\frac{m!}{\bm{\alpha}!}}\sqrt{\frac{m!}{\bm{\beta}!}}\left\langle \frac{1}{m!}\sum_{\sigma\in\mathfrak{S}_m}\tilde{\bm{e}}_{i_{\sigma(1)}\cdots i_{\sigma(m)}},\tilde{\bm{e}}_{j_1\cdots j_m}\right\rangle\\
&= \sqrt{\frac{m!}{\bm{\alpha}!}}\sqrt{\frac{m!}{\bm{\beta}!}}\frac{1}{m!}\sum_{\sigma\in\mathfrak{S}_m}\prod_{k=1}^m\bm{e}_{i_{\sigma(k)}}^\top\bm{e}_{j_k}\\
&= 0.
\end{align*}
Second, we consider the case of $\bm{\alpha} = \bm{\beta}$.
Let $(i_1,\dots,i_m) \coloneqq [\bm{\alpha}]$.
Then,
\begin{align*}
\|\mathcal{F}_{\bm{\alpha}}\|_{\rm F} &= \frac{m!}{\bm{\alpha}!}\langle\mathscr{S}\tilde{\bm{e}}_{i_1\cdots i_m},\tilde{\bm{e}}_{i_1\cdots i_m}\rangle\\
&= \frac{m!}{\bm{\alpha}!}\frac{1}{m!}\sum_{\sigma\in\mathfrak{S}_m}\prod_{k=1}^m\bm{e}_{i_k}^\top \bm{e}_{i_{\sigma(k)}}\\
&= 1.
\end{align*}
Therefore, $(\mathcal{F}_{\bm{\alpha}})_{\bm{\alpha}\in \mathbb{I}_{=m}^n}$ is an orthonormal basis for $\mathcal{S}^{n,m}$.

In addition, for $\bm{\alpha}\in\mathbb{I}^n_{=m}$, let
\begin{equation*}
\mathbb{N}_n^m[\bm{\alpha}] \coloneqq \{(i_1,\dots i_m)\in \mathbb{N}_n^m\mid \text{$(i_{\sigma(1)},\dots, i_{\sigma(m)}) = [\bm{\alpha}]$ for some $\sigma\in\mathfrak{S}_m$}\}.
\end{equation*}
From the symmetry of $\mathcal{A}$, it then follows that
\begin{align*}
\mathcal{A} &= \sum_{i_1,\dots,i_m=1}^n\mathcal{A}_{i_1\cdots i_m}\tilde{\bm{e}}_{i_1\cdots i_m}\\
&= \sum_{\bm{\alpha}\in \mathbb{I}_{=m}^n}\sum_{(i_1,\dots,i_m)\in \mathbb{N}_n^m[\bm{\alpha}]}\mathcal{A}_{i_1\cdots i_m}\tilde{\bm{e}}_{i_1\cdots i_m}\\
&= \sum_{\bm{\alpha}\in \mathbb{I}_{=m}^n}\mathcal{A}_{[\bm{\alpha}]}\left(\sum_{(i_1,\dots,i_m)\in \mathbb{N}_n^m[\bm{\alpha}]}\tilde{\bm{e}}_{i_1\cdots i_m}\right)\\
&= \sum_{\bm{\alpha}\in \mathbb{I}_{=m}^n}\mathcal{A}_{[\bm{\alpha}]}\left(\frac{1}{\bm{\alpha}!}\sum_{\sigma\in\mathfrak{S}_m}\pi_{\sigma}\tilde{\bm{e}}_{[\bm{\alpha}]}\right)\\
&= \sum_{\bm{\alpha}\in \mathbb{I}_{=m}^n}\sqrt{\frac{m !}{\bm{\alpha}!}}\mathcal{A}_{[\bm{\alpha}]}\left(\sqrt{\frac{m !}{\bm{\alpha}!}}\mathscr{S}\tilde{\bm{e}}_{[\bm{\alpha}]}\right)\\
&= \sum_{\bm{\alpha}\in \mathbb{I}_{=m}^n}\sqrt{\frac{m !}{\bm{\alpha}!}}\mathcal{A}_{[\bm{\alpha}]}\mathcal{F}_{\bm{\alpha}}.
\end{align*}
\end{proof}

For convenience, we write the coefficients of $\mathcal{A}\in\mathcal{S}^{n,m}$ with respect to the orthonormal basis $(\mathcal{F}_{\bm{\alpha}})_{\bm{\alpha}\in \mathbb{I}_{=m}^n}$ taken in Lemma~\ref{lem:orth_basis} as $(\mathcal{A}_{\bm{\alpha}}^{\mathcal{F}})_{\bm{\alpha}\in \mathbb{I}_{=m}^n}$.
That is, each $\mathcal{A}\in\mathcal{S}^{n,m}$ is written as
\begin{equation}
\mathcal{A} = \sum_{\bm{\alpha}\in \mathbb{I}_{=m}^n}\mathcal{A}_{\bm{\alpha}}^{\mathcal{F}}\mathcal{F}_{\bm{\alpha}}. \label{eq:A_orth_basis}
\end{equation}
Because $\mathcal{S}^{n,m}$ and $H^{n,m}$ are linearly isomorphic by the mapping~\eqref{eq:isom_Snm(V)_Hnm}, for each $\mathcal{A}\in\mathcal{S}^{n,m}$, there exists $\theta\in H^{n,m}$ such that $\langle \mathcal{A},\bm{x}^{\otimes m}\rangle = \theta(\bm{x})$.
The following lemma links the coefficients $(\mathcal{A}_{\bm{\alpha}}^{\mathcal{F}})_{\bm{\alpha}\in \mathbb{I}_{=m}^n}$ with the coefficients of $\theta$.

\begin{lemma}\label{lem:A_orth_coef}
Suppose that $\mathcal{A}\in\mathcal{S}^{n,m}$ and $(\theta_{\bm{\alpha}})_{\bm{\alpha}\in \mathbb{I}_{=m}^n} \in \mathbb{R}^{\mathbb{I}_{=m}^n}$ satisfy
\begin{equation}
\langle \mathcal{A},\bm{x}^{\otimes m}\rangle = \sum_{\bm{\alpha}\in \mathbb{I}_{=m}^n}\theta_{\bm{\alpha}} \bm{x}^{\bm{\alpha}}. \label{eq:Ax_thetax}
\end{equation}
Then, $\mathcal{A}_{\bm{\alpha}}^{\mathcal{F}} = \sqrt{\bm{\alpha}!/m!}\theta_{\bm{\alpha}}$ for all $\bm{\alpha}\in \mathbb{I}_{=m}^n$.
\end{lemma}
\begin{proof}
Let $\mathcal{A}$ be in the form~\eqref{eq:A_orth_basis}.
It then follows from Lemma~\ref{lem:tensor_inner_product} that
\begin{align*}
\langle \mathcal{A},\bm{x}^{\otimes m}\rangle &= \sum_{\bm{\alpha}\in \mathbb{I}_{=m}^n}\left\langle \mathcal{A}_{\bm{\alpha}}^{\mathcal{F}}\sqrt{\frac{m!}{\bm{\alpha}!}}\mathscr{S}\tilde{\bm{e}}_{[\bm{\alpha}]},\bm{x}^{\otimes m}\right\rangle \\
&=  \sum_{\bm{\alpha}\in \mathbb{I}_{=m}^n}\mathcal{A}_{\bm{\alpha}}^{\mathcal{F}}\sqrt{\frac{m!}{\bm{\alpha}!}}\langle \underbrace{\bm{e}_1\otimes \cdots\otimes\bm{e}_1}_{\text{$\alpha_1$ factors}}\otimes \cdots\otimes \underbrace{\bm{e}_n\otimes \cdots\otimes \bm{e}_n}_{\text{$\alpha_n$ factors}},\bm{x}^{\otimes m}\rangle \\
&= \sum_{\bm{\alpha}\in \mathbb{I}_{=m}^n}\mathcal{A}_{\bm{\alpha}}^{\mathcal{F}}\sqrt{\frac{m!}{\bm{\alpha}!}}\bm{x}^{\bm{\alpha}}.
\end{align*}
As this equals $\sum_{\bm{\alpha}\in \mathbb{I}_{=m}^n}\theta_{\bm{\alpha}} \bm{x}^{\bm{\alpha}}$ and by comparing the coefficients, we obtain the desired result.
\end{proof}
Lemmas~\ref{lem:orth_basis} and \ref{lem:A_orth_coef} will be used to prove the technical result (Lemma~\ref{eq:mom_matrix}) in Sect.~\ref{subsec:dual}.

\subsection{Copositive and completely positive cones}\label{subsec:CPCOP}
Let $\mathbb{K}$ be a closed cone in an $n$-dimensional real inner product space $(\mathbb{V},(\cdot,\cdot))$.
Then, we define
\begin{align*}
\COP^{n,m}(\mathbb{K}) &\coloneqq \{\mathcal{A}\in\mathcal{S}^{n,m}(\mathbb{V}) \mid \langle\mathcal{A},x^{\otimes m}\rangle \ge 0 \text{ for all $x\in\mathbb{K}$}\},\\
\CP^{n,m}(\mathbb{K}) &\coloneqq \conv\{x^{\otimes m} \mid x\in \mathbb{K}\}
\end{align*}
and call them the COP and CP cones (over $\mathbb{K}$), respectively.
In the case in which $\mathbb{V} = \mathbb{R}^n$ and $\mathbb{K} = \mathbb{R}_+^n$, the COP and CP cones reduce to the COP tensor cone, written as $\COP^{n,m}$ and the CP tensor cone given by, for example,~\cite{Qi2013,QXX2014}.
In the case of $m = 2$ and $\mathbb{V} = \mathbb{R}^n$, under the identification between $\mathcal{S}^{n,2}$ and $\mathbb{S}^n$, we have
\begin{align*}
\COP(\mathbb{K}) &\coloneqq \COP^{n,2}(\mathbb{K}) = \{\bm{A}\in\mathbb{S}^n \mid \bm{x}^\top\bm{A}\bm{x}\ge 0 \text{ for all $\bm{x}\in\mathbb{K}$}\},\\
\CP(\mathbb{K}) &\coloneqq \CP^{n,2}(\mathbb{K}) =  \conv\{\bm{x}\bm{x}^\top \mid \bm{x}\in \mathbb{K}\},
\end{align*}
which are the COP and CP matrix cones~\cite{GS2013}, respectively.
Because we considered both the tensor and matrix cases, we generally omit the terms ``tensor'' and ``matrix.''

We now discuss the duality between $\COP^{n,m}(\mathbb{K})$ and $\CP^{n,m}(\mathbb{K})$.

\begin{proposition}\label{prop:CPCOP}
Let $\mathbb{K}$ be a closed cone in $\mathbb{V}$.
\begin{enumerate}[(i)]
\item $\COP^{n,m}(\mathbb{K}) = \CP^{n,m}(\mathbb{K})^*$. \label{enum:dual}
\item If $\mathbb{K}$ is also pointed and convex, then $\CP^{n,m}(\mathbb{K})$ is a closed convex cone.
\label{enum:closed}
\end{enumerate}
\end{proposition}

\begin{proof}
We first prove \eqref{enum:dual}.
Let $\mathcal{A} \in \COP^{n,m}(\mathbb{K})$.
Then, for any $x_i\in \mathbb{K}$ and $\lambda_i\ge 0$ such that $\sum_{i}\lambda_i = 1$, we have
\begin{equation}
\left\langle \mathcal{A},\sum_{i}\lambda_ix_i^{\otimes m}\right\rangle = \sum_i\lambda_i\langle \mathcal{A},x_i^{\otimes m}\rangle. \label{eq:dual}
\end{equation}
As $\langle \mathcal{A},x_i^{\otimes m}\rangle$ and $\lambda_i$ are nonnegative for all $i$, \eqref{eq:dual} is also nonnegative, which implies that $\mathcal{A} \in \CP^{n,m}(\mathbb{K})^*$.
Conversely, let $\mathcal{A} \in \CP^{n,m}(\mathbb{K})^*$.
For any $x\in\mathbb{K}$, because $x^{\otimes m} \in \CP^{n,m}(\mathbb{K})$, $\langle \mathcal{A},x^{\otimes m}\rangle \ge 0$ follows from the definition of dual cones.
Therefore, we obtain $\mathcal{A}\in\COP^{n,m}(\mathbb{K})$.

We now prove \eqref{enum:closed}.
The convexity of $\CP^{n,m}(\mathbb{K})$ follows from its definition.
In addition, $\CP^{n,m}(\mathbb{K})$ is a cone because $\mathbb{K}$ is a cone.
To prove the closedness, let $\{\mathcal{A}_k\}_k \subseteq \CP^{n,m}(\mathbb{K})$ and suppose it converges to some $\mathcal{A}_{\infty}\in\mathcal{S}^{n,m}(\mathbb{V})$.
Note that $\CP^{n,m}(\mathbb{K})$ can be represented as $\CP^{n,m}(\mathbb{K}) = \cone\{x^{\otimes m}\mid x\in\mathbb{K}\}$
as it is a convex cone containing zero (the origin).
Therefore, by Carath\'{e}odory's theorem for cones~\cite[Exercise~B.1.7]{Bertsekas1999}, every element of $\CP^{n,m}(\mathbb{K})$ can be written as the sum of at most $d \coloneqq \dim\mathcal{S}^{n,m}(\mathbb{V})$ elements in the form of $x^{\otimes m}$ with $x\in\mathbb{K}$; for every $k$, there exist $x_{ki}$ with $x_{ki}\in\mathbb{K}$ $(i = 1,\dots,d)$ such that $\mathcal{A}_k = \sum_{i=1}^dx_{ki}^{\otimes m}$.
As $\setint(\mathbb{K}^*)$ is nonempty under the assumption on $\mathbb{K}$, we take $a\in\setint(\mathbb{K}^*)$ arbitrarily.
Then, we obtain
\begin{equation*}
\langle \mathcal{A}_k,a^{\otimes m}\rangle = \sum_{i=1}^d\langle x_{ki}^{\otimes m},a^{\otimes m}\rangle = \sum_{i=1}^d(x_{ki},a)^m \to \langle \mathcal{A}_{\infty},a^{\otimes m}\rangle\ (k\to \infty).
\end{equation*}
Given that $x_{ki}\in\mathbb{K}$, $(x_{ki},a) \ge 0$ for any $k$ and $i$.
Therefore, $\{(x_{ki},a)\}_k$ is bounded for each $i$.

Then, $\{x_{ki}\}_k$ is bounded for each $i$.
To observe this, let $\{k(l) \mid l\in\mathbb{N}\}$ denote the set of indices $k$ such that $x_{ki} \neq 0$.
Showing the boundedness of $\{x_{k(l)i}\}_l$ is sufficient.
Let $B \coloneqq \mathbb{K} \cap \{y\in \mathbb{V}\mid (y,a) = 1\}$.
Note that $B$ is compact.
Then, for each $l$, there exist $\alpha_{li}> 0$ and $y_{li}\in B$ such that $x_{k(l)i} = \alpha_{li}y_{li}$.
Given that $(x_{k(l)i},a) = \alpha_{li}$, the sequence $\{\alpha_{li}\}_l$ is bounded.
Combining it with the boundedness of $\{y_{li}\}_l \subseteq B$ leads to the boundedness of
$\{x_{k(l)i}\}_l = \{\alpha_{li}y_{li}\}_l$.

Thus, by taking a subsequence, if necessary, we assume that $\{x_{ki}\}_k$ converges to some $x_{\infty i}$ for each $i$.
The closedness of $\mathbb{K}$ implies that $x_{\infty i}\in\mathbb{K}$.
Therefore,
\begin{equation*}
\mathcal{A}_{\infty} = \lim_{k\to\infty}\mathcal{A}_{k} = \lim_{k\to\infty}\sum_{i=1}^dx_{ki}^{\otimes m} = \sum_{i=1}^dx_{\infty i}^{\otimes m} \in \mathcal{CP}^{n,m}(\mathbb{K}),
\end{equation*}
which means that $\CP^{n,m}(\mathbb{K})$ is closed.
\end{proof}

\begin{corollary}\label{cor:CPCOP_dual}
Let $\mathbb{K}$ be a pointed closed convex cone.
Then, $\COP^{n,m}(\mathbb{K})$ and $\CP^{n,m}(\mathbb{K})$ are dual to each other.
\end{corollary}

\begin{proof}
It follows from \eqref{enum:closed} in Proposition~\ref{prop:CPCOP} that $\CP^{n,m}(\mathbb{K})$ is a closed convex cone.
Taking the dual of \eqref{enum:dual} in Proposition~\ref{prop:CPCOP} and using Theorem~\ref{thm:cone}, we obtain $\COP^{n,m}(\mathbb{K})^* = \CP^{n,m}(\mathbb{K})$.
\end{proof}

In this study, we focused only on the case in which $\mathbb{K}$ is a symmetric cone, which is a pointed closed convex cone.
In this case, Corollary~\ref{cor:CPCOP_dual} is applicable to $\COP^{n,m}(\mathbb{K})$ and $\CP^{n,m}(\mathbb{K})$.

\subsection{Homogeneous polynomial function on inner product space}
Let $(\mathbb{V},(\cdot,\cdot))$ be an $n$-dimensional real inner product space.
A homogeneous polynomial function of degree $m$ on $\mathbb{V}$ is the mapping $\mathbb{V}\ni x \mapsto \langle\mathcal{A},x^{\otimes m}\rangle $ for some $\mathcal{A}\in\mathcal{S}^{n,m}(\mathbb{V})$.
$H^{n,m}(\mathbb{V})$ denotes the set of homogeneous polynomial functions of degree $m$ on $\mathbb{V}$, i.e., $H^{n,m}(\mathbb{V}) = \{\langle \mathcal{A},x^{\otimes m}\rangle \mid \mathcal{A}\in\mathcal{S}^{n,m}(\mathbb{V})\}$.
As $H^{n,m} = \{\langle \mathcal{A},\bm{x}^{\otimes m}\rangle \mid \mathcal{A}\in\mathcal{S}^{n,m}\}$, $H^{n,m}(\mathbb{R}^n)$ agrees with $H^{n,m}$.

As the definition of $\Sigma^{n,2m}$, $\Sigma^{n,2m}(\mathbb{V})$ denotes the set of sums of squares of homogeneous polynomial functions of degree $m$ on $\mathbb{V}$.
To represent the set $\Sigma^{n,2m}(\mathbb{V})$ more explicitly, we prove the following lemma.

\begin{lemma}\label{lem:Ax^2}
$\langle \mathcal{A},x^{\otimes m}\rangle ^2 = \langle \mathscr{S}(\mathcal{A}\otimes \mathcal{A}),x^{\otimes 2m}\rangle$ for any $\mathcal{A}\in\mathcal{S}^{n,m}(\mathbb{V})$.
\end{lemma}
\begin{proof}
Fix an orthonormal basis $v_1,\dots,v_n$ in $\mathbb{V}$ arbitrarily.
In addition, using the basis, we write $\mathcal{A}\in\mathcal{S}^{n,m}(\mathbb{V})$ in the form~\eqref{eq:tensor_representation}.
Given that
\begin{equation*}
\langle \mathcal{A},x^{\otimes m}\rangle = \sum_{i_1,\dots,i_m=1}^n\mathcal{A}_{i_1\cdots i_m}\prod_{k=1}^m(x,v_{i_k}),
\end{equation*}
we have
\begin{equation}
\langle \mathcal{A},x^{\otimes m}\rangle^2 = \sum_{i_1,\dots,i_{2m}=1}^n\mathcal{A}_{i_1\cdots i_m}\mathcal{A}_{i_{m+1}\cdots i_{2m}}\prod_{k=1}^{2m}(x,v_{i_k}). \label{eq:Ax^2}
\end{equation}
Moreover, as $\mathcal{A}\otimes\mathcal{A} = \sum_{i_1,\dots,i_{2m}=1}^n\mathcal{A}_{i_1\cdots i_m}\mathcal{A}_{i_{m+1}\cdots i_{2m}}\tilde{v}_{i_1\cdots i_{2m}}$,
$\langle \mathcal{A}\otimes \mathcal{A},x^{\otimes 2m}\rangle$ agrees with \eqref{eq:Ax^2}.
Therefore, by applying Lemma~\ref{lem:tensor_inner_product}, we obtain the desired result.
\end{proof}
Using Lemma~\ref{lem:Ax^2}, we can express $\Sigma^{n,2m}(\mathbb{V})$ as $\conv\{\langle \mathscr{S}(\mathcal{A}\otimes \mathcal{A}),x^{\otimes 2m}\rangle \mid \mathcal{A}\in\mathcal{S}^{n,m}(\mathbb{V})\}$.
Through the isomorphism $H^{n,2m}(\mathbb{V}) \ni \langle \mathcal{A},x^{\otimes 2m}\rangle \mapsto \mathcal{A}\in\mathcal{S}^{n,2m}(\mathbb{V})$, the set $\Sigma^{n,2m}(\mathbb{V})$ is mapped onto the set $\conv\{\mathscr{S}(\mathcal{A}\otimes \mathcal{A})\mid \mathcal{A}\in\mathcal{S}^{n,2m}(\mathbb{V})\}$ denoted by $\SOS^{n,2m}(\mathbb{V})$.
We define $\MOM^{n,2m}(\mathbb{V}) \coloneqq \SOS^{n,2m}(\mathbb{V})^*$ and call it the moment cone.
Because $\SOS^{n,2m}(\mathbb{V})$ is a closed convex cone~\cite[Lemma~2.2]{CLQ2016}, $\SOS^{n,2m}(\mathbb{V})$ and $\MOM^{n,2m}(\mathbb{V})$ are dual to each other.

\section{Sum-of-squares-based inner-approximation hierarchy}\label{sec:inner_approx_COP_SOS}
Let $(\mathbb{E},\circ,\bullet)$ be a Euclidean Jordan algebra of dimension $n$.
In this section, we aim to provide an inner-approximation hierarchy described by an SOS constraint for the COP cone $\COP^{n,m}(\mathbb{E}_+)$.
In Sect.~\ref{subsec:previous_work}, we first provide an inner-approximation hierarchy for the cone of homogeneous polynomials that are nonnegative over a symmetric cone in $\mathbb{R}^n$.
Using the results in Sect.~\ref{subsec:previous_work}, we provide the desired approximation hierarchy in Sect.~\ref{subsec:case_tensor}.
In Sect.~\ref{subsec:dual}, we discuss its dual.
Subsequently, we fix an orthonormal basis $v_1,\dots,v_n$ for the given Euclidean Jordan algebra $(\mathbb{E},\circ,\bullet)$ and let $\phi\colon\mathbb{E}\to\mathbb{R}^n$ be the associated isometry.

\subsection{Approximation hierarchy for homogeneous polynomials}\label{subsec:previous_work}
If we define $\bm{x}\lozenge\bm{y} \coloneqq \phi(\phi^{-1}(\bm{x})\circ \phi^{-1}(\bm{y}))$ and $\bm{x}\blacklozenge\bm{y} \coloneqq \bm{x}^\top\bm{y}$ for $\bm{x},\bm{y}\in\mathbb{R}^n$, then $(\mathbb{R}^n,\lozenge,\blacklozenge)$ is also a Euclidean Jordan algebra.
Hereafter, to emphasize that $\mathbb{R}^n$ is a Euclidean Jordan algebra and $(\mathbb{R}^n,\lozenge,\blacklozenge)$ depends on the choice of $\phi$, we write the Euclidean Jordan algebra $\mathbb{R}^n$ as $\phi(\mathbb{E})$.
We note that the symmetric cone $\phi(\mathbb{E})_+$ associated with the Euclidean Jordan algebra $(\phi(\mathbb{E}),\lozenge,\blacklozenge)$ satisfies
\begin{equation}
\phi(\mathbb{E})_+ = \{\bm{x}\lozenge\bm{x}\mid \bm{x}\in\phi(\mathbb{E})\} = \{\phi(x)\lozenge \phi(x)\mid x\in\mathbb{E}\} = \phi(\mathbb{E}_+). \label{eq:symmetric_cone}
\end{equation}
We illustrate an example of \eqref{eq:symmetric_cone}, the isometry $\phi$, and the product $\lozenge$ using the Euclidian Jordan algebra introduced in Example~\ref{ex:sdc}.

\begin{example}
Let $(\mathbb{S}^n,\circ,\bullet)$ be the Euclidian Jordan algebra introduced in Example~\ref{ex:sdc}.
We define the linear mapping $\svec\colon \mathbb{S}^n\to\mathbb{R}^{\frac{n(n+1)}{2}}$ as
\begin{equation*}
\svec(\bm{X}) \coloneqq (X_{11},\sqrt{2}X_{12},X_{22},\dots,\sqrt{2}X_{1n},\dots,\sqrt{2}X_{n-1,n},X_{nn})
\end{equation*}
for each $\bm{X}\in\mathbb{S}^n$.
Then, the mapping $\svec(\cdot)$ is an isometry between $\mathbb{S}^n$ and $\mathbb{R}^{\frac{n(n+1)}{2}}$; i.e., $\bm{X}\bullet \bm{Y} = \svec(\bm{X})^\top\svec(\bm{Y})$ holds for all $\bm{X},\bm{Y}\in\mathbb{S}^n$.
Let $\smat(\cdot)$ denote the inverse mapping of $\svec(\cdot)$.
If we define $\bm{x}\lozenge \bm{y} \coloneqq \svec(\smat(\bm{x})\circ \smat(\bm{y}))$, and $\bm{x}\blacklozenge \bm{y} \coloneqq \bm{x}^\top\bm{y}$ for $\bm{x},\bm{y}\in\mathbb{R}^{\frac{n(n+1)}{2}}$, then $(\mathbb{R}^{\frac{n(n+1)}{2}},\lozenge,\blacklozenge)$ is also a Euclidean Jordan algebra.
The symmetric cone associated with the Euclidean Jordan algebra $(\mathbb{R}^{\frac{n(n+1)}{2}},\lozenge,\blacklozenge)$ is $\svec(\mathbb{S}_+^n)$.
\end{example}

Here, we derive an inner-approximation hierarchy for the cone of homogeneous polynomials that are nonnegative over the symmetric cone $\phi(\mathbb{E}_+)$.
Let
\begin{equation*}
\widetilde{\COP}^{n,m}(\phi(\mathbb{E_+})) \coloneqq \{\theta\in H^{n,m}\mid \theta(\bm{x}) \ge 0 \text{ for all $\bm{x}\in\phi(\mathbb{E}_+)$}\}
\end{equation*}
be the cone of homogeneous polynomials in $n$ variables of degree $m$ that are nonnegative over the symmetric cone $\phi(\mathbb{E}_+)$.
For each $r\in \mathbb{N}$, we define
\begin{equation*}
\widetilde{\mathcal{K}}_{\mathrm{NN},r}^{n,m}(\phi(\mathbb{E}_+)) \coloneqq \{\theta\in H^{n,m} \mid (\bm{x}^\top\bm{x})^r\theta(\bm{x}\lozenge\bm{x}) \in \Sigma^{n,2(r+m)}\}.
\end{equation*}

\begin{proposition}\label{prop:COP_approx_poly}
Each $\widetilde{\mathcal{K}}_{\mathrm{NN},r}^{n,m}(\phi(\mathbb{E}_+))$ is a closed convex cone, and the sequence $\{\widetilde{\mathcal{K}}_{\mathrm{NN},r}^{n,m}(\phi(\mathbb{E}_+))\}_r$ satisfies the following two conditions:
\begin{enumerate}[(i)]
\item $\widetilde{\mathcal{K}}_{\mathrm{NN},r}^{n,m}(\phi(\mathbb{E}_+)) \subseteq \widetilde{\mathcal{K}}_{\mathrm{NN},r+1}^{n,m}(\phi(\mathbb{E}_+)) \subseteq \widetilde{\COP}^{n,m}(\phi(\mathbb{E_+}))$ for all $r\in \mathbb{N}$. \label{enum:increasing}
\item $\setint\widetilde{\COP}^{n,m}(\phi(\mathbb{E_+})) \subseteq \bigcup_{r=0}^{\infty}\widetilde{\mathcal{K}}_{\mathrm{NN},r}^{n,m}(\phi(\mathbb{E}_+))$. \label{enum:convergence}
\end{enumerate}
\end{proposition}
In the following, the notation ``$\widetilde{\mathcal{K}}_{\mathrm{NN},r}^{n,m}(\phi(\mathbb{E}_+)) \uparrow \widetilde{\COP}^{n,m}(\phi(\mathbb{E_+}))$'' is used to represent the two conditions mentioned in Proposition~\ref{prop:COP_approx_poly}.
The notation is not limited to the sequence $\{\widetilde{\mathcal{K}}_{\mathrm{NN},r}^{n,m}(\phi(\mathbb{E}_+))\}_r$.
To prove Proposition~\ref{prop:COP_approx_poly}, we exploit Reznick's Positivstellensatz, which is described as follows:
\begin{theorem}[{\cite[Theorem~3.12]{Reznick1995}}]\label{theorem:Reznick}
Let $\theta\in H^{n,2m}$.
If $\theta(\bm{x}) > 0$ for all $\bm{x}\in\mathbb{R}^n\setminus\{\bm{0}\}$, then there exists $r_0\in \mathbb{N}$ such that $(\bm{x}^\top\bm{x})^{r_0}\theta(\bm{x}) \in \Sigma^{n,2(r_0+m)}$.
\end{theorem}

\begin{proof}[Proof of Proposition~\ref{prop:COP_approx_poly}]
Note that $\theta(\bm{x}\lozenge\bm{x})\in H^{n,2m}$ for each $\theta\in H^{n,m}$ because the product $\lozenge$ is bilinear; thus, the set $\widetilde{\mathcal{K}}_{\mathrm{NN},r}^{n,m}(\phi(\mathbb{E}_+))$ is well-defined for each $r\in\mathbb{N}$.
$\widetilde{\mathcal{K}}_{\mathrm{NN},r}^{n,m}(\phi(\mathbb{E}_+))$ can be shown to be a closed convex cone from the counterpart properties of $\Sigma^{n,2(r+m)}$.
In the following, we prove $\widetilde{\mathcal{K}}_{\mathrm{NN},r}^{n,m}(\phi(\mathbb{E}_+)) \uparrow \widetilde{\COP}^{n,m}(\phi(\mathbb{E}_+))$.

To prove \eqref{enum:increasing}, let $\theta\in\widetilde{\mathcal{K}}_{\mathrm{NN},r}^{n,m}(\phi(\mathbb{E}_+))$.
Then, there exist $p_1,\dots,p_N\in H^{n,r+m}$ such that
\begin{equation}
(\bm{x}^\top\bm{x})^r\theta(\bm{x}\lozenge\bm{x}) = \sum_{i=1}^Np_i^2(\bm{x}). \label{eq:sos_constraint}
\end{equation}
Using this, we obtain
\begin{equation*}
(\bm{x}^\top\bm{x})^{r+1}\theta(\bm{x}\lozenge\bm{x}) = \sum_{i=1}^N\sum_{j=1}^n(x_jp_i(\bm{x}))^2 \in \Sigma^{n,2(r+m+1)},
\end{equation*}
which means that $\theta\in \widetilde{\mathcal{K}}_{{\rm NN},r+1}^{n,m}(\phi(\mathbb{E}_+))$.
Now, we assume that $\theta \not\in \widetilde{\COP}^{n,m}(\phi(\mathbb{E_+}))$.
Then, from \eqref{eq:symmetric_cone}, it follows that there exists $\widetilde{\bm{x}}\in \phi(\mathbb{E})$ such that $\theta(\widetilde{\bm{x}}\lozenge\widetilde{\bm{x}}) < 0$.
As $\widetilde{\bm{x}} \not= \bm{0}$ and $\widetilde{\bm{x}}^\top\widetilde{\bm{x}} > 0$, we have $(\widetilde{\bm{x}}^\top\widetilde{\bm{x}})^r\theta(\widetilde{\bm{x}}\lozenge\widetilde{\bm{x}}) < 0$.
However, \eqref{eq:sos_constraint} implies that $(\widetilde{\bm{x}}^\top\widetilde{\bm{x}})^r\theta(\widetilde{\bm{x}}\lozenge\widetilde{\bm{x}})$ must take a nonnegative value, which is a contradiction.
Therefore, we obtain $\theta\in \widetilde{\COP}^{n,m}(\phi(\mathbb{E_+}))$.

To prove \eqref{enum:convergence}, let $\theta\in \setint\widetilde{\COP}^{n,m}(\phi(\mathbb{E_+}))$.
Then, as $\phi(\mathbb{E_+})$ is a closed cone, it follows that $\theta(\bm{y}) > 0$ for all $\bm{y}\in \phi(\mathbb{E_+})\setminus\{\bm{0}\}$, i.e., $\theta(\bm{x}\lozenge\bm{x}) > 0$ for all $\bm{x}\in \phi(\mathbb{E})\setminus\{\bm{0}\} = \mathbb{R}^n\setminus\{\bm{0}\}$ (see \cite[Observation~1]{ZVP2006}, for example).
Thus, by Theorem~\ref{theorem:Reznick}, there exists $r_0\in\mathbb{N}$ such that $(\bm{x}^\top\bm{x})^{r_0}\theta(\bm{x}\lozenge\bm{x}) \in \Sigma^{n,2(r_0+m)}$, which means that $\theta\in \widetilde{\mathcal{K}}_{\mathrm{NN},r_0}^{n,m}(\phi(\mathbb{E}_+)) \subseteq \bigcup_{r=0}^{\infty}\widetilde{\mathcal{K}}_{\mathrm{NN},r}^{n,m}(\phi(\mathbb{E}_+))$.
\end{proof}

Note that the set $\widetilde{\mathcal{K}}_{\mathrm{NN},r}^{n,m}(\phi(\mathbb{E}_+))$ is defined by an SOS constraint.
This constraint can be written as a semidefinite constraint of size $|\mathbb{I}_{=r+m}^n| = {n+r+m-1 \choose n-1}$, which is polynomial in $n$ and $m$ for each fixed $r\in\mathbb{N}$.

\subsection{Approximation hierarchy for symmetric tensors}\label{subsec:case_tensor}
In this subsection, we translate the result of Proposition~\ref{prop:COP_approx_poly} to the case of symmetric tensors.
Given that $(x\bullet x)^r\langle \mathcal{A},(x\circ x)^{\otimes m}\rangle \in H^{n,2(r+m)}(\mathbb{E})$ for each $r\in \mathbb{N}$ and $\mathcal{A}\in\mathcal{S}^{n,m}(\mathbb{E})$, there exists a unique $\mathcal{A}^{(r)}\in \mathcal{S}^{n,2(r+m)}(\mathbb{E})$ such that
\begin{equation*}
(x\bullet x)^r \langle \mathcal{A},(x\circ x)^{\otimes m}\rangle = \langle \mathcal{A}^{(r)},x^{\otimes 2(r+m)}\rangle.
\end{equation*}
Using the symmetric tensor $\mathcal{A}^{(r)}$, we define
\begin{equation*}
 \mathcal{K}_{\mathrm{NN},r}^{n,m}(\mathbb{E}_+) \coloneqq \{\mathcal{A}\in\mathcal{S}^{n,m}(\mathbb{E})\mid \mathcal{A}^{(r)}\in \SOS^{n,2(r+m)}(\mathbb{E})\}.
\end{equation*}

\begin{theorem}
Using the orthonormal basis $v_1,\dots,v_n$ for $\mathbb{E}$,
we define $\psi\in \Hom(\mathcal{S}^{n,m}(\mathbb{E}),H^{n,m})$ in the same manner as for \eqref{eq:isom_Snm(V)_Hnm}.
Then,
\begin{enumerate}[(i)]
\item $\psi(\COP^{n,m}(\mathbb{E}_+)) = \widetilde{\COP}^{n,m}(\phi(\mathbb{E}_+))$. \label{enum:COP}
\item $\psi(\mathcal{K}_{\mathrm{NN},r}^{n,m}(\mathbb{E}_+)) = \widetilde{\mathcal{K}}_{\mathrm{NN},r}^{n,m}(\phi(\mathbb{E}_+))$. \label{enum:K}
\item Each $\mathcal{K}_{\mathrm{NN},r}^{n,m}(\mathbb{E}_+)$ is a closed convex cone, and the sequence $\{\mathcal{K}_{\mathrm{NN},r}^{n,m}(\mathbb{E}_+)\}_r$ satisfies $\mathcal{K}_{\mathrm{NN},r}^{n,m}(\mathbb{E}_+) \uparrow \COP^{n,m}(\mathbb{E}_+)$. \label{enum:K_convergence}
\end{enumerate}
\end{theorem}
\begin{proof}
We first prove \eqref{enum:COP}.
Let $\theta = \psi(\mathcal{A}) \in \psi(\COP^{n,m}(\mathbb{E}_+))$ and $\mathcal{A} \in \COP^{n,m}(\mathbb{E}_+)$.
Then, for any $x\in \mathbb{E}$, we have
\begin{align*}
\theta(\phi(x)\lozenge \phi(x)) &= \langle \mathcal{A},\phi^{-1}(\phi(x)\lozenge\phi(x))^{\otimes m}\rangle\\
&= \langle \mathcal{A},(x\circ x)^{\otimes m}\rangle\\
&\ge 0,
\end{align*}
using $\phi^{-1}(\phi(x)\lozenge\phi(x)) = x\circ x$ and $\mathcal{A}\in\COP^{n,m}(\mathbb{E}_+)$.
Therefore, we have $\theta\in \widetilde{\COP}^{n,m}(\phi(\mathbb{E}_+))$.
Conversely, let $\theta \in \widetilde{\COP}^{n,m}(\phi(\mathbb{E}_+))$ and $\mathcal{A} \coloneqq \psi^{-1}(\theta)\in \mathcal{S}^{n,m}(\mathbb{E})$.
Then, for any $x\in \mathbb{E}$, in the same manner as the discussion above, we obtain $\langle \mathcal{A},(x\circ x)^{\otimes m}\rangle = \theta(\phi(x)\lozenge\phi(x)) \ge 0$,
using $\phi(x)\lozenge\phi(x)\in \phi(\mathbb{E}_+)$.
Therefore, $\mathcal{A}\in\COP^{n,m}(\mathbb{E}_+)$, and thus, $\theta = \psi(\mathcal{A}) \in \psi(\COP^{n,m}(\mathbb{E}_+))$.

Second, we prove \eqref{enum:K}.
Let $\theta = \psi(\mathcal{A}) \in \psi(\mathcal{K}_{\mathrm{NN},r}^{n,m}(\mathbb{E}_+))$ and $\mathcal{A} \in \mathcal{K}_{\mathrm{NN},r}^{n,m}(\mathbb{E}_+)$.
As $\mathcal{A}\in \mathcal{K}_{\mathrm{NN},r}^{n,m}(\mathbb{E}_+)$, there exist $\mathcal{A}_1,\dots,\mathcal{A}_N\in \mathcal{S}^{n,r+m}(\mathbb{E})$ such that $(x\bullet x)^r\langle \mathcal{A},(x\circ x)^{\otimes m}\rangle = \sum_{i=1}^N\langle\mathcal{A}_i,x^{\otimes (r+m)}\rangle^2$.
As $\langle \mathcal{A}_i,\phi^{-1}(\bm{x})^{\otimes (r+m)}\rangle\in H^{n,r+m}$ for each $i$, it follows that
\begin{align*}
(\bm{x}^\top\bm{x})^r\theta(\bm{x}\lozenge\bm{x}) &= (\phi^{-1}(\bm{x}) \bullet \phi^{-1}(\bm{x}))^r\langle\mathcal{A},\phi^{-1}(\bm{x}\lozenge\bm{x})^{\otimes m}\rangle\\
&= (\phi^{-1}(\bm{x}) \bullet \phi^{-1}(\bm{x}))^r\langle \mathcal{A},(\phi^{-1}(\bm{x})\circ \phi^{-1}(\bm{x}))^{\otimes m}\rangle \\
&= \sum_{i=1}^N\langle \mathcal{A}_i,\phi^{-1}(\bm{x})^{\otimes (r+m)}\rangle^2 \in \Sigma^{n,2(r+m)}.
\end{align*}
Therefore, we have $\theta\in \widetilde{\mathcal{K}}_{\mathrm{NN},r}^{n,m}(\phi(\mathbb{E}_+))$.
Conversely, let $\theta \in \widetilde{\mathcal{K}}_{\mathrm{NN},r}^{n,m}(\phi(\mathbb{E}_+))$.
Then, there exist $\mathcal{A}_1,\dots,\mathcal{A}_N\in \mathcal{S}^{n,r+m}(\mathbb{E})$ such that $(\bm{x}^\top\bm{x})^r\theta(\bm{x}\lozenge\bm{x}) = \sum_{i=1}^N\langle \mathcal{A}_i,\phi^{-1}(\bm{x})^{\otimes (r+m)}\rangle^2$.
Let $\mathcal{A} \coloneqq \psi^{-1}(\theta)\in \mathcal{S}^{n,m}(\mathbb{E})$.
Then, in the same manner as in \eqref{enum:COP}, we have
\begin{align*}
(x\bullet x)^r\langle\mathcal{A},(x\circ x)^{\otimes m}\rangle &= (\phi(x)^\top\phi(x))^r\theta(\phi(x)\lozenge\phi(x))\\
&= \sum_{i=1}^N\langle \mathcal{A}_i,x^{\otimes (r+m)}\rangle^2 \in \Sigma^{n,2(r+m)}(\mathbb{E}).
\end{align*}

Finally, \eqref{enum:K_convergence} can be proven by the linear isomorphism of $\psi$ and $\widetilde{\mathcal{K}}_{\mathrm{NN},r}^{n,m}(\phi(\mathbb{E}_+)) \uparrow \widetilde{\COP}^{n,m}(\phi(\mathbb{E_+}))$.
\end{proof}

Note that $\mathcal{K}_{\mathrm{NN},r}^{n,m}(\mathbb{E}_+)$ does not depend on the choice of the orthonormal basis $v_1,\dots,v_n$ or isometry $\phi$.
We call the sequence $\{\mathcal{K}_{\mathrm{NN},r}^{n,m}(\mathbb{E}_+)\}_r$ the NN-type inner-approximation hierarchy.
In the case where the symmetric cone $\mathbb{E}_+$ is the nonnegative orthant $\mathbb{R}_+^n$, this hierarchy is essentially identical to the SOS-based one provided by Iqbal and Ahmed~\cite[Eq.~(19)]{IA2022}.
They are generalizations of that provided by Parrilo~\cite{Parrilo2000} to tensors.

\begin{remark}\label{rem:SOScone}
The NN-type inner-approximation hierarchy can be extended to the \textit{SOS cones} proposed by Papp and Alizadeh~\cite{PA2013}.
Let $\mathbb{A}$ and $ \mathbb{B}$ be real inner product spaces of dimensions $l$ and $n$, respectively, and let $\diamond\colon\mathbb{A}\times \mathbb{A}\to \mathbb{B}$ be a bilinear mapping.
The SOS cone is then defined as $\Sigma_{\diamond} \coloneqq \conv\{x_i \diamond x_i \mid x_i\in \mathbb{A}\}$.
Note that each element of $\Sigma_{\diamond}$ can be written as the sum of at most $n$ elements $x_1\diamond x_1,\dots,x_n\diamond x_n$ such that $x_1,\dots,x_n\in\mathbb{A}$, by Carath\'{e}odory's theorem for cones.
If $(\mathbb{A},\mathbb{B},\diamond)$ is formally real, or equivalently, if $\Sigma_{\diamond}$ is proper~\cite[Theorem~3.3]{PA2013}, then
\begin{equation*}
\mathcal{K}_r^{n,m}(\Sigma_{\diamond}) \coloneqq \left\{\mathcal{A}\in\mathcal{S}^{n,m}(\mathbb{B}) \relmiddle|
\begin{aligned}
\left(\sum_{i=1}^nx_i \bullet_{\mathbb{A}}x_i\right)^r\left\langle \mathcal{A},\left(\sum_{i=1}^n x_i\diamond x_i\right)^{\otimes m}\right\rangle\\
\in \Sigma^{ln,2(r+m)}(\mathbb{A}^n)
\end{aligned}
\right\}
\end{equation*}
is a closed convex cone for each $r\in\mathbb{N}$ and satisfies $\mathcal{K}_r^{n,m}(\Sigma_{\diamond}) \uparrow \COP^{n,m}(\Sigma_{\diamond})$, where $\bullet_{\mathbb{A}}$ denotes the inner product on $\mathbb{A}$.
\end{remark}

\subsection{Dual of $\mathcal{K}_{\mathrm{NN},r}^{n,m}(\mathbb{E}_+)$}\label{subsec:dual}
Next, we discuss the dual cone of $\mathcal{K}_{\mathrm{NN},r}^{n,m}(\mathbb{E}_+)$.
By considering its dual, we can provide an outer-approximation hierarchy for the CP cone $\CP^{n,m}(\mathbb{E}_+)$.
Although the closure hull operator is generally required to describe the dual cone, we succeeded in removing it for the case in which the symmetric cone $\mathbb{E}_+$ is the nonnegative orthant $\mathbb{R}_+^n$.

Let $\mathcal{C}^{(r)}\colon \mathcal{S}^{n,2(r+m)}(\mathbb{E})\to\mathcal{S}^{n,m}(\mathbb{E})$ be the adjoint of the linear mapping $\mathcal{A} \mapsto \mathcal{A}^{(r)}$, so that
\begin{equation}
\langle \mathcal{A},\mathcal{C}^{(r)}(\mathcal{X})\rangle = \langle\mathcal{A}^{(r)},\mathcal{X}\rangle \text{ for all  $\mathcal{A}\in\mathcal{S}^{n,m}(\mathbb{E})$ and $\mathcal{X}\in\mathcal{S}^{n,2(r+m)}(\mathbb{E})$}. \label{eq:C(X)}
\end{equation}
Using the linear mapping $\mathcal{C}^{(r)}$, we define $\mathcal{C}_r^{n,m}(\mathbb{E}_+) \coloneqq \{\mathcal{C}^{(r)}(\mathcal{X}) \mid \mathcal{X} \in \MOM^{n,2(r+m)}(\mathbb{E})\}$.

\begin{proposition}
It follows that $\mathcal{K}_{\mathrm{NN},r}^{n,m}(\mathbb{E}_+) = \mathcal{C}_r^{n,m}(\mathbb{E}_+)^*$, and thus, $\mathcal{K}_{\mathrm{NN},r}^{n,m}(\mathbb{E}_+)^* = \cl\mathcal{C}_r^{n,m}(\mathbb{E}_+)$.
\end{proposition}
\begin{proof}
Let $\mathcal{A}\in \mathcal{K}_{\mathrm{NN},r}^{n,m}(\mathbb{E}_+)$.
Then, $\mathcal{A}^{(r)}\in\SOS^{n,2(r+m)}(\mathbb{E})$ by definition.
For any $\mathcal{X}\in \MOM^{n,2(r+m)}(\mathbb{E})$, it follows from \eqref{eq:C(X)} and the duality between $\MOM^{n,2(r+m)}(\mathbb{E})$ and $\SOS^{n,2(r+m)}(\mathbb{E})$ that $\langle \mathcal{A},\mathcal{C}^{(r)}(\mathcal{X})\rangle = \langle \mathcal{A}^{(r)},\mathcal{X}\rangle \ge 0$, which means that $\mathcal{A}\in \mathcal{C}_r^{n,m}(\mathbb{E}_+)^*$.

Conversely, let $\mathcal{A}\in \mathcal{C}_r^{n,m}(\mathbb{E}_+)^*$.
Then, for any $\mathcal{X}\in \MOM^{n,2(r+m)}(\mathbb{E})$, because $\mathcal{C}^{(r)}(\mathcal{X}) \in \mathcal{C}_r^{n,m}(\mathbb{E}_+)$, it follows that $\langle \mathcal{A}^{(r)},\mathcal{X}\rangle = \langle \mathcal{A},\mathcal{C}^{(r)}(\mathcal{X})\rangle \ge 0$.
Therefore, $\mathcal{A}^{(r)} \in \MOM^{n,2(r+m)}(\mathbb{E})^* = \SOS^{n,2(r+m)}(\mathbb{E})$, which means that $\mathcal{A}\in \mathcal{K}_{\mathrm{NN},r}^{n,m}(\mathbb{E}_+)$.

Given that $\mathcal{C}_r^{n,m}(\mathbb{E}_+)$ is a convex cone, by taking the dual, we have $\mathcal{K}_{\mathrm{NN},r}^{n,m}(\mathbb{E}_+)^* = \cl\mathcal{C}_r^{n,m}(\mathbb{E}_+)$.
\end{proof}

We could not prove that $\mathcal{C}_r^{n,m}(\mathbb{E}_+)$ itself is closed, and the closure hull operator is required to describe the dual of $\mathcal{K}_{\mathrm{NN},r}^{n,m}(\mathbb{E}_+)$.
Therefore, whether $\mathcal{K}_{\mathrm{NN},r}^{n,m}(\mathbb{E}_+)$ and $\mathcal{C}_r^{n,m}(\mathbb{E}_+)$ itself are dual to each other is an open problem for a general symmetric cone $\mathbb{E}_+$.
The difficulty originates from a lack of understanding of the moment cone $\MOM^{n,2m}(\mathbb{E})$.

As a special case, we consider the symmetric cone $\mathbb{E}_+$ to be the nonnegative orthant $\mathbb{R}_+^n$.
In this case, we demonstrated in Proposition~3.9 that the dual of $\SOS^{n,2m}(\mathbb{R}^n)$ can be explicitly described by a semidefinite constraint;
this answers a question posed by Chen et al.~\cite{CLQ2016} of whether the membership problem of the dual of $\SOS^{n,2m}(\mathbb{R}^n)$ can be solved in polynomial time or not.
Furthermore, the closedness of $\mathcal{C}_r^{n,m}(\mathbb{R}_+^n)$ (written as $\mathcal{C}_r^{n,m}$ hereafter) is shown by exploiting this result.

\begin{definition}
For $\mathcal{X} \in \mathcal{S}^{n,2m}$, let $\bm{M}^{n,m}(\mathcal{X})$ be a matrix in $\mathbb{S}^{\mathbb{I}_{=m}^n}$ with the $(\bm{\alpha},\bm{\beta})$th element $\mathcal{X}_{[\bm{\alpha}+\bm{\beta}]}$ for $\bm{\alpha},\bm{\beta}\in\mathbb{I}_{=m}^n$.
Then, we define $\mathcal{M}^{n,2m} \coloneqq \{\mathcal{X}\in\mathcal{S}^{n,2m} \mid \bm{M}^{n,m}(\mathcal{X})\in\mathbb{S}_+^{\mathbb{I}_{=m}^n}\}$.
\end{definition}

We aim to demonstrate that the dual of $\SOS^{n,2m}(\mathbb{R}^n)$ agrees with $\mathcal{M}^{n,2m}$.

\begin{lemma}\label{eq:mom_matrix}
For $\mathcal{A}\in\mathcal{S}^{n,m}$, let $\bm{\theta}$ be the element of $\mathbb{R}^{\mathbb{I}_{=m}^n}$ satisfying \eqref{eq:Ax_thetax}.
Then, $\langle \mathscr{S}(\mathcal{A}\otimes \mathcal{A}),\mathcal{X}\rangle = \bm{\theta}^\top \bm{M}^{n,m}(\mathcal{X})\bm{\theta}$ for all $\mathcal{X}\in\mathcal{S}^{n,2m}$.
\end{lemma}
\begin{proof}
Let $(\mathcal{F}_{\bm{\alpha}})_{\bm{\alpha}\in \mathbb{I}_{=m}^n}$ be the orthonormal basis for $\mathcal{S}^{n,m}$ defined in Lemma~\ref{lem:orth_basis}.
From Lemma~\ref{lem:Ax^2} and \eqref{eq:Ax_thetax}, we have
\begin{equation*}
\langle \mathscr{S}(\mathcal{A}\otimes \mathcal{A}),\bm{x}^{\otimes 2m}\rangle = \langle \mathcal{A},\bm{x}^{\otimes m}\rangle^2 =  \sum_{\bm{\gamma}\in \mathbb{I}_{=2m}^n}\left(\sum_{\substack{\bm{\alpha},\bm{\beta}\in \mathbb{I}_{=m}^n \\ \bm{\gamma} = \bm{\alpha} + \bm{\beta}}} \theta_{\bm{\alpha}}\theta_{\bm{\beta}}\right)\bm{x}^{\bm{\gamma}}. \label{eq:Ax_thetax_squared}
\end{equation*}
Therefore, from Lemma~\ref{lem:A_orth_coef}, when we represent $\mathscr{S}(\mathcal{A}\otimes \mathcal{A})$ with the orthonormal basis $(\mathcal{F}_{\bm{\gamma}})_{\bm{\gamma}\in \mathbb{I}_{=2m}^n}$ for $\mathcal{S}^{n,2m}$, the coefficient $\mathscr{S}(\mathcal{A}\otimes \mathcal{A})_{\bm{\gamma}}$ can be written as
\begin{equation*}
\mathscr{S}(\mathcal{A}\otimes \mathcal{A})_{\bm{\gamma}}^{\mathcal{F}} = \sqrt{\frac{\bm{\gamma}!}{(2m)!}}\sum_{\substack{\bm{\alpha},\bm{\beta}\in \mathbb{I}_{=m}^n \\ \bm{\gamma} = \bm{\alpha} + \bm{\beta}}}\theta_{\bm{\alpha}}\theta_{\bm{\beta}}
\end{equation*}
for each $\bm{\gamma}\in \mathbb{I}_{=2m}^n$.
In addition, from Lemma~\ref{lem:orth_basis}, each $\mathcal{X}\in \mathcal{S}^{n,2m}$ can be written as
\begin{equation*}
\mathcal{X} = \sum_{\bm{\gamma}\in \mathbb{I}_{=2m}^n}\sqrt{\frac{(2m)!}{\bm{\gamma}!}}\mathcal{X}_{[\bm{\gamma}]}\mathcal{F}_{\bm{\gamma}}.
\end{equation*}
Then, we have
\begin{align*}
\langle \mathscr{S}(\mathcal{A}\otimes \mathcal{A}),\mathcal{X}\rangle &= \sum_{\bm{\gamma}\in \mathbb{I}_{=2m}^n}\left(\sqrt{\frac{\bm{\gamma}!}{(2m)!}}\sum_{\substack{\bm{\alpha},\bm{\beta}\in \mathbb{I}_{=m}^n \\ \bm{\gamma} = \bm{\alpha} + \bm{\beta}}}\theta_{\bm{\alpha}}\theta_{\bm{\beta}}\right)\left(\sqrt{\frac{(2m)!}{\bm{\gamma}!}}\mathcal{X}_{[\bm{\gamma}]}\right)\\
&= \sum_{\bm{\alpha},\bm{\beta}\in \mathbb{I}_{=m}^n}\mathcal{X}_{[\bm{\alpha}+\bm{\beta}]}\theta_{\bm{\alpha}}\theta_{\bm{\beta}}\\
&= \bm{\theta}^\top\bm{M}^{n,m}(\mathcal{X})\bm{\theta}.
\end{align*}
\end{proof}

\begin{proposition}\label{prop:MOM}
The dual of $\SOS^{n,2m}(\mathbb{R}^n)$ is $\mathcal{M}^{n,2m}$.
In particular, it follows that $\mathcal{M}^{n,2m} = \MOM^{n,2m}(\mathbb{R}^n)$.

\end{proposition}
\begin{proof}
Let $\mathcal{X}\in \mathcal{M}^{n,2m}$.
For each $\mathcal{A}\in \SOS^{n,2m}(\mathbb{R}^n)$, there exist $\mathcal{A}^{(1)},\dots,\mathcal{A}^{(k)}\in\mathcal{S}^{n,m}$ such that $\mathcal{A} = \sum_{i=1}^k\mathscr{S}(\mathcal{A}^{(i)}\otimes \mathcal{A}^{(i)})$.
For each $\mathcal{A}^{(i)}$, let $\bm{\theta}^{(i)} = (\theta_{\bm{\alpha}}^{(i)})_{\bm{\alpha}\in \mathbb{I}_{=m}^n} \in \mathbb{R}^{\mathbb{I}_{=m}^n}$ be such that it satisfies \eqref{eq:Ax_thetax}.
Then, it follows from Lemma~\ref{eq:mom_matrix} and $\bm{M}^{n,m}(\mathcal{X}) \in \mathbb{S}_+^{\mathbb{I}_{=m}^n}$ that
\begin{equation*}
\langle \mathcal{A},\mathcal{X}\rangle = \sum_{i=1}^k\langle\mathscr{S}(\mathcal{A}^{(i)}\otimes \mathcal{A}^{(i)}),\mathcal{X}\rangle = \sum_{i=1}^k(\bm{\theta}^{(i)})^\top \bm{M}^{n,m}(\mathcal{X})\bm{\theta}^{(i)} \ge 0,
\end{equation*}
which means that $\mathcal{X}\in \SOS^{n,2m}(\mathbb{R}^n)^*$.

Conversely, suppose that $\mathcal{X}\in \SOS^{n,2m}(\mathbb{R}^n)^*$.
We take $\bm{\theta} = (\theta_{\bm{\alpha}})_{\bm{\alpha}\in \mathbb{I}_{=m}^n} \in \mathbb{R}^{\mathbb{I}_{=m}^n}$ arbitrarily and let $\mathcal{A}$ be the element of $\mathcal{S}^{n,m}$ satisfying \eqref{eq:Ax_thetax}.
Then, it follows from Lemma~\ref{eq:mom_matrix} and $\mathscr{S}(\mathcal{A}\otimes \mathcal{A})\in \SOS^{n,2m}(\mathbb{R}^n)$ that $\bm{\theta}^\top\bm{M}^{n,m}(\mathcal{X})\bm{\theta} = \langle \mathscr{S}(\mathcal{A}\otimes \mathcal{A}),\mathcal{X}\rangle \ge 0.$
Therefore, $\mathcal{X}\in\mathcal{M}^{n,2m}$.

By the definition of $\MOM^{n,2m}(\mathbb{R}^n)$, it follows that $\MOM^{n,2m}(\mathbb{R}^n) = \SOS^{n,2m}(\mathbb{R}^n)^*$.
Combining this with $\SOS^{n,2m}(\mathbb{R}^n)^* = \mathcal{M}^{n,2m}$ shown above, we have $\mathcal{M}^{n,2m} = \MOM^{n,2m}(\mathbb{R}^n)$.
\end{proof}

Using Proposition~\ref{prop:MOM}, we show the closedness of $\mathcal{C}_r^{n,m}$.
When the Euclidean Jordan algebra $(\mathbb{E},\circ,\bullet)$ is that given by Example~\ref{ex:nno}, the $(i_1,\dots,i_m)$th element of $\mathcal{C}^{(r)}(\mathcal{X})\in \mathcal{S}^{n,m}$ is
\begin{equation}
\mathcal{C}^{(r)}(\mathcal{X})_{i_1\cdots i_m} = \sum_{\bm{\alpha}\in \mathbb{I}_{=r}^n}\frac{r!}{\bm{\alpha}!}\mathcal{X}_{[2\bm{\alpha} + 2\sum_{l=1}^m\bm{e}_{i_l}]}. \label{eq:C(X)_nno}
\end{equation}
See \cite[Eq.~(29)]{IA2022} for this derivation.

\begin{lemma}\label{lem:odd_zero}
Suppose that $\mathcal{X}\in\mathcal{M}^{n,2m}$, i.e., $\bm{M}^{n,m}(\mathcal{X}) \in \mathbb{S}_+^{\mathbb{I}_{=m}^n}$.
Using $\mathcal{X}$, we define $\mathcal{X}'\in \mathcal{S}^{n,2m}$ as
\begin{equation*}
\mathcal{X}'_{[\bm{\gamma}]} \coloneqq
\begin{cases}
\mathcal{X}_{[\bm{\gamma}]} & (\text{if all of the elements of $\bm{\gamma}$ are even}),\\
0 & (\text{if some of the elements of $\bm{\gamma}$ are odd})
\end{cases}
\end{equation*}
for each $\bm{\gamma}\in \mathbb{I}_{=2m}^n$.
It then follows that $\bm{M}^{n,m}(\mathcal{X}') \in  \mathbb{S}_+^{\mathbb{I}_{=m}^n}$.
\end{lemma}
\begin{proof}
The set $\mathbb{I}_{=m}^n$ is partitioned as two disjoint sets $\mathbb{I}_{=m,\mathrm{even}}^n$ and $\mathbb{I}_{=m,\mathrm{odd}}^n$, where
\begin{align*}
\mathbb{I}_{=m,\mathrm{even}}^n &\coloneqq \{\bm{\alpha}\in \mathbb{I}_{=m}^n\mid \text{All of the elements of $\bm{\alpha}$ are even}\},\\
\mathbb{I}_{=m,\mathrm{odd}}^n &\coloneqq \{\bm{\alpha}\in \mathbb{I}_{=m}^n\mid \text{Some of the elements of $\bm{\alpha}$ are odd}\}.
\end{align*}
In addition, the elements of $\{0,1\}^n\setminus \{\bm{0}\}$ are ordered as $\bm{\delta}_1,\dots,\bm{\delta}_{2^n-1}$ (e.g., $\bm{\delta}_1 = (1,0,\dots,0)$). Then, let
\begin{equation*}
\mathbb{I}_{=m,\mathrm{odd},i}^n \coloneqq \{\bm{\alpha} \in \mathbb{I}_{=m,\mathrm{odd}}^n\mid \text{All of the elements of $\bm{\alpha}-\bm{\delta}_i$ are even}\}
\end{equation*}
for $i = 1,\dots,2^n-1$.
Note that the parity between $\bm{\alpha}$ and $\bm{\delta}_i$ agrees for every $\bm{\alpha}\in \mathbb{I}_{=m,\mathrm{odd},i}^n$ and that the sets $\mathbb{I}_{=m,\mathrm{odd},i}^n$ $(i = 1,\dots,2^n-1)$ are disjoint.
Then, for $\bm{\alpha},\bm{\beta}\in \mathbb{I}_{=m}^n$, all elements of $\bm{\alpha} + \bm{\beta}$ are even if and only if $\bm{\alpha}$ and $\bm{\beta}$ belong to the same set of $\mathbb{I}_{=m,\mathrm{even}}^n,\mathbb{I}_{=m,\mathrm{odd},1}^n,\dots,\mathbb{I}_{=m,\mathrm{odd},2^n-1}^n$.
Therefore, from the definition of $\mathcal{X}'$, we note that
\begin{equation*}
\bm{M}^{n,m}(\mathcal{X}')_{IJ} = \begin{cases}
\bm{M}^{n,m}(\mathcal{X})_{IJ} & (\text{if $I = J$}),\\
\bm{O} & (\text{otherwise})
\end{cases}
\end{equation*}
for $I$, $J = \mathbb{I}_{=m,\mathrm{even}}^n, \mathbb{I}_{=m,\mathrm{odd},1}^n,\dots,\mathbb{I}_{=m,\mathrm{odd},2^n-1}^n$, where $\bm{M}^{n,m}(\mathcal{X}')_{IJ}$ is the submatrix obtained by extracting the rows of $\bm{M}^{n,m}(\mathcal{X}')$ indexed by $I$ and columns indexed by $J$.
Because $\bm{M}^{n,m}(\mathcal{X})$ is semidefinite, $\bm{M}^{n,m}(\mathcal{X}')$ is as well.
\end{proof}

\begin{proposition}
$\mathcal{C}_r^{n,m}$ is a closed convex cone, and thus, $\mathcal{K}_{\mathrm{NN},r}^{n,m}(\mathbb{R}_+^n)$ and $\mathcal{C}_r^{n,m}$ are dual to each other.
\end{proposition}
\begin{proof}
Showing the closedness of $\mathcal{C}_r^{n,m}$ is sufficient.
Let $\{\mathcal{A}_k\}_k \subseteq \mathcal{C}_r^{n,m}$ and suppose that the sequence converges to some $\mathcal{A}_{\infty} \in \mathcal{S}^{n,m}$.
For each $k$, there exists $\mathcal{X}_k\in \mathcal{M}^{n,2(r+m)}$ such that $\mathcal{A}_k = \mathcal{C}^{(r)}(\mathcal{X}_k)$.
We note from \eqref{eq:C(X)_nno} that $\mathcal{C}^{(r)}(\mathcal{X}_k)$ is independent of $(\mathcal{X}_k)_{[\bm{\gamma}]}$ such that some elements of $\bm{\gamma} \in \mathbb{I}_{=2(r+m)}^n$ are odd.
Thus, by Lemma~\ref{lem:odd_zero}, we can assume that such $(\mathcal{X}_k)_{[\bm{\gamma}]}$ is 0 without loss of generality.
As $\mathcal{C}^{(r)}(\mathcal{X}_k) \to \mathcal{A}_{\infty}$ $(k\to\infty)$, we observe that
\begin{equation*}
\|\mathcal{C}^{(r)}(\mathcal{X}_k)\|_{\rm F} = \sqrt{
\sum_{i_1,\dots,i_m=1}^n\left(\sum_{\bm{\alpha}\in \mathbb{I}_{=r}^n}\frac{r!}{\bm{\alpha}!}(\mathcal{X}_k)_{[2\bm{\alpha} + 2\sum_{l=1}^m\bm{e}_{i_l}]}\right)^2} \to \|\mathcal{A}_{\infty}\|_{\rm F}.
\end{equation*}
Note that each $(\mathcal{X}_k)_{[2\bm{\alpha} + 2\sum_{l=1}^m\bm{e}_{i_l}]}$ is nonnegative because it is a diagonal element of the semidefinite matrix $\bm{M}^{n,r+m}(\mathcal{X}_k)$.
Therefore, $\{(\mathcal{X}_k)_{[2\bm{\alpha} + 2\sum_{l=1}^m\bm{e}_{i_l}]}\}_k$ is bounded for each $(i_1,\dots,i_m)\in \mathbb{N}_n^m$ and $\bm{\alpha}\in \mathbb{I}_{=r}^n$, and $\{\mathcal{X}_k\}_k$ is as well.
Thus, by taking a subsequence if necessary, we assume that the sequence $\{\mathcal{X}_k\}_k$ converges to some $\mathcal{X}_{\infty}$.
Given that $\{\mathcal{X}_k\}_k \subseteq \mathcal{M}^{n,2(r+m)}$ and $\mathcal{M}^{n,2(r+m)}$ is closed, we have $\mathcal{X}_{\infty} \in \mathcal{M}^{n,2(r+m)}$.
Therefore,
\begin{equation*}
\mathcal{A}_{\infty} = \lim_{k\to\infty}\mathcal{C}^{(r)}(\mathcal{X}_k) = \mathcal{C}^{(r)}(\mathcal{X}_{\infty}) \in \mathcal{C}_r^{n,m},
\end{equation*}
which implies that $\mathcal{C}_r^{n,m}$ is closed.
\end{proof}

\section{Approximation hierarchies exploiting those for the usual copositive cone}\label{sec:approx_COP_spectrum}
In this section, we provide other approximation hierarchies for the COP cone over a symmetric cone by exploiting those for the usual COP cone.
Let $(\mathbb{E},\circ,\bullet)$ be a Euclidean Jordan algebra of dimension $n$.
In addition, for $\mathcal{A}\in\mathcal{S}^{n,m}(\mathbb{E})$ and an ordered Jordan frame $(c_1,\dots,c_{\rk})\in\mathfrak{F}(\mathbb{E})$, let $\mathcal{A}(c_1,\dots,c_{\rk})\in\mathcal{S}^{\rk,m}$ be the tensor with the $(i_1,\dots,i_m)$th element $\langle\mathcal{A},c_{i_1}\otimes\cdots\otimes c_{i_m}\rangle$.
(The tensor is guaranteed to be symmetric by the symmetry of $\mathcal{A}$.)

The following lemma is key to providing the desired approximation hierarchies.

\begin{lemma}\label{lem:COP_spectrum}
\begin{equation}
\COP^{n,m}(\mathbb{E}_+) = \bigcap_{(c_1,\dots,c_{\rk})\in\mathfrak{F}(\mathbb{E})}\{\mathcal{A}\in\mathcal{S}^{n,m}(\mathbb{E})\mid\mathcal{A}(c_1,\dots,c_{\rk})\in\COP^{\rk,m}\}. \label{eq:COP_characterization}
\end{equation}
\end{lemma}

\begin{proof}
Let $\mathcal{A}\in\COP^{n,m}(\mathbb{E}_+)$.
For any $(c_1,\dots,c_{\rk})\in\mathfrak{F}(\mathbb{E})$ and $\bm{x}\in\mathbb{R}_+^{\rk}$, we have
\begin{align*}
\langle \mathcal{A}(c_1,\dots,c_{\rk}),\bm{x}^{\otimes m}\rangle &= \sum_{i_1,\dots,i_m=1}^{\rk}\langle\mathcal{A},c_{i_1}\otimes \cdots \otimes c_{i_m}\rangle x_{i_1}\cdots x_{i_m}\\
&= \left\langle \mathcal{A},\left(\sum_{i=1}^{\rk}x_ic_i\right)^{\otimes m}\right\rangle,
\end{align*}
which is nonnegative given that $\sum_{i=1}^{\rk}x_ic_i \in \mathbb{E}_+$.

Conversely, suppose that $\mathcal{A}$ belongs to the right-hand side set of \eqref{eq:COP_characterization}.
For any $x\in\mathbb{E}_+$, there exist $(x_1,\dots,x_{\rk})\in\mathbb{R}_+^{\rk}$ and $(c_1,\dots,c_{\rk})\in\mathfrak{F}(\mathbb{E})$ such that $x = \sum_{i=1}^{\rk}x_ic_i$.
As $\mathcal{A}(c_1,\dots,c_{\rk})\in\COP^{\rk,m}$, we can show $\langle \mathcal{A},x^{\otimes m}\rangle\ge 0$, i.e., $\mathcal{A}\in\COP^{n,m}(\mathbb{E}_+)$ in the same manner as the above discussion.
\end{proof}

The characterization of the COP cone over a symmetric cone is somewhat redundant because the Jordan frames we considered are ordered, and thus, the same set appears multiple times on the right-hand side of \eqref{eq:COP_characterization}.
To solve this problem, we provide a more concise characterization of $\COP^{n,m}(\mathbb{E}_+)$.

\begin{definition}\label{def:permutation_invariant}
For each $\sigma\in \mathfrak{S}_n$, let $\mathscr{P}_{\sigma}$ be the linear transformation on $\mathcal{S}^{n,m}$ defined by $(\mathscr{P}_{\sigma}\mathcal{A})_{i_1\cdots i_m} \coloneqq \mathcal{A}_{\sigma(i_1)\cdots \sigma(i_m)}$ for each $\mathcal{A}\in\mathcal{S}^{n,m}$ and $(i_1,\dots,i_m)\in\mathbb{N}_n^m$.
A set $\mathcal{K} \subseteq \mathcal{S}^{n,m}$ is said to be permutation-invariant if $\mathscr{P}_{\sigma}\mathcal{K} = \mathcal{K}$ for all $\sigma\in\mathfrak{S}_n$.
\end{definition}
Note that because $\mathscr{P}_{\sigma}$ is invertible and $\mathscr{P}_{\sigma}^{-1} = \mathscr{P}_{\sigma^{-1}}$ for all $\sigma \in \mathfrak{S}_n$, Definition~\ref{def:permutation_invariant} can be equivalently stated such that $\mathscr{P}_{\sigma}\mathcal{A} \in \mathcal{K}$ holds for all $\mathcal{A}\in\mathcal{K}$ and $\sigma\in\mathfrak{S}_n$.

\begin{lemma}\label{lem:COP_permutation_invariant}
$\COP^{n,m}$ is permutation-invariant.
\end{lemma}

\begin{proof}
Let $\mathcal{A}\in \COP^{n,m}$.
Then, for any $\sigma\in\mathfrak{S}_n$ and $\bm{x}\in\mathbb{R}_+^n$, we have
\begin{align*}
\langle \mathscr{P}_{\sigma}\mathcal{A},\bm{x}^{\otimes m}\rangle &= \sum_{i_1,\dots,i_m=1}^n\mathcal{A}_{\sigma(i_1)\cdots \sigma(i_m)}x_{i_1}\cdots x_{i_m}\\
&= \sum_{i_1,\dots,i_m=1}^n\mathcal{A}_{i_1\cdots i_m}x_{\sigma^{-1}(i_1)}\cdots x_{\sigma^{-1}(i_m)}\\
&= \langle \mathcal{A},(x_{\sigma^{-1}(1)},\dots,x_{{\sigma}^{-1}(n)})^{\otimes m}\rangle\\
&\ge 0,
\end{align*}
for which the last inequality follows from the fact that $(x_{\sigma^{-1}(1)},\dots,x_{{\sigma}^{-1}(n)}) \in \mathbb{R}_+^n$ if $\bm{x}\in\mathbb{R}_+^n$.
\end{proof}

\begin{lemma}\label{lem:frame_permutation_invariant}
Let $\mathcal{K} \subseteq \mathcal{S}^{\rk,m}$ be a permutation-invariant set and let $\mathcal{A}\in\mathcal{S}^{n,m}(\mathbb{E})$.
If $\mathcal{A}(c_1,\dots,c_{\rk})\in\mathcal{K}$ for a given $(c_1,\dots,c_{\rk})\in\mathfrak{F}(\mathbb{E})$, then $\mathcal{A}(c_{\sigma(1)},\dots,c_{\sigma(\rk)})\in \mathcal{K}$ for all $\sigma\in\mathfrak{S}_{\rk}$.
\end{lemma}

\begin{proof}
Because $\mathcal{K}$ is permutation-invariant, we have $\mathscr{P}_{\sigma}\mathcal{A}(c_1,\dots,c_{\rk}) \in \mathcal{K}$.
Then, the $(i_1,\dots,i_m)$th element of $\mathscr{P}_{\sigma}\mathcal{A}(c_1,\dots,c_{\rk})$ is
\begin{equation*}
\mathcal{A}(c_1,\dots,c_{\rk})_{\sigma(i_1)\cdots \sigma(i_m)} = \langle \mathcal{A},c_{\sigma(i_1)}\otimes \cdots\otimes c_{\sigma(i_m)}\rangle,
\end{equation*}
which equals the $(i_1,\dots,i_m)$th element of $\mathcal{A}(c_{\sigma(1)},\dots,c_{\sigma(\rk)})$.
Therefore, we have
\begin{equation*}
\mathcal{A}(c_{\sigma(1)},\dots,c_{\sigma(\rk)}) = \mathscr{P}_{\sigma}\mathcal{A}(c_1,\dots,c_m) \in \mathcal{K}.
\end{equation*}
\end{proof}

The symmetric group $\mathfrak{S}_{\rk}$ acts on the set $\mathfrak{F}(\mathbb{E})$ by $\sigma \cdot (c_1,\dots,c_{\rk}) = (c_{\sigma(1)},\dots,c_{\sigma(\rk)})$.
Let $\mathfrak{F}_{\rm c}(\mathbb{E})$ be a complete set of representatives of $\mathfrak{S}_{\rk}$-orbits in $\mathfrak{F}(\mathbb{E})$.
Then, using Lemmas~\ref{lem:COP_permutation_invariant} and~\ref{lem:frame_permutation_invariant}, we obtain the following, more concise, characterization of $\COP^{n,m}(\mathbb{E}_+)$, compared with that using Lemma~\ref{lem:COP_spectrum}.

\begin{theorem}\label{thm:permutation_invariant}
Let $\mathcal{K} \subseteq \mathcal{S}^{\rk,m}$ be a permutation-invariant set and consider $\mathfrak{F}_{\rm c}(\mathbb{E}) \subseteq \mathfrak{F} \subseteq \mathfrak{F}(\mathbb{E})$.
Then, the set
\begin{equation*}
\bigcap_{(c_1,\dots,c_{\rk})\in\mathfrak{F}}\{\mathcal{A}\in\mathcal{S}^{n,m}(\mathbb{E})\mid\mathcal{A}(c_1,\dots,c_{\rk})\in\mathcal{K}\}
\end{equation*}
is the same regardless of the choice of $\mathfrak{F}$.
In particular, the claim holds when $\mathcal{K} = \COP^{\rk,m}$, and it follows that
\begin{equation}
\COP^{n,m}(\mathbb{E}_+) = \bigcap_{(c_1,\dots,c_{\rk})\in\mathfrak{F}}\{\mathcal{A}\in\mathcal{S}^{n,m}(\mathbb{E})\mid\mathcal{A}(c_1,\dots,c_{\rk})\in\COP^{\rk,m}\}. \label{eq:COP_characterization_concise}
\end{equation}
\end{theorem}

The idea for providing inner- and outer-approximation hierarchies for $\COP^{n,m}(\mathbb{E}_+)$ is to approximate $\COP^{\rk,m}$ on the right-hand side set in \eqref{eq:COP_characterization_concise} from the inside and outside.

\subsection{Inner-approximation hierarchy}\label{subsec:inner_approx}
Throughout this subsection, we only consider the case of $m = 2$.
Generally, even if $\{\mathcal{I}_r^{\rk}\}_r$ is an inner-approximation hierarchy for $\COP^{\rk,2}$, i.e., $\mathcal{I}_r^{\rk}\uparrow \COP^{\rk,2}$, the sequence obtained by replacing $\COP^{\rk,2}$ on the right-hand side set in \eqref{eq:COP_characterization_concise} with $\mathcal{I}_r^{\rk}$ is not guaranteed to converge to $\COP^{n,2}(\mathbb{E}_+)$.
However, if we choose the polyhedral inner-approximation hierarchy provided by de Klerk and Pasechnik~\cite{dP2002} as $\{\mathcal{I}_r^{\rk}\}_r$, the induced sequence is indeed an inner-approximation hierarchy for $\COP^{n,2}(\mathbb{E}_+)$.
The hierarchy $\{\mathcal{I}_{{\rm dP},r}^{\rk}\}_r$ provided by de Klerk and Pasechnik~\cite{dP2002} is defined as
\begin{equation*}
\mathcal{I}_{{\rm dP},r}^{\rk} \coloneqq \{\bm{A}\in\mathbb{S}^{\rk} \mid (\bm{x}^\top\bm{1})^r\bm{x}^\top\bm{A}\bm{x} \text{ has only nonnegative coefficients}\}
\end{equation*}
and satisfies $\mathcal{I}_{{\rm dP},r}^{\rk} \uparrow \COP^{\rk,2}$.

The following theorem plays an important role in proving that the sequence induced by $\{\mathcal{I}_{{\rm dP},r}^{\rk}\}_r$ converges to $\COP^{n,2}(\mathbb{E}_+)$.

\begin{theorem}[{\cite[Corollary~3.5]{dP2002}}; also see {\cite[Theorem~1]{PR2001}}] \label{thm:upper_bound}
Let $\bm{A} \in \setint(\COP^{\rk,2})$, and set $L \coloneqq \max_{1\le i,j\le \rk}|A_{ij}|$ and $\lambda \coloneqq \min_{\bm{x}\in \Delta_=^{\rk-1}}\bm{x}^\top\bm{A}\bm{x} > 0$.
If $r\in\mathbb{N}$ satisfies $r > L/\lambda -2$, then $\bm{A} \in \mathcal{I}_{{\rm dP},r}^{\rk}$.
\end{theorem}

\begin{proposition}\label{prop:inner_approx}
We fix a set $\mathfrak{F}$ with $\mathfrak{F}_{\rm c}(\mathbb{E}) \subseteq \mathfrak{F} \subseteq \mathfrak{F}(\mathbb{E})$.
Let
\begin{equation}
\mathcal{I}_{{\rm dP},r}^{n}(\mathbb{E}_+) \coloneqq \bigcap_{(c_1,\dots,c_{\rk})\in\mathfrak{F}}\{\mathcal{A}\in\mathcal{S}^{n,2}(\mathbb{E}) \mid \mathcal{A}(c_1,\dots,c_{\rk})\in\mathcal{I}_{{\rm dP},r}^{\rk}\}. \label{eq:inner_approx}
\end{equation}
Then, each $\mathcal{I}_{{\rm dP},r}^{n}(\mathbb{E}_+)$ is a closed convex cone and $\mathcal{I}_{{\rm dP},r}^{n}(\mathbb{E}_+) \uparrow \COP^{n,2}(\mathbb{E}_+)$.
\end{proposition}

\begin{proof}
Given that each $\mathcal{I}_{{\rm dP},r}^{\rk}$ is a closed convex cone and that the mapping $\mathcal{A}\mapsto \langle\mathcal{A},c_i\otimes c_j\rangle$ is linear for each $(c_1,\dots,c_{\rk})\in\mathfrak{F}$ and $i,j=1,\dots,\rk$, $\mathcal{I}_{{\rm dP},r}^{n}(\mathbb{E}_+)$ is a closed convex cone.
In the following, we prove $\mathcal{I}_{{\rm dP},r}^{n}(\mathbb{E}_+) \uparrow \COP^{n,2}(\mathbb{E}_+)$.
As the sequence $\{\mathcal{I}_{{\rm dP},r}^{\rk}\}_r$ satisfies $\mathcal{I}_{{\rm dP},r}^{\rk} \subseteq \mathcal{I}_{{\rm dP},r+1}^{\rk} \subseteq \COP^{\rk,2}$ for all $r\in\mathbb{N}$, it follows from Theorem~\ref{thm:permutation_invariant} that the sequence $\{\mathcal{I}_{{\rm dP},r}^{n}(\mathbb{E}_+)\}_r$ satisfies $\mathcal{I}_{{\rm dP},r}^{n}(\mathbb{E}_+) \subseteq \mathcal{I}_{{\rm dP},r+1}^{n}(\mathbb{E}_+) \subseteq \COP^{n,2}(\mathbb{E}_+)$ for all $r\in\mathbb{N}$.
Next, let $\mathcal{A}\in\setint\COP^{n,2}(\mathbb{E}_+)$.
Using $\mathcal{A}$, we define
\begin{align*}
L(\mathcal{A};c_1,\dots,c_{\rk}) &\coloneqq \max_{i\le i,j\le \rk}|\langle\mathcal{A},c_i\otimes c_j\rangle|,\\
\lambda(\mathcal{A};c_1,\dots,c_{\rk}) &\coloneqq \min_{\bm{x}\in \Delta_=^{\rk-1}}\bm{x}^\top\mathcal{A}(c_1,\dots,c_{\rk})\bm{x}
\end{align*}
for each $(c_1,\dots,c_{\rk})\in\mathfrak{F}(\mathbb{E})$, and also define
\begin{align*}
L(\mathcal{A};\widetilde{\mathfrak{F}}) &\coloneqq \sup_{(c_1,\dots,c_{\rk})\in\widetilde{\mathfrak{F}}}L(\mathcal{A};c_1,\dots,c_{\rk}),\\
\lambda(\mathcal{A};\widetilde{\mathfrak{F}}) &\coloneqq \inf_{(c_1,\dots,c_{\rk})\in\widetilde{\mathfrak{F}}}\lambda(\mathcal{A};c_1,\dots,c_{\rk})
\end{align*}
for $\widetilde{\mathfrak{F}} \in \{\mathfrak{F},\mathfrak{F}(\mathbb{E})\}$.
Given that $\mathcal{A}\in\setint\COP^{n,2}(\mathbb{E}_+)$ and that the sets $\Delta_=^{\rk-1}$ and $\mathfrak{F}(\mathbb{E})$ are compact, we have $L(\mathcal{A};\mathfrak{F}(\mathbb{E})) < +\infty$ and $\lambda(\mathcal{A};\mathfrak{F}(\mathbb{E})) > 0$.
Because $L(\mathcal{A};\mathfrak{F}) \le L(\mathcal{A};\mathfrak{F}(\mathbb{E}))$ and $\lambda(\mathcal{A};\mathfrak{F}) \ge \lambda(\mathcal{A};\mathfrak{F}(\mathbb{E}))$, we obtain $L(\mathcal{A};\mathfrak{F}) < +\infty$ and $\lambda(\mathcal{A};\mathfrak{F}) > 0$.
Now, let $r_0 \coloneqq \lceil L(\mathcal{A};\mathfrak{F})/\lambda(\mathcal{A};\mathfrak{F})\rceil \in \mathbb{N}$ and fix $(c_1,\dots,c_{\rk})\in\mathfrak{F}$ arbitrarily.
Then, because
\begin{equation*}
r_0 > \frac{L(\mathcal{A};\mathfrak{F})}{\lambda(\mathcal{A};\mathfrak{F})} - 2 \ge \frac{L(\mathcal{A};c_1,\dots,c_{\rk})}{\lambda(\mathcal{A};c_1,\dots,c_{\rk})} - 2,
\end{equation*}
Theorem~\ref{thm:upper_bound} implies that $\mathcal{A}(c_1,\dots,c_{\rk}) \in \mathcal{I}_{{\rm dP},r_0}^{\rk}$.
Because $r_0$ is independent of the choice of $(c_1,\dots,c_{\rk})\in\mathfrak{F}$, we obtain $\mathcal{A} \in \mathcal{I}_{{\rm dP},r_0}^n(\mathbb{E}_+) \subseteq \bigcup_{r=0}^{\infty}\mathcal{I}_{{\rm dP},r}^n(\mathbb{E}_+)$.
\end{proof}

\begin{remark}
As it can be seen from the proof of Proposition~\ref{prop:inner_approx}, if a non-decreasing sequence $\{\mathcal{I}_r^{\rk}\}_r$ satisfies $\mathcal{I}_{{\rm dP},r}^{\rk} \subseteq \mathcal{I}_r^{\rk} \subseteq \COP^{\rk,2}$ for all $r \in \mathbb{N}$, the sequence obtained by replacing $\COP^{\rk,2}$ on the right-hand side set in \eqref{eq:COP_characterization_concise} with $\mathcal{I}_r^{\rk}$ is also an inner-approximation hierarchy for $\COP^{n,2}(\mathbb{E}_+)$.
In addition, Proposition~\ref{prop:inner_approx} can be extended to the case of general $m$ by using the polyhedral inner-approximation hierarchy provided by Iqbal and Ahmed~\cite{IA2022}, which is a generalization of that provided by de Klerk and Pasechnik~\cite{dP2002}, for $\COP^{\rk,m}$.
\end{remark}

Note that $\mathcal{I}_{{\rm dP},r}^{n}(\mathbb{E}_+)$ is defined as the intersection of the \textit{infinitely} many sets in general even if $m$, the order of tensors, is limited to 2.
This means that the inner-approximation hierarchy induces a semi-infinite conic constraint.
Initially, the tractability of the approximation hierarchy seems to be the same as the COP cone because $\COP^{n,2}(\mathbb{E}_+)$ can also be described by a semi-infinite constraint.
However, when the symmetric cone $\mathbb{E}_+$ is the direct product of a nonnegative orthant and \textit{one} second-order cone, each $\mathcal{I}_{{\rm dP},r}^{n}(\mathbb{E}_+)$ can be represented by \textit{finitely} many semidefinite constraints.

\subsubsection{Full expression of $\mathcal{I}_{{\rm dP},r}^{n}(\mathbb{E}_+)$}\label{subsubsec:full_expression_dP}
Let $(\mathbb{E}_1,\circ_1,\bullet_1)$ and $(\mathbb{E}_2,\circ_2,\bullet_2)$ be the Euclidean Jordan algebras associated with the nonnegative orthant $\mathbb{R}_+^{n_1}$ and second-order cone $\mathbb{L}^{n_2}$ shown in Examples~\ref{ex:nno} and~\ref{ex:soc}, respectively.
Then, $\mathbb{E} \coloneqq \mathbb{E}_1\times \mathbb{E}_2 = \mathbb{R}^{n_1 + n_2}$ is the Euclidean Jordan algebra with the induced symmetric cone $\mathbb{R}_+^{n_1} \times \mathbb{L}^{n_2}$ and rank $\rk = n_1 + 2$.
We set $n \coloneqq n_1 + n_2$ and reindex $(1,\dots,\rk)$ as $(11,\dots,1n_1,21,22)$, i.e., $1i \coloneqq i$ for $i = 1,\dots,n_1$ and $2i \coloneqq n_1 + i$ for $i = 1,2$.
Sect.~\ref{subsec:outer_approx} uses this notation.
In addition,
\begin{equation}
\mathfrak{F} = \left\{\left(\begin{pmatrix}
\bm{e}_{11}\\
\bm{0}_{n_2}
\end{pmatrix},\dots,\begin{pmatrix}
\bm{e}_{1n_1}\\
\bm{0}_{n_2}
\end{pmatrix},\begin{pmatrix}
\bm{0}_{n_1}\\
1/2\\
\bm{v}/2
\end{pmatrix},\begin{pmatrix}
\bm{0}_{n_1}\\
1/2\\
-\bm{v}/2
\end{pmatrix}\right) \relmiddle| \bm{v}\in S^{n_2-2}\right\} \label{eq:F}
\end{equation}
is a set satisfying $\mathfrak{F}_{\rm c}(\mathbb{E}) \subseteq \mathfrak{F} \subseteq \mathfrak{F}(\mathbb{E})$ (see Example~\ref{ex:nno}, Example~\ref{ex:soc}, and Proposition~\ref{prop:primitive_idempotent}).

Under the identification between $\mathcal{S}^{n,2}$ and $\mathbb{S}^n$, the inner-approximation hierarchy \eqref{eq:inner_approx} with the set \eqref{eq:F} is
\begin{equation*}
\mathcal{I}_{{\rm dP},r}^{n}(\mathbb{E}_+) = \bigcap_{\bm{v}\in S^{n_2-2}}\left\{\bm{A}\in\mathbb{S}^n \relmiddle|
f_r(\bm{x};\bm{A},\bm{v}) \text{ has only nonnegative coefficients}\right\},
\end{equation*}
where
\begin{align}
f(\bm{x};\bm{A},\bm{v}) &\coloneqq \begin{pmatrix}
2\sum_{i=1}^{n_1}x_{1i}\bm{e}_{1i}\\
x_{21} + x_{22}\\
(x_{21} - x_{22})\bm{v}
\end{pmatrix}^\top \bm{A}\begin{pmatrix}
2\sum_{i=1}^{n_1}x_{1i}\bm{e}_{1i}\\
x_{21} + x_{22}\\
(x_{21} - x_{22})\bm{v}
\end{pmatrix}, \label{eq:f}\\
f_r(\bm{x};\bm{A},\bm{v}) &\coloneqq (\bm{x}^\top\bm{1})^rf(\bm{x};\bm{A},\bm{v}).\nonumber
\end{align}
Note that $f(\bm{x};\bm{A},\bm{v})$ is doubled for the convenience of the following calculation; however, the set $\mathcal{I}_{{\rm dP},r}^{n}(\mathbb{E}_+)$ is unchanged.
Let $\bm{A}\in\mathbb{S}^n$ be partitioned as follows:
\begin{equation*}
\bm{A} = \begin{pmatrix}
\bm{A}^{(11)} & \bm{A}^{(121)} & (\bm{A}^{(122)})^\top\\
(\bm{A}^{(121)})^\top & A^{(2121)} & (\bm{A}^{(2122)})^\top\\
\bm{A}^{(122)} & \bm{A}^{(2122)} & \bm{A}^{(2222)}
\end{pmatrix}
\end{equation*}
with $\bm{A}^{(11)}\in\mathbb{S}^{n_1}$, $\bm{A}^{(121)}\in\mathbb{R}^{n_1}$, $\bm{A}^{(122)}\in\mathbb{R}^{(n_2-1)\times n_1}$, $A^{(2121)}\in\mathbb{R}$, $\bm{A}^{(2122)}\in\mathbb{R}^{n_2-1}$, and $\bm{A}^{(2222)}\in\mathbb{S}^{n_2-1}$.
In addition, for such $\bm{A}\in\mathbb{S}^n$ and
\begin{equation*}
\bm{\alpha} = (\underbrace{\alpha_{11},\dots,\alpha_{1n_1}}_{\coloneqq \bm{\alpha}_1},\alpha_{21},\alpha_{22})\in \mathbb{N}^{\rk},
\end{equation*}
we let
\begin{align*}
M^{(11)}(\bm{A},\bm{\alpha}) &\coloneqq 4\{\bm{\alpha}_1^\top\bm{A}^{(11)}\bm{\alpha}_1-\bm{\alpha}_1^\top\diag(\bm{A}^{(11)})\} + 4(\alpha_{21} + \alpha_{22})\bm{\alpha}_1^\top \bm{A}^{(121)} \\
&\quad+ (\alpha_{21}+\alpha_{22})(\alpha_{21}+\alpha_{22}-1)A^{(2121)},\\
\bm{M}^{(21)}(\bm{A},\bm{\alpha}) &\coloneqq 2(\alpha_{21}-\alpha_{22})\bm{A}^{(122)}\bm{\alpha}_1 + (\alpha_{21}-\alpha_{22})(\alpha_{21} + \alpha_{22}-1)\bm{A}^{(2122)}, \\
\bm{M}^{(22)}(\bm{A},\bm{\alpha}) &\coloneqq \{(\alpha_{21} - \alpha_{22})^2 - (\alpha_{21}+\alpha_{22})\}\bm{A}^{(2222)}
\end{align*}
and define
\begin{equation*}
\bm{M}(\bm{A},\bm{\alpha}) \coloneqq \begin{pmatrix}
M^{(11)}(\bm{A},\bm{\alpha}) & \bm{M}^{(21)}(\bm{A},\bm{\alpha})^\top\\
\bm{M}^{(21)}(\bm{A},\bm{\alpha}) & \bm{M}^{(22)}(\bm{A},\bm{\alpha})
\end{pmatrix} \in \mathbb{S}^{n_2},
\end{equation*}
where $\diag(\bm{A}^{(11)})\in\mathbb{R}^{n_1}$ is the vector of the diagonal elements of $\bm{A}^{(11)}$.
Then, after some calculations, we have
\begin{equation*}
f_r(\bm{x};\bm{A},\bm{v}) = \sum_{\bm{\alpha}\in \mathbb{I}_{=r+2}^{\rk}}\frac{r!}{\bm{\alpha}!}\begin{pmatrix}
1\\
\bm{v}
\end{pmatrix}^\top\bm{M}(\bm{A},\bm{\alpha})\begin{pmatrix}
1\\
\bm{v}
\end{pmatrix}\bm{x}^{\bm{\alpha}}. \label{eq:fr_calc}
\end{equation*}
Therefore,
\begin{align}
\mathcal{I}_{{\rm dP},r}^{n}(\mathbb{E}_+) &= \bigcap_{\bm{\alpha}\in \mathbb{I}_{=r+2}^{\rk}} \bigcap_{\bm{v}\in S^{n_2-2}}\left\{\bm{A}\in\mathbb{S}^n \relmiddle| \begin{pmatrix}
1\\
\bm{v}
\end{pmatrix}^\top\bm{M}(\bm{A},\bm{\alpha})\begin{pmatrix}
1\\
\bm{v}
\end{pmatrix} \ge 0\right\} \nonumber\\
&= \bigcap_{\bm{\alpha}\in \mathbb{I}_{=r+2}^{\rk}}\{\bm{A}\in\mathbb{S}^n \mid \bm{M}(\bm{A},\bm{\alpha}) \in \COP(\partial\mathbb{L}^{n_2})\} \nonumber\\
&= \bigcap_{\bm{\alpha}\in \mathbb{I}^{\rk}_{=r+2}}\left\{\bm{A}\in\mathbb{S}^n \relmiddle|
\begin{aligned}
&\text{There exists $t_{\bm{\alpha}}\in\mathbb{R}$ such that}\\
&\bm{M}(\bm{A},\bm{\alpha}) - t_{\bm{\alpha}}\begin{pmatrix}
1 & \bm{0}\\
\bm{0} & -\bm{I}_{n_2-1}
\end{pmatrix} \in \mathbb{S}_+^{n_2}
\end{aligned}
\right\}, \label{eq:characterization_C}
\end{align}
where the last equation follows from \cite[Corollary 6]{SZ2003}.
In summary, $\mathcal{I}_{{\rm dP},r}^{n}(\mathbb{E}_+)$ can be described by $|\mathbb{I}_{=r+2}^{\rk}|$ semidefinite constraints, which is bounded by $\rk^{r+2}$ and whose size is $n_2$.
We call the sequence $\{\mathcal{I}_{{\rm dP},r}^{n}(\mathbb{E}_+)\}_r$ the dP-type inner-approximation hierarchy.

\subsubsection{Concise expression of $\mathcal{I}_{{\rm dP},r}^{n}(\mathbb{E}_+)$}\label{subsubsec:easier_expression_dP}
In this subsubsection, we explain that the expression~\eqref{eq:characterization_C} can be made more concise.
First, we show that the number of constraints in \eqref{eq:characterization_C} can be halved.
For $\bm{\alpha} = (\bm{\alpha}_1,\alpha_{21},\alpha_{22})\in \mathbb{I}_{=r+2}^{\rk}$, let $\widetilde{\bm{\alpha}} \coloneqq (\bm{\alpha}_1,\alpha_{22},\alpha_{21}) \in \mathbb{I}_{=r+2}^{\rk}$.
As $M^{(11)}(\bm{A},\widetilde{\bm{\alpha}}) = M^{(11)}(\bm{A},\bm{\alpha})$, $\bm{M}^{(21)}(\bm{A},\widetilde{\bm{\alpha}}) = -\bm{M}^{(21)}(\bm{A},\bm{\alpha})$, and $\bm{M}^{(22)}(\bm{A},\widetilde{\bm{\alpha}}) = \bm{M}^{(22)}(\bm{A},\bm{\alpha})$ for each $\bm{A}\in \mathbb{S}^n$, we have
\begin{align}
&\begin{pmatrix}
1\\
\bm{v}
\end{pmatrix}^\top\bm{M}(\bm{A},\widetilde{\bm{\alpha}})\begin{pmatrix}
1\\
\bm{v}
\end{pmatrix} \ge 0 \text{ for all $\bm{v}\in S^{n_2-2}$} \nonumber\\
&\quad\Longleftrightarrow \begin{pmatrix}
1\\
\bm{v}
\end{pmatrix}^\top\begin{pmatrix}
M^{(11)}(\bm{A},\bm{\alpha}) & -\bm{M}^{(21)}(\bm{A},\bm{\alpha})^\top\\
-\bm{M}^{(21)}(\bm{A},\bm{\alpha}) & \bm{M}^{(22)}(\bm{A},\bm{\alpha})
\end{pmatrix}
\begin{pmatrix}
1\\
\bm{v}
\end{pmatrix} \ge 0\text{ for all $\bm{v}\in S^{n_2-2}$} \nonumber\\
&\quad\Longleftrightarrow \begin{pmatrix}
1\\
-\bm{v}
\end{pmatrix}^\top\begin{pmatrix}
M^{(11)}(\bm{A},\bm{\alpha}) & \bm{M}^{(21)}(\bm{A},\bm{\alpha})^\top\\
\bm{M}^{(21)}(\bm{A},\bm{\alpha}) & \bm{M}^{(22)}(\bm{A},\bm{\alpha})
\end{pmatrix}
\begin{pmatrix}
1\\
-\bm{v}
\end{pmatrix} \ge 0\text{ for all $\bm{v}\in S^{n_2-2}$}. \label{eq:reduce_num_constraints}
\end{align}
When $\bm{v}$ takes every element of $S^{n_2-2}$, $-\bm{v}$ also takes every element of $S^{n_2-2}$.
Thus, \eqref{eq:reduce_num_constraints} is equivalent to
\begin{equation*}
\begin{pmatrix}
1\\
\bm{v}
\end{pmatrix}^\top\bm{M}(\bm{A},\bm{\alpha})
\begin{pmatrix}
1\\
\bm{v}
\end{pmatrix} \ge 0\text{ for all $\bm{v}\in S^{n_2-2}$}.
\end{equation*}
That is, taking the intersection with respect to $\bm{\alpha}\in \mathbb{I}_{=r+2}^{\rk}$ with $\alpha_{21}\le \alpha_{22}$ in \eqref{eq:characterization_C} is sufficient.
This can be explained by the fact that the ordering of a Jordan frame can be ignored because $\mathcal{I}_{{\rm dP},r}^{\rk}$ is permutation-invariant; thus, we can apply Theorem~\ref{thm:permutation_invariant}.

Next, we show that some semidefinite constraints in \eqref{eq:characterization_C} can be written as second-order cone or non-negativity constraints.
If $\bm{\alpha} = (\bm{\alpha}_1,\alpha_{21},\alpha_{22})\in \mathbb{I}^{\rk}_{=r+2}$ satisfies
\begin{equation*}
\alpha_{21} = \frac{1}{2}k(k-1),\ \alpha_{22} = \frac{1}{2}k(k+1)\ (k = 0,\dots,\lfloor\sqrt{r+2}\rfloor),
\end{equation*}
then $\bm{M}^{(22)}(\bm{A},\bm{\alpha}) = \bm{O}$.
In this case,
\begin{align*}
&\begin{pmatrix}
1\\
\bm{v}
\end{pmatrix}^\top\bm{M}(\bm{A},\bm{\alpha})\begin{pmatrix}
1\\
\bm{v}
\end{pmatrix} \ge 0 \text{ for all $\bm{v}\in S^{n_2-2}$}\\
&\quad\Longleftrightarrow M^{(11)}(\bm{A},\bm{\alpha}) + 2\bm{M}^{(21)}(\bm{A},\bm{\alpha})^\top\bm{v} \ge 0\text{ for all $\bm{v}\in S^{n_2-2}$}\\
&\quad\Longleftrightarrow M^{(11)}(\bm{A},\bm{\alpha}) - 2\|\bm{M}^{(21)}(\bm{A},\bm{\alpha})\|_2 \ge 0\\
&\quad\Longleftrightarrow \begin{pmatrix}
M^{(11)}(\bm{A},\bm{\alpha})\\
2\bm{M}^{(21)}(\bm{A},\bm{\alpha})
\end{pmatrix}\in \mathbb{L}^{n_2}.
\end{align*}
In particular, when $k=0$, i.e., $\alpha_{21} = \alpha_{22} = 0$, as $\bm{M}^{(21)}(\bm{A},\bm{\alpha})$ is also zero, the above second-order cone constraint reduces to the non-negativity constraint $M^{(11)}(\bm{A},\bm{\alpha}) \ge 0$.

Finally, we show that the size of some semidefinite constraints can be reduced by 1.
If $\bm{\alpha} = (\bm{\alpha}_1,\alpha_{21},\alpha_{22})\in \mathbb{I}^{\rk}_{=r+2}$ satisfies $\alpha_{21} = \alpha_{22} \neq 0$, we have $\bm{M}^{(21)}(\bm{A},\bm{\alpha}) = \bm{0}$.
Then,
\begin{align*}
&\begin{pmatrix}
1\\
\bm{v}
\end{pmatrix}^\top\bm{M}(\bm{A},\bm{\alpha})\begin{pmatrix}
1\\
\bm{v}
\end{pmatrix} \ge 0 \text{ for all $\bm{v}\in S^{n_2-2}$}\\
&\quad\Longleftrightarrow M^{(11)}(\bm{A},\bm{\alpha}) + \bm{v}^\top\bm{M}^{(22)}(\bm{A},\bm{\alpha})\bm{v} \ge 0 \text{ for all $\bm{v}\in S^{n_2-2}$}\\
&\quad\Longleftrightarrow M^{(11)}(\bm{A},\bm{\alpha}) + \lambda_{\min}(\bm{M}^{(22)}(\bm{A},\bm{\alpha})) \ge 0\\
&\quad\Longleftrightarrow M^{(11)}(\bm{A},\bm{\alpha})\bm{I}_{n_2-1} + \bm{M}^{(22)}(\bm{A},\bm{\alpha}) \in \mathbb{S}_+^{n_2-1},
\end{align*}
where $\lambda_{\min}(\cdot)$ denotes the minimum eigenvalue for an input matrix.

\subsection{Outer-approximation hierarchy}\label{subsec:outer_approx}
Unlike the case of inner-approximation hierarchies, an outer-approximation hierarchy for $\COP^{\rk,m}$ always induces that for $\COP^{n,m}(\mathbb{E}_+)$.
\begin{proposition}\label{prop:outer_approx}
Let $\{\mathcal{O}_r^{\rk,m}\}_r$ be a sequence such that each $\mathcal{O}_r^{\rk,m}$ is a closed convex cone, and the sequence satisfies the following two conditions:
\begin{enumerate}[(i)]
\item $\mathcal{O}_{r+1}^{\rk,m} \subseteq \mathcal{O}_r^{\rk,m}$ for all $r\in \mathbb{N}$.
\item $\COP^{\rk,m} = \bigcap_{r=0}^{\infty}\mathcal{O}_r^{\rk,m}$. \label{enum:outer_converge}
\end{enumerate}
In the following, the notation ``$\mathcal{O}_r^{\rk,m} \downarrow \COP^{\rk,m}$'' is used to represent the two conditions as in the case of inner-approximation hierarchies.
We fix a set $\mathfrak{F}$ with $\mathfrak{F}_{\rm c}(\mathbb{E}) \subseteq \mathfrak{F} \subseteq \mathfrak{F}(\mathbb{E})$.
Let
\begin{equation}
\mathcal{O}_r^{n,m}(\mathbb{E}_+) \coloneqq \bigcap_{(c_1,\dots,c_{\rk})\in\mathfrak{F}}\{\mathcal{A}\in\mathcal{S}^{n,m}(\mathbb{E}) \mid \mathcal{A}(c_1,\dots,c_{\rk})\in\mathcal{O}_r^{\rk,m}\}. \label{eq:outer_approx}
\end{equation}
Then, each $\mathcal{O}_r^{n,m}(\mathbb{E}_+)$ is a closed convex cone and $\mathcal{O}_r^{n,m}(\mathbb{E}_+) \downarrow \COP^{n,m}(\mathbb{E}_+)$.
\end{proposition}
\begin{proof}
Because $\mathcal{O}_r^{n,m}(\mathbb{E}_+)$ can easily be shown to be a closed convex cone and $\mathcal{O}_{r+1}^{n,m}(\mathbb{E}_+) \subseteq \mathcal{O}_r^{n,m}(\mathbb{E}_+)$ for all $r\in \mathbb{N}$ in the same manner as Proposition~\ref{prop:inner_approx}, we prove only $\COP^{n,m}(\mathbb{E}_+) = \bigcap_{r=0}^{\infty}\mathcal{O}_r^{n,m}(\mathbb{E}_+)$.
The ``$\subseteq$'' part follows from Theorem~\ref{thm:permutation_invariant}.
To prove the ``$\supseteq$'' part, let $\mathcal{A}\in \bigcap_{r=0}^{\infty}\mathcal{O}_r^{n,m}(\mathbb{E}_+)$.
Then, $\mathcal{A}(c_1,\dots,c_{\rk})\in \mathcal{O}_r^{\rk,m}$ for all $r\in\mathbb{N}$ and $(c_1,\dots,c_{\rk})\in \mathfrak{F}$.
By the convergence assumption~\eqref{enum:outer_converge} on $\{\mathcal{O}_r^{\rk,m}\}_r$, we have $\mathcal{A}(c_1,\dots,c_{\rk})\in \COP^{\rk,m}$ for all $(c_1,\dots,c_{\rk})\in \mathfrak{F}$.
Thus, from Theorem~\ref{thm:permutation_invariant}, we obtain $\mathcal{A}\in \COP^{n,m}(\mathbb{E}_+)$.
\end{proof}

\begin{remark}
The set $\bigcup_{r=0}^{\infty}\mathcal{I}_{{\rm dP},r}^n(\mathbb{E}_+)$ involves the intersection with respect to Jordan frames and the union with respect to the depth parameter $r$.
Therefore, to establish the convergence $\setint\COP^{n,2}(\mathbb{E}_+) \subseteq \bigcup_{r=0}^{\infty}\mathcal{I}_{{\rm dP},r}^n(\mathbb{E}_+)$ in the proof of Proposition~\ref{prop:inner_approx}, it is necessary to take $r_0$ such that it is independent of Jordan frames.

On the other hand, the set $\bigcap_{r=0}^{\infty}\mathcal{O}_r^{n,m}(\mathbb{E}_+)$ involves the intersection with respect to the depth parameter as well as Jordan frames.
Thus, as seen in the proof of Proposition~\ref{prop:outer_approx}, the order of the depth parameter and Jordan frames is interchangeable in showing the convergence $\COP^{n,m}(\mathbb{E}_+) = \bigcap_{r=0}^{\infty}\mathcal{O}_r^{n,m}(\mathbb{E}_+)$, which makes the proof easier than the inner-approximation hierarchy.
\end{remark}

As in Proposition~\ref{prop:inner_approx}, the outer-approximation hierarchy induces a semi-infinite conic constraint.
However, as with the inner-approximation hierarchy, when the symmetric cone $\mathbb{E}_+$ is the direct product of a nonnegative orthant and second-order cone and $m = 2$, if we choose an appropriate polyhedral outer-approximation hierarchy as $\{\mathcal{O}_r^{\rk,2}\}_r$ (written as $\{\mathcal{O}_r^{\rk}\}_r$ hereafter), each $\mathcal{O}_r^{n,2}(\mathbb{E}_+)$ (written as $\mathcal{O}_r^n(\mathbb{E}_+)$ hereafter) can be represented by finitely many semidefinite constraints.

\subsubsection{Full expression of $\mathcal{O}_r^n(\mathbb{E}_+)$}
Let $\mathbb{E}$ be the same Euclidean Jordan algebra with the induced symmetric cone $\mathbb{R}_+^{n_1} \times \mathbb{L}^{n_2}$ as defined in Sect.~\ref{subsubsec:full_expression_dP}.
The polyhedral outer-approximation hierarchy $\{\mathcal{O}_r^{\rk}\}_r$ we use is based on a discretization of the standard simplex and written as
\begin{equation}
\mathcal{O}_r^{\rk} = \bigcap_{\bm{x}\in \delta_r^{\rk-1}}\{\bm{A}\in\mathbb{S}^n \mid \bm{x}^\top\bm{A}\bm{x} \ge 0\}, \label{eq:outer_approx_discretization}
\end{equation}
where $\delta_r^{\rk-1}$ is a finite subset of $\Delta_=^{\rk-1}$ for each $r\in \mathbb{N}$.
This type of outer-approximation hierarchy includes, for example, that given by Y{\i}ld{\i}r{\i}m~\cite{Yildirim2012}.
The outer-approximation hierarchy~\eqref{eq:outer_approx} induced by the set~\eqref{eq:F} and polyhedral outer-approximation hierarchy~\eqref{eq:outer_approx_discretization} is
\begin{equation*}
\mathcal{O}_r^n(\mathbb{E}_+) = \bigcap_{\bm{v}\in S^{n_2-2}}\bigcap_{\bm{x}\in \delta_r^{\rk-1}}\{\bm{A}\in\mathbb{S}^n \mid f(\bm{x};\bm{A},\bm{v}) \ge 0\},
\end{equation*}
where $f(\bm{x};\bm{A},\bm{v})$ is defined as \eqref{eq:f}.
Let
\begin{equation*}
\bm{N}(\bm{x},\bm{A}) \coloneqq \sum_{\bm{\alpha}\in \mathbb{I}_{=2}^{\rk}}\frac{1}{\bm{\alpha}!}\bm{M}(\bm{A},\bm{\alpha})\bm{x}^{\bm{\alpha}}.
\end{equation*}
Then,
\begin{align}
\mathcal{O}_r^n(\mathbb{E}_+) &= \bigcap_{\bm{x}\in\delta_r^{\rk-1}}\bigcap_{\bm{v}\in S^{n_2-2}}\left\{\bm{A}\in\mathbb{S}^n \relmiddle| \begin{pmatrix}
1\\
\bm{v}
\end{pmatrix}^\top\bm{N}(\bm{x},\bm{A})\begin{pmatrix}
1\\
\bm{v}
\end{pmatrix} \ge 0\right\} \nonumber\\
&= \bigcap_{\bm{x}\in\delta_r^{\rk-1}}\{\bm{A}\in\mathbb{S}^n \mid \bm{N}(\bm{x},\bm{A}) \in \COP(\partial \mathbb{L}^{n_2})\} \nonumber\\
&= \bigcap_{\bm{x}\in\delta_r^{\rk-1}}\left\{\bm{A}\in\mathbb{S}^n \relmiddle|\begin{aligned}
&\text{There exists $t_{\bm{x}}\in\mathbb{R}$ such that}\\
&\bm{N}(\bm{x},\bm{A}) - t_{\bm{x}}\begin{pmatrix}
1 & \bm{0}\\
\bm{0} & -\bm{I}_{n_2-1}
\end{pmatrix}
\in \mathbb{S}_+^{n_2}
\end{aligned}
\right\}. \label{eq:characterization_O}
\end{align}
In summary, $\mathcal{O}_r^n(\mathbb{E}_+)$ can be described by the $|\delta_r^{\rk-1}|$ semidefinite constraints of size $n_2$.
In particular, if we use the outer-approximation hierarchy $\{\mathcal{O}_r^{\rk}\}_r$ proposed by Y{\i}ld{\i}r{\i}m~\cite{Yildirim2012}, then $\delta_r^{\rk-1}$ is given by
\begin{equation}
\delta_r^{\rk-1} = \bigcup_{k=0}^r\{\bm{x}\in\Delta_=^{\rk-1}\mid (k+2)\bm{x}\in\mathbb{N}^{\rk}\}, \label{eq:delta_Yildirim}
\end{equation}
and $|\delta_r^{\rk-1}|$ is bounded by $\rk^2(\frac{\rk^{r+1}-1}{\rk-1})$, which is polynomial in $\rk$ for every fixed $r\in \mathbb{N}$ (see~\cite[Eq.~(10)]{Yildirim2012}).
We call the sequence $\{\mathcal{O}_r^n(\mathbb{E}_+)\}_r$ obtained by exploiting the hierarchy given by Y{\i}ld{\i}r{\i}m~\cite{Yildirim2012} the Y{\i}ld{\i}r{\i}m-type outer-approximation hierarchy.

\subsubsection{Concise expression of $\mathcal{O}_r^n(\mathbb{E}_+)$}\label{subsubsec:easier_expression_Yildirim}
As with the inner-approximation hierarchy, we can make the expression \eqref{eq:characterization_O} more concise.
Let $\bm{N}(\bm{x},\bm{A})$ be partitioned as follows:
\begin{equation*}
\bm{N}(\bm{x},\bm{A}) = \begin{pmatrix}
N^{(11)}(\bm{x},\bm{A}) & \bm{N}^{(21)}(\bm{x},\bm{A})^\top\\
\bm{N}^{(21)}(\bm{x},\bm{A}) & \bm{N}^{(22)}(\bm{x},\bm{A})
\end{pmatrix},
\end{equation*}
where $N^{(11)}(\bm{x},\bm{A})\in\mathbb{R}$, $\bm{N}^{(21)}(\bm{x},\bm{A})\in\mathbb{R}^{n_2-1}$, and $\bm{N}^{(22)}(\bm{x},\bm{A})\in\mathbb{S}^{n_2-1}$ and $\bm{x} = (\bm{x}_1,x_{21},x_{22})\in\mathbb{R}^{\rk}$.
Then, we have
\begin{align*}
N^{(11)}(\bm{x},\bm{A}) &= 4\bm{x}_1^\top \bm{A}^{(11)}\bm{x}_1 + 4(x_{21}+x_{22})\bm{x}_1^\top\bm{A}^{(121)} + (x_{21}+x_{22})^2 A^{(2121)}\\
\bm{N}^{(21)}(\bm{x},\bm{A}) &= 2(x_{21}-x_{22})\bm{A}^{(122)}\bm{x}_1 + (x_{21}^2-x_{22}^2)\bm{A}^{(2122)}\\
\bm{N}^{(22)}(\bm{x},\bm{A}) &= (x_{21}-x_{22})^2\bm{A}^{(2222)}.
\end{align*}

First, the number of constraints in \eqref{eq:characterization_O} may be reduced.
Suppose that $\delta_r^{\rk-1}$ is permutation-invariant in the sense of Definition~\ref{def:permutation_invariant}.
(Note that $\mathcal{S}^{\rk,1} = \mathbb{R}^{\rk}$.)
For example, \eqref{eq:delta_Yildirim} is permutation-invariant.
For $\bm{x} = (\bm{x}_1,x_{21},x_{22})\in\delta_r^{\rk-1}$, let $\widetilde{\bm{x}} \coloneqq (\bm{x}_1,x_{22},x_{21})$, then $\widetilde{\bm{x}}\in\delta_r^{\rk-1}$ because $\delta_r^{\rk-1}$ is permutation-invariant.
Given that $N^{(11)}(\widetilde{\bm{x}},\bm{A}) = N^{(11)}(\bm{x},\bm{A})$, $\bm{N}^{(21)}(\widetilde{\bm{x}},\bm{A}) = -\bm{N}^{(21)}(\bm{x},\bm{A})$, and $\bm{N}^{(22)}(\widetilde{\bm{x}},\bm{A}) = \bm{N}^{(22)}(\bm{x},\bm{A})$ for each $\bm{A}\in\mathbb{S}^n$, taking the intersection with respect to $\bm{x}\in\delta_r^{\rk-1}$ with $x_{21}\le x_{22}$ in \eqref{eq:characterization_O} is sufficient for the same reason as for Sect.~\ref{subsubsec:easier_expression_dP}.

Second, some semidefinite constraints in \eqref{eq:characterization_O} can be written as non-negativity constraints.
If $\bm{x} = (\bm{x}_1,x_{21},x_{22})\in\delta_r^{\rk-1}$ satisfies $x_{21} = x_{22}$, then we have $\bm{N}^{(21)}(\bm{x},\bm{A}) = \bm{0}$ and $\bm{N}^{(22)}(\bm{x},\bm{A}) = \bm{O}$.
Therefore, the constraint
\begin{equation*}
\begin{pmatrix}
1\\
\bm{v}
\end{pmatrix}^\top\bm{N}(\bm{x},\bm{A})\begin{pmatrix}
1\\
\bm{v}
\end{pmatrix} \ge 0 \text{ for all $\bm{v}\in S^{n_2-2}$}
\end{equation*}
reduces to $N^{(11)}(\bm{x},\bm{A}) \ge 0$.

\section{Comparison with other approximation hierarchies}\label{sec:compare_hierarchy}
The previous sections provide new approximation hierarchies applicable to the COP cone $\COP(\mathbb{K})$ over the cone $\mathbb{K} = \mathbb{R}^{n_1}\times \mathbb{L}^{n_2}$.
For such $\mathbb{K}$, those given by Zuluaga et al.~\cite{ZVP2006} and Lasserre~\cite{Lasserre2014} are also applicable to $\COP(\mathbb{K})$.
In this section, we analytically compare some properties of the proposed approximation hierarchies for $\COP(\mathbb{K})$ with other existing hierarchies.
In the following, we set $n \coloneqq n_1 + n_2$ as in Sect.~\ref{sec:approx_COP_spectrum} and reindex $(1,\dots,n)$ as $(11,\dots,1n_1,21,\dots,2n_2)$, i.e., $1i \coloneqq i$ for $i = 1,\dots,n_1$ and $2i \coloneqq n_1 + i$ for $i = 1,\dots,n_2$.
We use this notation for numerical experiments in Sect.~\ref{sec:experiment} as well.

First, the cone $\mathbb{K} = \mathbb{R}^{n_1}\times \mathbb{L}^{n_2}$ can be represented as a semialgebraic set
\begin{equation}
\left\{\bm{x}\in\mathbb{R}^n \relmiddle|
\begin{array}{l}
\phi_i^{(1)}(\bm{x})\coloneqq x_{1i} \ge 0\ (i = 1,\dots,n_1),\\
\phi_{n_1+1}^{(1)}(\bm{x})\coloneqq x_{21} \ge 0,\\
\phi_{n_1+2}^{(1)}(\bm{x})\coloneqq \bm{e}^\top\bm{x} \ge 0,\\
\phi^{(2)}(\bm{x})\coloneqq x_{21}^2 - \sum_{i=2}^{n_2}x_{2i}^2 \ge 0
\end{array}\right\}, \label{eq:nno_soc_semialgebraic}
\end{equation}
where $\bm{e} \coloneqq (\bm{1}_{n_1+1},\bm{0}_{n_2-1})\in \setint(\mathbb{K})$; thus, the inner-approximation hierarchy given by Zuluaga et al.~\cite{ZVP2006} is applicable to $\COP(\mathbb{K})$.
Note that the inequality $\bm{e}^\top\bm{x} \ge 0$ in \eqref{eq:nno_soc_semialgebraic} is redundant but necessary to deriving the hierarchy (see \cite[Assumption~1]{ZVP2006}).
The inner-approximation hierarchy for $\COP(\mathbb{K})$ is summarized as follows:
\begin{theorem}[{\cite[Proposition~17]{ZVP2006}}]
Let
\begin{equation}
E^{n,m}(\mathbb{K}) \coloneqq \conv\left\{\psi^2\prod_{j=1}^k\phi_{i_j}\relmiddle|
\begin{array}{l}
k\in \mathbb{N},\ m-\sum_{j=1}^k\deg(\phi_{i_j})\in \mathbb{N}\text{ is even},\\
\psi\in H^{n,(m-\sum_{j=1}^k\deg(\phi_{i_j}))/2},\\
\phi_{i_j}\in \{\phi_1^{(1)},\dots,\phi_{n_1+2}^{(1)},\phi^{(2)}\}\ (j=1,\dots,k)
\end{array}
\right\} \label{eq:Enm}
\end{equation}
and $\mathcal{K}_{{\rm ZVP},r}(\mathbb{K}) \coloneqq \{\bm{A}\in\mathbb{S}^n\mid (\bm{e}^\top\bm{x})^r\bm{x}^\top\bm{A}\bm{x}\in E^{n,r+2}(\mathbb{K})\}$ for each $r\in\mathbb{N}$.
Then, the sequence $\{\mathcal{K}_{{\rm ZVP},r}(\mathbb{K})\}_r$ satisfies $\mathcal{K}_{{\rm ZVP},r}(\mathbb{K}) \uparrow \COP(\mathbb{K})$.
\end{theorem}
In the following, we call the sequence $\{\mathcal{K}_{{\rm ZVP},r}(\mathbb{K})\}_r$ the ZVP-type inner-approximation hierarchy.
Although the representation~\eqref{eq:Enm} of the set $E^{n,m}(\mathbb{K})$ is somewhat abstract, we can represent it recursively.

\begin{lemma}\label{lem:Enm_recursion}
If $m = 2k$ for some $k\in\mathbb{N}$, then
\begin{align*}
&E^{n,m}(\mathbb{K}) \\
&\quad= \conv\left(\Sigma^{n,2k} \cup \left\{\psi\phi_i^{(1)}\relmiddle|
\begin{array}{l}
i = 1,\dots,n_1+2,\\
\psi\in E^{n,2k-1}(\mathbb{K})
\end{array}
\right\} \cup \{\psi\phi_1^{(2)}\mid \psi\in E^{n,2k-2}(\mathbb{K})\}\right)\\
&\quad= \left\{\psi^{(0)} + \sum_{i=1}^{n_1+2}\psi_i^{(1)}\phi_i^{(1)} + \psi^{(2)}\phi^{(2)} \relmiddle|
\begin{array}{l}
\psi^{(0)}\in\Sigma^{n,2k},\\
\psi_i^{(1)}\in E^{n,2k-1}(\mathbb{K}),\\
\psi^{(2)}\in E^{n,2k-2}(\mathbb{K})
\end{array}
\right\}.
\end{align*}
If $m = 2k + 1$ for some $k\in\mathbb{N}$, then
\begin{align*}
E^{n,m}(\mathbb{K}) &= \conv\left(\left\{\psi\phi_i^{(1)}\relmiddle|
\begin{array}{l}
i = 1,\dots,n_1+2,\\
\psi\in E^{n,2k}(\mathbb{K})
\end{array}
\right\} \cup \{\psi\phi_1^{(2)}\mid \psi\in E^{n,2k-1}(\mathbb{K})\}\right)\\
&= \left\{\sum_{i=1}^{n_1+2}\psi_i^{(1)}\phi_i^{(1)} + \psi^{(2)}\phi^{(2)} \relmiddle|
\begin{array}{l}
\psi_i^{(1)}\in E^{n,2k}(\mathbb{K}),\\
\psi^{(2)}\in E^{n,2k-1}(\mathbb{K})
\end{array}
\right\}.
\end{align*}
\end{lemma}
From Lemma~\ref{lem:Enm_recursion}, we note that each $\mathcal{K}_{{\rm ZVP},r}(\mathbb{K})$ can be described by semidefinite constraints.
More precisely, the size and number of the semidefinite constraints that define the set $E^{n,m}(\mathbb{K})$ can be calculated.
\begin{proposition}\label{prop:ZVP_size_number}
Let
\begin{equation*}
a_m \coloneqq \frac{1}{\sqrt{\rk^2+4}}\left\{\left(\frac{\rk+\sqrt{\rk^2+4}}{2}\right)^{m+1}-\left(\frac{\rk-\sqrt{\rk^2+4}}{2}\right)^{m+1}\right\}
\end{equation*}
for each $m\in\mathbb{N}$.
Note that $a_m$ is the $m$th term of the recurrence $a_{m+2} = \rk a_{m+1} + a_m$ with initial conditions $a_0 = 1$ and $a_1 = \rk$ and is of order $O(n_1^m)$.
If $m = 2k$ for some $k\in\mathbb{N}$, then $E^{n,m}(\mathbb{K})$ is described by $a_{2i}$ semidefinite constraints of size $|\mathbb{I}_{=k-i}^n|$ $(i = 0,\dots,k)$.
If $m = 2k+1$ for some $k\in\mathbb{N}$, then $E^{n,m}(\mathbb{K})$ is described by $a_{2i+1}$ semidefinite constraints of size $|\mathbb{I}_{=k-i}^n|$ $(i = 0,\dots,k)$.
\end{proposition}

Second, we introduce the outer-approximation hierarchy given by Lasserre~\cite{Lasserre2014}.
Let $\Delta(\mathbb{K}) \coloneqq \{\bm{x}\in\mathbb{K}\mid \bm{e}^\top\bm{x}\le 1\}$, which is a compact set of $\mathbb{R}^n$.
Note that $\bm{A}\in \COP(\mathbb{K})$ if and only if $\bm{A}\in \COP(\Delta(\mathbb{K}))$ for each $\bm{A}\in \mathbb{S}^n$ with a slight abuse of notation.
Let $\nu$ be the finite Borel measure uniformly supported on $\Delta(\mathbb{K})$, i.e., $\nu(B) \coloneqq \int_B 1_{\Delta(\mathbb{K})}d\bm{x}$
for each $B$ in the Borel $\sigma$-algebra of $\mathbb{R}^n$, where $1_{\Delta(\mathbb{K})}$ is the indicator function of $\Delta(\mathbb{K})$, and the notation $d\bm{x}$ represents the Lebesgue measure.
Then, the moment $\bm{y} = (y_{\bm{\alpha}})_{\bm{\alpha}\in\mathbb{N}^n}$ of the measure $\nu$ satisfies
\begin{align}
y_{\bm{\alpha}} &\coloneqq \int_{\mathbb{R}^n}\bm{x}^{\bm{\alpha}}d\nu \nonumber\\
&= \begin{cases}
\frac{2\bm{\alpha}_1!\left(\sum_{i=1}^{n_2}\alpha_{2i}+n_2-1\right)!\prod_{i=2}^{n_2}\Gamma(\beta_{2i})}{\left(\sum_{i=2}^{n_2}\alpha_{2i}+n_2-1\right)(n+|\bm{\alpha}|)!\Gamma(\sum_{i=2}^{n_2}\beta_{2i})} & \text{(if all of $\alpha_{22},\dots,\alpha_{2n_n}$ are even)},\\
0 & \text{(if some of $\alpha_{22},\dots,\alpha_{2n_n}$ are odd)}
\end{cases} \label{eq:moment}
\end{align}
for each $\bm{\alpha} = (\bm{\alpha}_1,\alpha_{21},\dots,\alpha_{2n_2})\in\mathbb{N}^n$, where $\Gamma(\cdot)$ denotes the gamma function and $\beta_{2i} \coloneqq (\alpha_{2i}+1)/2$ for $i = 2,\dots,n_2$.
See Appendix~\ref{apdx:moment} for the calculation of \eqref{eq:moment}.
Using the moment, the outer-approximation hierarchy for $\COP(\mathbb{K})$ given by Lasserre~\cite{Lasserre2014} can be constructed, which we call the Lasserre-type outer-approximation hierarchy.

\begin{theorem}[{\cite[Sect.~2.4]{Lasserre2014}}]
For each $r\in\mathbb{N}$, we define
\begin{equation*}
\mathcal{K}_{{\rm L},r}(\mathbb{K}) \coloneqq \left\{\bm{A}\in\mathbb{S}^n\relmiddle| \bm{M}_r(f_{\bm{A}}\bm{y}) \in \mathbb{S}_+^{\mathbb{I}_{\le r}^n}\right\},
\end{equation*}
where $\bm{M}_r(f_{\bm{A}}\bm{y})$ is the symmetric matrix with the $(\bm{\alpha},\bm{\beta})$th element $\sum_{i,j=1}^nA_{ij}y_{\bm{\alpha}+\bm{\beta}+\bm{e}_i+\bm{e}_j}$ for each $\bm{\alpha},\bm{\beta}\in\mathbb{I}_{\le r}^n$.
Then, the sequence $\{\mathcal{K}_{{\rm L},r}(\mathbb{K})\}_r$ satisfies $\mathcal{K}_{{\rm L},r}(\mathbb{K}) \downarrow \COP(\mathbb{K})$.
\end{theorem}

\renewcommand{\arraystretch}{1.2}
\begin{table}[t]
\caption{Comparison of the number and size of semidefinite constraints defining the $r$th level of each approximation hierarchy $\{\mathcal{K}_r\}_r$ for $\COP(\mathbb{K})$}
\label{tab:comparison_approximation_hierarchy}
\centering
\begin{tabular}{llll}
\hline
Type     & Direction & Number                                                                                                            & Size                                                                                                                                                                                                                      \\
\hline
dP (proposed)      & Inner          & $|\mathbb{I}_{=r+2}^{\rk}|$ & $n_2$ \\
ZVP      & Inner          &  \multicolumn{2}{l}{
\begin{tabular}{l}
Includes at most $\rk$ semidefinite constraints of\\size $|\mathbb{I}^n_{=\lfloor \frac{r}{2}\rfloor+1}|$ (see Proposition~\ref{prop:ZVP_size_number} for details)
\end{tabular}
} \\
NN (proposed)  & Inner          & 1  & $|\mathbb{I}_{=r+2}^n|$                                                                                                                                                                                                                        \\
\hline
Y{\i}ld{\i}r{\i}m (proposed)& Outer & $|\delta_r^{\rk-1}|$ (see \eqref{eq:delta_Yildirim}) & $n_2$ \\
Lasserre & Outer          & 1 & $|\mathbb{I}_{\le r}^n|$\\
\hline
\end{tabular}
\end{table}
\renewcommand{\arraystretch}{1.0}
The above discussion indicates that the three (dP-, ZVP-, and NN-type) inner-approximation hierarchies and two (Y{\i}ld{\i}r{\i}m- and Lasserre-type) outer-approximation hierarchies for $\COP(\mathbb{K})$ are basically described by semidefinite constraints.
Table~\ref{tab:comparison_approximation_hierarchy} summarizes the approximation hierarchies, from which we can observe the characteristics of each approximation hierarchy.
In particular, the dP- and Y{\i}ld{\i}r{\i}m-type approximation hierarchies have features that differ from those of the ZVP-, NN-, and Lasserre-type approximation hierarchies.

First, the number of semidefinite constraints defining the dP- and Y{\i}ld{\i}r{\i}m-type approximation hierarchies is exponential in $r$ but depends only on $n_1$ and not on $n_2$ because $\rk = n_1 + 2$.
In addition, their size is linear in $n_2$.
Thus, they would not be affected much by the increase in $n_2$, and $n_1$ determines to extent to which depth parameter $r$ can be computationally increased.
Conversely, the other approximation hierarchies include semidefinite constraints whose maximum size is exponential in $r$ and dependent on $n = n_1 + n_2$.
Thus, they would be considerably affected by the increase in $n_2$ as well as in $n_1$.
The numerical experiment conducted in Sect.~\ref{sec:experiment} demonstrates this theoretical comparison.

Second, the dP- and Y{\i}ld{\i}r{\i}m-type approximation hierarchies are defined by multiple but small semidefinite constraints, which means that linear conic programming over these hierarchies can be reformulated as SDP with a block diagonal matrix structure.
In this case, we can conduct some of the operations in the primal-dual interior-point methods independently for each block~\cite{FKN1997}, thereby reducing the computational and spatial complexity.

\section{Numerical experiments}\label{sec:experiment}
In this section, we consider the following COPP problem with the COP cone $\COP(\mathbb{K})$ over $\mathbb{K} = \mathbb{R}_+^{n_1}\times \mathbb{L}^{n_2}$:
\begin{equation}
\begin{alignedat}{3}
&\maximize_{y,\bm{S}} && \quad y\\
&\subjectto && \quad y\bm{E}_n + \bm{S} = \bm{C},\\
&&& \quad \bm{S}\in \COP(\mathbb{K}),
\end{alignedat}\label{prob:COPP_dual}
\end{equation}
where $\bm{C}$ is a symmetric positive definite matrix.
Note that the dual problem of \eqref{prob:COPP_dual} is
\begin{equation}
\begin{alignedat}{3}
&\minimize_{\bm{X}} && \quad \langle \bm{C},\bm{X}\rangle\\
&\subjectto && \quad \langle\bm{E}_n,\bm{X}\rangle = 1,\\
&&& \quad \bm{X}\in \CP(\mathbb{K}).
\end{alignedat}\label{prob:COPP_main}
\end{equation}
Both \eqref{prob:COPP_dual} and its dual problem~\eqref{prob:COPP_main} satisfy Slater's condition; thus, \eqref{prob:COPP_dual} is ideal in a sense.

\begin{lemma}\label{lem:strong_feasible}
Both problems~\eqref{prob:COPP_dual} and ~\eqref{prob:COPP_main} satisfy Slater's condition, i.e., they have a feasible interior solution if $\bm{C}$ is a symmetric positive definite matrix.
\end{lemma}
\begin{proof}
Let $y_0 \coloneqq 0$ and $\bm{S}_0 \coloneqq \bm{C}$.
Then, $(y_0,\bm{S}_0)$ is a feasible interior solution of \eqref{prob:COPP_dual}.
Next, let
\begin{align*}
\bm{X}_0' &\coloneqq \begin{pmatrix}
\bm{E}_{n_1} + \bm{I}_{n_1} & \bm{1}_{n_1} & \bm{O}_{n_1\times (n_2-1)}\\
\bm{1}_{n_1}^\top & 2n_2+1 & \bm{0}_{n_2-1}^\top\\
\bm{O}_{(n_2-1)\times n_1} & \bm{0}_{n_2-1} & 2\bm{I}_{n_2-1}
\end{pmatrix},\\
\bm{X}_0 &\coloneqq \frac{1}{n_1^2+3n_1+4n_2-1}\bm{X}_0'
\end{align*}
and we prove that $\bm{X}_0$ is a feasible interior solution of problem~\eqref{prob:COPP_main}.
Let
\begin{align*}
\bm{u}_i^{(1)} &\coloneqq (0,\dots,0,\underbrace{1}_{\text{$i$th}},0\dots,0)\in \mathbb{K} && (i = 11,\dots,1n_1),\\
\bm{u}_i^{(2)} &\coloneqq (\bm{0}_{n_1},1,0,\dots,0,\underbrace{1}_{\text{$i$th}},0,\dots,0)\in \mathbb{K}&& (i = 22,\dots,2n_2),\\
\bm{u}_i^{(3)} &\coloneqq (\bm{0}_{n_1},1,0,\dots,0,\underbrace{-1}_{\text{$i$th}},0,\dots,0)\in \mathbb{K}&& (i = 22,\dots,2n_2),\\
\bm{u}^{(4)} &\coloneqq (\bm{1}_{n_1+1},\bm{0}_{n_2-1})\in \setint(\mathbb{K}).
\end{align*}
Then, they span $\mathbb{R}^n$ as each $\bm{x}\in\mathbb{R}^n$ can be written as
\begin{equation*}
\bm{x} = \sum_{i=11}^{1n_1}x_i\bm{u}_i^{(1)} + \frac{x_{21}}{2}(\bm{u}_{22}^{(2)}+\bm{u}_{22}^{(3)}) + \sum_{i=22}^{2n_2}\frac{x_i}{2}(\bm{u}_i^{(2)}-\bm{u}_i^{(3)}).
\end{equation*}
Therefore, it follows from~\cite[Theorem~3.3]{GS2013} that
\begin{align*}
\bm{X}_0' &=
\sum_{i=11}^{1n_1}\bm{u}_i^{(1)}(\bm{u}_i^{(1)})^\top + \sum_{i=22}^{2n_2}\bm{u}_i^{(2)}(\bm{u}_i^{(2)})^\top + \sum_{i=22}^{2n_2}\bm{u}_i^{(3)}(\bm{u}_i^{(3)})^\top + \bm{u}^{(4)}(\bm{u}^{(4)})^\top \\
&\in \setint\CP(\mathbb{K}).
\end{align*}
Given that $\langle \bm{E}_n,\bm{X}_0'\rangle = n_1^2+3n_1+4n_2-1 > 0$, we obtain $\langle \bm{E}_n,\bm{X}_0\rangle = 1$ and $\bm{X}_0 \in \setint\CP(\mathbb{K})$.
\end{proof}

All experiments in this section were conducted on a computer with an Intel Core i5-8279U 2.40 GHz CPU and 16 GB of memory.
The modeling language YALMIP~\cite{Lofberg2004} (version 20210331), the MOSEK solver~\cite{MOSEK} (version 9.3.3), and MATLAB (R2022a), were used to solve optimization problems.
Based on Lemma~\ref{lem:strong_feasible}, a coefficient matrix $\bm{C}$ in problem~\eqref{prob:COPP_dual} was randomly generated such that it was symmetric positive definite.
We measured three types of time when solving the optimization problems:
\begin{itemize}
\item preparetime: Time taken before calling YALMIP commands \verb|optimize| or \verb|solvesos|.
\item yalmiptime: Time between calling the above commands and beginning to solve an optimization problem in MOSEK.
\item solvertime: Time required to solve an optimization problem in MOSEK.
\end{itemize}
We defined the total time as the sum of the three types of time, and the calculation was considered invalid when the total time exceeded 7200~s.

\subsection{Comparison of approximation hierarchies}\label{subsec:compare_hierarchy}
Table~\ref{tab:comparison_approximation_hierarchy} lists five approximation hierarchies for $\COP(\mathbb{K})$ we have introduced.
Here, we solve optimization problems obtained by replacing the COP cone $\COP(\mathbb{K})$ in \eqref{prob:COPP_dual} with the approximation hierarchies.
For convenience, we hereafter call such problems dP-type approximation problems (of depth $r$), for example.
The YALMIP command \verb|optimize| was used when solving dP-, Y{\i}ld{\i}r{\i}m-, and Lasserre-type approximation problems, and \verb|solvesos| was used when solving ZVP- and NN-type approximation problems.
For each approximation hierarchy, we continuously increased the parameter $r$ that decides the depth of the hierarchy until the total time exceeded 7200~s.
When solving dP- and  Y{\i}ld{\i}r{\i}m-type approximation problems, the concise expressions mentioned in Sects.~\ref{subsubsec:easier_expression_dP} and~\ref{subsubsec:easier_expression_Yildirim} were adopted.
(Sect.~\ref{sect:effect_easier_expression} investigates their numerical effect.)
The pair $(n_1,n_2)$ was set to $(20,5)$, $(5,20)$, and $(5,25)$.

\begin{landscape}
\begin{table}
\caption{Optimal values (optv) of the optimization problems obtained by replacing the COP cone $\COP(\mathbb{R}_+^{20}\times \mathbb{L}^5)$ in \eqref{prob:COPP_dual} with each approximation hierarchy, as well as the  solver time (solt) and total time (tott) required to solve them.
All values are rounded to the second decimal place.
The asterisk * indicates that the total time exceeded 7200~s. Because more than 7200~s were required to solve the dP-type approximation problem of depth 6, only results up to depth $r\le 5$ are shown.}
\label{tab:compare_(n1,n2)=(20,5)}
\centering
\scriptsize
\begin{tabular}{rrrrcrrrcrrrcrrrcrrr}
\hline
  &   \multicolumn{3}{c}{dP}                                                                &  & \multicolumn{3}{c}{ZVP}                                                              &                      & \multicolumn{3}{c}{NN}                                                          &                      & \multicolumn{3}{c}{Y{\i}ld{\i}r{\i}m}                                                           &  & \multicolumn{3}{c}{Lasserre}                                                        \\
\cline{2-4} \cline{6-8} \cline{10-12} \cline{14-16} \cline{18-20}
\multicolumn{1}{c}{$r$} & \multicolumn{1}{c}{optv} & \multicolumn{1}{c}{solt} & \multicolumn{1}{c}{tott} &  & \multicolumn{1}{c}{optv}                   & \multicolumn{1}{c}{solt}                    & \multicolumn{1}{c}{tott}                     &                      & \multicolumn{1}{c}{optv}                   & \multicolumn{1}{c}{solt}                     & \multicolumn{1}{c}{tott}                     &                      & \multicolumn{1}{c}{optv} & \multicolumn{1}{c}{solt} & \multicolumn{1}{c}{tott} &  & \multicolumn{1}{c}{optv}                       & \multicolumn{1}{c}{solt}                  & \multicolumn{1}{c}{tott}                  \\
\hline
0                     & $-\infty$                  & 0.00                        & 0.72                        &                      & 1.52 & 0.04 & 4.20   &                      & 1.52 & 364.56 & 373.31 &                      & 6.09                       & 0.01                        & 0.48                        &                      & $+\infty$ & 0.00 & 0.30    \\
1                     & $-\infty$                  & 0.02                        & 3.24                        &                      & 1.52 & 4.09 & 234.49 &                      &\multicolumn{1}{c}{*}                       &\multicolumn{1}{c}{*}                         &\multicolumn{1}{c}{*}                         &                      & 5.14                       & 0.06                        & 3.53                        &                      & $+\infty$ & 0.00 & 24.85   \\
2                     & $-0.22$                      & 0.16                        & 18.87                       &                      &\multicolumn{1}{c}{*}                       &\multicolumn{1}{c}{*}                       &\multicolumn{1}{c}{*}                         &                      &\multicolumn{1}{c}{*}                       &\multicolumn{1}{c}{*}                         &\multicolumn{1}{c}{*}                         &                      & 4.42                       & 0.36                        & 23.99                       &                      & $+\infty$ & 0.10 & 4110.63 \\
3                     & 0.08                       & 1.85                        & 106.15                      &                      &\multicolumn{1}{c}{*}                       &\multicolumn{1}{c}{*}                       &\multicolumn{1}{c}{*}                         &                      &\multicolumn{1}{c}{*}                       &\multicolumn{1}{c}{*}                         &\multicolumn{1}{c}{*}                         &                      & 3.90                       & 2.09                        & 151.11                      &                      &\multicolumn{1}{c}{*}                           &\multicolumn{1}{c}{*}                       &\multicolumn{1}{c}{*}                          \\
4                     & 0.18                       & 32.60                       & 693.39                      &                      &\multicolumn{1}{c}{*}                       &\multicolumn{1}{c}{*}                       &\multicolumn{1}{c}{*}                         &                      &\multicolumn{1}{c}{*}                       &\multicolumn{1}{c}{*}                         &\multicolumn{1}{c}{*}                         &                      & 3.49                       & 12.62                       & 1200.10                     &                      &\multicolumn{1}{c}{*}                           &\multicolumn{1}{c}{*}                       &\multicolumn{1}{c}{*}                          \\
5                     & 0.26                       & 513.54                      & 6357.99                     &                      &\multicolumn{1}{c}{*}                       &\multicolumn{1}{c}{*}                       &\multicolumn{1}{c}{*}                         &                      &\multicolumn{1}{c}{*}                       &\multicolumn{1}{c}{*}                         &\multicolumn{1}{c}{*}                         &                      & \multicolumn{1}{c}{*}      & \multicolumn{1}{c}{*}       & \multicolumn{1}{c}{*}       &                      &\multicolumn{1}{c}{*}                           &\multicolumn{1}{c}{*}                       & \multicolumn{1}{c}{*}\\
\hline
\end{tabular}
\end{table}
\end{landscape}

\begin{landscape}
\begin{table}
\caption{Optimal values (optv) of the optimization problems obtained by replacing the COP cone $\COP(\mathbb{R}_+^5\times \mathbb{L}^{20})$ in \eqref{prob:COPP_dual} with each approximation hierarchy, as well as the solver time (solt) and total time (tott) required to solve them.
All values are rounded to the second decimal place.
The asterisk * indicates that the total time exceeded 7200~s. Because more than 7200~s were required to solve the dP-type approximation problem of depth 17, only results up to depth $r\le 16$ are shown.}
\label{tab:compare_(n1,n2)=(5,20)}
\centering
\scriptsize
\begin{tabular}{rrrrcrrrcrrrcrrrcrrr}
\hline
  & \multicolumn{3}{c}{dP}                                                                 &  & \multicolumn{3}{c}{ZVP}                                                              &                      & \multicolumn{3}{c}{NN}                                                          &                      & \multicolumn{3}{c}{Y{\i}ld{\i}r{\i}m}                                                           &  & \multicolumn{3}{c}{Lasserre}                                                        \\
\cline{2-4} \cline{6-8} \cline{10-12} \cline{14-16} \cline{18-20}
\multicolumn{1}{c}{$r$} & \multicolumn{1}{c}{optv} & \multicolumn{1}{c}{solt} & \multicolumn{1}{c}{tott} &  & \multicolumn{1}{c}{optv}                   & \multicolumn{1}{c}{solt}                    & \multicolumn{1}{c}{tott}                     &                      & \multicolumn{1}{c}{optv}                   & \multicolumn{1}{c}{solt}                     & \multicolumn{1}{c}{tott}                     &                      & \multicolumn{1}{c}{optv} & \multicolumn{1}{c}{solt} & \multicolumn{1}{c}{tott} &  & \multicolumn{1}{c}{optv}                       & \multicolumn{1}{c}{solt}                  & \multicolumn{1}{c}{tott}                  \\
\hline
0                     & $-\infty$                  & 0.00                        & 0.23                        &                      & 1.79                       & 0.03                        & 1.17                        &  & 1.79                       & 270.90                      & 280.22                      & \multicolumn{1}{l}{} & 2.25                       & 0.01                        & 0.18                        &                      & $+\infty$                   & 0.00                        & 0.35                        \\
1                     & $-\infty$                  & 0.02                        & 0.26                        &                      & 1.79                       & 0.30                        & 4.88                        &  & \multicolumn{1}{c}{*}                         & \multicolumn{1}{c}{*}                          & \multicolumn{1}{c}{*}                          &                      & 2.25                       & 0.03                        & 0.36                        &                      & $+\infty$                   & 0.00                        & 18.29                       \\
2                     & $-\infty$                  & 0.05                        & 0.62                        &                      & 1.79                       & 569.21                      & 751.10                      &  & \multicolumn{1}{c}{*}                         & \multicolumn{1}{c}{*}                          & \multicolumn{1}{c}{*}                          &                      & 2.18                       & 0.09                        & 0.84                        &                      & $+\infty$                   & 0.11                        & 2781.12                     \\
3                     & $-\infty$                  & 0.15                        & 1.32                        &                      & \multicolumn{1}{c}{*}                         & \multicolumn{1}{c}{*}                          & \multicolumn{1}{c}{*}                          &                      & \multicolumn{1}{c}{*}                         & \multicolumn{1}{c}{*}                          & \multicolumn{1}{c}{*}                          &                      & 2.10                       & 0.23                        & 2.06                        &                      & \multicolumn{1}{c}{*}                         & \multicolumn{1}{c}{*}                          & \multicolumn{1}{c}{*}                          \\
4                     & $-0.11$                      & 0.37                        & 2.77                        &                      & \multicolumn{1}{c}{*}                         & \multicolumn{1}{c}{*}                          & \multicolumn{1}{c}{*}                          &                      & \multicolumn{1}{c}{*}                         & \multicolumn{1}{c}{*}                          & \multicolumn{1}{c}{*}                          &                      & 2.02                       & 0.52                        & 4.42                        &                      & \multicolumn{1}{c}{*}                         & \multicolumn{1}{c}{*}                          & \multicolumn{1}{c}{*}                          \\
5                     & 0.40                       & 0.65                        & 4.95                        &                      & \multicolumn{1}{c}{*}                         & \multicolumn{1}{c}{*}                          & \multicolumn{1}{c}{*}                          &                      & \multicolumn{1}{c}{*}                         & \multicolumn{1}{c}{*}                          & \multicolumn{1}{c}{*}                          &                      & 1.95                       & 1.05                        & 10.19                       &                      & \multicolumn{1}{c}{*}                         & \multicolumn{1}{c}{*}                          & \multicolumn{1}{c}{*}                          \\
6                     & 0.91                       & 1.30                        & 9.33                        &                      & \multicolumn{1}{c}{*}                         & \multicolumn{1}{c}{*}                          & \multicolumn{1}{c}{*}                          &                      & \multicolumn{1}{c}{*}                         & \multicolumn{1}{c}{*}                          & \multicolumn{1}{c}{*}                          &                      & 1.87                       & 2.07                        & 22.01                       &                      & \multicolumn{1}{c}{*}                         & \multicolumn{1}{c}{*}                          & \multicolumn{1}{c}{*}                          \\
7                     & 1.19                       & 2.12                        & 17.04                       &                      & \multicolumn{1}{c}{*}                         & \multicolumn{1}{c}{*}                          & \multicolumn{1}{c}{*}                          &                      & \multicolumn{1}{c}{*}                         & \multicolumn{1}{c}{*}                          & \multicolumn{1}{c}{*}                          &                      & 1.81                       & 3.94                        & 49.55                       &                      & \multicolumn{1}{c}{*}                         & \multicolumn{1}{c}{*}                          & \multicolumn{1}{c}{*}                          \\
8                     & 1.30                       & 3.89                        & 32.91                       &                      & \multicolumn{1}{c}{*}                         & \multicolumn{1}{c}{*}                          & \multicolumn{1}{c}{*}                          &                      & \multicolumn{1}{c}{*}                         & \multicolumn{1}{c}{*}                          & \multicolumn{1}{c}{*}                          &                      & 1.79                       & 7.33                        & 110.76                      &                      & \multicolumn{1}{c}{*}                         & \multicolumn{1}{c}{*}                          & \multicolumn{1}{c}{*}                          \\
9                     & 1.37                       & 6.23                        & 62.80                       &                      & \multicolumn{1}{c}{*}                         & \multicolumn{1}{c}{*}                          & \multicolumn{1}{c}{*}                          &                      & \multicolumn{1}{c}{*}                         & \multicolumn{1}{c}{*}                          & \multicolumn{1}{c}{*}                          &                      & 1.79                       & 13.33                       & 261.56                      &                      & \multicolumn{1}{c}{*}                         & \multicolumn{1}{c}{*}                          & \multicolumn{1}{c}{*}                          \\
10                    & 1.42                       & 9.46                        & 118.91                      &                      & \multicolumn{1}{c}{*}                         & \multicolumn{1}{c}{*}                          & \multicolumn{1}{c}{*}                          &                      & \multicolumn{1}{c}{*}                         & \multicolumn{1}{c}{*}                          & \multicolumn{1}{c}{*}                          &                      & 1.79                       & 21.98                       & 603.44                      &                      & \multicolumn{1}{c}{*}                         & \multicolumn{1}{c}{*}                          & \multicolumn{1}{c}{*}                          \\
11                    & 1.45                       & 15.59                       & 230.01                      &                      & \multicolumn{1}{c}{*}                         & \multicolumn{1}{c}{*}                          & \multicolumn{1}{c}{*}                          &                      & \multicolumn{1}{c}{*}                         & \multicolumn{1}{c}{*}                          & \multicolumn{1}{c}{*}                          &                      & 1.79                       & 40.07                       & 1491.77                     &                      & \multicolumn{1}{c}{*}                         & \multicolumn{1}{c}{*}                          & \multicolumn{1}{c}{*}                          \\
12                    & 1.47                       & 22.72                       & 440.28                      &                      & \multicolumn{1}{c}{*}                         & \multicolumn{1}{c}{*}                          & \multicolumn{1}{c}{*}                          &                      & \multicolumn{1}{c}{*}                         & \multicolumn{1}{c}{*}                          & \multicolumn{1}{c}{*}                          &                      & 1.79                       & 63.83                       & 3266.76                     &                      & \multicolumn{1}{c}{*}                         & \multicolumn{1}{c}{*}                          & \multicolumn{1}{c}{*}                          \\
13                    & 1.48                       & 38.53                       & 889.31                      &                      & \multicolumn{1}{c}{*}                         & \multicolumn{1}{c}{*}                          & \multicolumn{1}{c}{*}                          &                      & \multicolumn{1}{c}{*}                         & \multicolumn{1}{c}{*}                          & \multicolumn{1}{c}{*}                          &                      & 1.79                       & 99.29                       & 7158.94                     & \multicolumn{1}{c}{} & \multicolumn{1}{c}{*}                         & \multicolumn{1}{c}{*}                          & \multicolumn{1}{c}{*}                          \\
14                    & 1.49                       & 49.62                       & 1642.49                     &                      & \multicolumn{1}{c}{*}                         & \multicolumn{1}{c}{*}                          & \multicolumn{1}{c}{*}                          &                      & \multicolumn{1}{c}{*}                         & \multicolumn{1}{c}{*}                          & \multicolumn{1}{c}{*}                          &                      & \multicolumn{1}{c}{*}                         & \multicolumn{1}{c}{*}                          & \multicolumn{1}{c}{*}                          & \multicolumn{1}{c}{} & \multicolumn{1}{c}{*}                         & \multicolumn{1}{c}{*}                          & \multicolumn{1}{c}{*}                          \\
15                    & 1.50                       & 69.87                       & 2987.17                     &                      & \multicolumn{1}{c}{*}                         & \multicolumn{1}{c}{*}                          & \multicolumn{1}{c}{*}                          &                      & \multicolumn{1}{c}{*}                         & \multicolumn{1}{c}{*}                          & \multicolumn{1}{c}{*}                          &                      & \multicolumn{1}{c}{*}                         & \multicolumn{1}{c}{*}                          & \multicolumn{1}{c}{*}                          & \multicolumn{1}{c}{} & \multicolumn{1}{c}{*}                         & \multicolumn{1}{c}{*}                          & \multicolumn{1}{c}{*}                          \\
16                    & 1.52                       & 101.21                      & 5274.74                     & \multicolumn{1}{l}{} & \multicolumn{1}{c}{*}                         & \multicolumn{1}{c}{*}                          & \multicolumn{1}{c}{*}                          & \multicolumn{1}{l}{} & \multicolumn{1}{c}{*}                         & \multicolumn{1}{c}{*}                          & \multicolumn{1}{c}{*}                          & \multicolumn{1}{l}{} & \multicolumn{1}{c}{*}                         & \multicolumn{1}{c}{*}                          & \multicolumn{1}{c}{*}                          & \multicolumn{1}{l}{} & \multicolumn{1}{c}{*}                         & \multicolumn{1}{c}{*}                          & \multicolumn{1}{c}{*} \\
\hline
\end{tabular}
\end{table}
\end{landscape}

\begin{landscape}
\begin{table}
\centering
\caption{Optimal values (optv) of the optimization problems obtained by replacing the COP cone $\COP(\mathbb{R}_+^{5}\times \mathbb{L}^{25})$ in \eqref{prob:COPP_dual} with each approximation hierarchy, as well as the solver time (solt) and total time (tott) required to solve them.
All values are rounded to the second decimal place.
The asterisk * indicates that the total time exceeded 7200~s. Because more than 7200~s were required to solve the dP-type approximation problem of depth 16, only results up to depth $r\le 15$ are shown.}
\label{tab:compare_(n1,n2)=(5,25)}
\scriptsize
\begin{tabular}{rrrrcrrrcrrrcrrrcrrr}
\hline
  & \multicolumn{3}{c}{dP}                                                                 &  & \multicolumn{3}{c}{ZVP}                                                              &                      & \multicolumn{3}{c}{NN}                                                          &                      & \multicolumn{3}{c}{Y{\i}ld{\i}r{\i}m}                                                           &  & \multicolumn{3}{c}{Lasserre}                                                        \\
\cline{2-4} \cline{6-8} \cline{10-12} \cline{14-16} \cline{18-20}
\multicolumn{1}{c}{$r$} & \multicolumn{1}{c}{optv} & \multicolumn{1}{c}{solt} & \multicolumn{1}{c}{tott} &  & \multicolumn{1}{c}{optv}                   & \multicolumn{1}{c}{solt}                    & \multicolumn{1}{c}{tott}                     &                      & \multicolumn{1}{c}{optv}                   & \multicolumn{1}{c}{solt}                     & \multicolumn{1}{c}{tott}                     &                      & \multicolumn{1}{c}{optv} & \multicolumn{1}{c}{solt} & \multicolumn{1}{c}{tott} &  & \multicolumn{1}{c}{optv}                       & \multicolumn{1}{c}{solt}                  & \multicolumn{1}{c}{tott}                  \\
\hline
0                     & $-\infty$                  & 0.02                        & 1.14                        &                      & 1.83 & 0.06    & 1.35    &  & 1.83 & 3864.82 & 3904.27 &  & 2.23                       & 0.01                        & 0.18                        &                      & $+\infty$ & 0.00 & 0.33  \\
1                     & $-\infty$                  & 0.02                        & 0.47                        &                      & 1.83 & 0.47    & 5.59    &  &\multicolumn{1}{c}{*}                       &\multicolumn{1}{c}{*}                          &\multicolumn{1}{c}{*}                          &                      & 2.23                       & 0.05                        & 0.38                        &                      & $+\infty$ & 0.00 & 37.90 \\
2                     & $-\infty$                  & 0.09                        & 0.76                        &                      & 1.83 & 6126.23 & 6587.04 &  &\multicolumn{1}{c}{*}                       &\multicolumn{1}{c}{*}                          &\multicolumn{1}{c}{*}                          &                      & 2.18                       & 0.12                        & 0.89                        &                      &\multicolumn{1}{c}{*}                           &\multicolumn{1}{c}{*}                       &\multicolumn{1}{c}{*}                        \\
3                     & $-\infty$                  & 0.17                        & 1.55                        &                      &\multicolumn{1}{c}{*}                       &\multicolumn{1}{c}{*}                          &\multicolumn{1}{c}{*}                          &                      &\multicolumn{1}{c}{*}                       &\multicolumn{1}{c}{*}                          &\multicolumn{1}{c}{*}                          &                      & 2.13                       & 0.38                        & 2.30                        &                      &\multicolumn{1}{c}{*}                           &\multicolumn{1}{c}{*}                       &\multicolumn{1}{c}{*}                        \\
4                     & $-\infty$                  & 0.45                        & 3.85                        &                      &\multicolumn{1}{c}{*}                       &\multicolumn{1}{c}{*}                          &\multicolumn{1}{c}{*}                          &                      &\multicolumn{1}{c}{*}                       &\multicolumn{1}{c}{*}                          &\multicolumn{1}{c}{*}                          &                      & 2.06                       & 0.68                        & 5.00                        &                      &\multicolumn{1}{c}{*}                           &\multicolumn{1}{c}{*}                       &\multicolumn{1}{c}{*}                        \\
5                     & 0.25                       & 1.24                        & 6.70                        &                      &\multicolumn{1}{c}{*}                       &\multicolumn{1}{c}{*}                          &\multicolumn{1}{c}{*}                          &                      &\multicolumn{1}{c}{*}                       &\multicolumn{1}{c}{*}                          &\multicolumn{1}{c}{*}                          &                      & 1.98                       & 1.65                        & 12.36                       &                      &\multicolumn{1}{c}{*}                           &\multicolumn{1}{c}{*}                       &\multicolumn{1}{c}{*}                        \\
6                     & 0.95                       & 2.56                        & 13.70                       &                      &\multicolumn{1}{c}{*}                       &\multicolumn{1}{c}{*}                          &\multicolumn{1}{c}{*}                          &                      &\multicolumn{1}{c}{*}                       &\multicolumn{1}{c}{*}                          &\multicolumn{1}{c}{*}                          &                      & 1.92                       & 3.18                        & 28.62                       &                      &\multicolumn{1}{c}{*}                           &\multicolumn{1}{c}{*}                       &\multicolumn{1}{c}{*}                        \\
7                     & 1.22                       & 4.02                        & 25.25                       &                      &\multicolumn{1}{c}{*}                       &\multicolumn{1}{c}{*}                          &\multicolumn{1}{c}{*}                          &                      &\multicolumn{1}{c}{*}                       &\multicolumn{1}{c}{*}                          &\multicolumn{1}{c}{*}                          &                      & 1.87                       & 5.71                        & 69.53                       &                      &\multicolumn{1}{c}{*}                           &\multicolumn{1}{c}{*}                       &\multicolumn{1}{c}{*}                        \\
8                     & 1.33                       & 5.73                        & 48.05                       &                      &\multicolumn{1}{c}{*}                       &\multicolumn{1}{c}{*}                          &\multicolumn{1}{c}{*}                          &                      &\multicolumn{1}{c}{*}                       &\multicolumn{1}{c}{*}                          &\multicolumn{1}{c}{*}                          &                      & 1.84                       & 12.62                       & 173.91                      &                      &\multicolumn{1}{c}{*}                           &\multicolumn{1}{c}{*}                       &\multicolumn{1}{c}{*}                        \\
9                     & 1.41                       & 11.50                       & 102.73                      &                      &\multicolumn{1}{c}{*}                       &\multicolumn{1}{c}{*}                          &\multicolumn{1}{c}{*}                          &                      &\multicolumn{1}{c}{*}                       &\multicolumn{1}{c}{*}                          &\multicolumn{1}{c}{*}                          &                      & 1.83                       & 22.40                       & 431.75                      &                      &\multicolumn{1}{c}{*}                           &\multicolumn{1}{c}{*}                       &\multicolumn{1}{c}{*}                        \\
10                    & 1.46                       & 15.19                       & 199.37                      &                      &\multicolumn{1}{c}{*}                       &\multicolumn{1}{c}{*}                          &\multicolumn{1}{c}{*}                          &                      &\multicolumn{1}{c}{*}                       &\multicolumn{1}{c}{*}                          &\multicolumn{1}{c}{*}                          &                      & 1.83                       & 38.45                       & 1076.99                     &                      &\multicolumn{1}{c}{*}                           &\multicolumn{1}{c}{*}                       &\multicolumn{1}{c}{*}                        \\
11                    & 1.49                       & 23.34                       & 400.18                      &                      &\multicolumn{1}{c}{*}                       &\multicolumn{1}{c}{*}                          &\multicolumn{1}{c}{*}                          &                      &\multicolumn{1}{c}{*}                       &\multicolumn{1}{c}{*}                          &\multicolumn{1}{c}{*}                          &                      & 1.83                       & 64.26                       & 2689.87                     &                      &\multicolumn{1}{c}{*}                           &\multicolumn{1}{c}{*}                       &\multicolumn{1}{c}{*}                        \\
12                    & 1.52                       & 39.72                       & 780.18                      &                      &\multicolumn{1}{c}{*}                       &\multicolumn{1}{c}{*}                          &\multicolumn{1}{c}{*}                          &                      &\multicolumn{1}{c}{*}                       &\multicolumn{1}{c}{*}                          &\multicolumn{1}{c}{*}                          &                      & 1.83                       & 94.82                       & 6229.90                     &                      &\multicolumn{1}{c}{*}                           &\multicolumn{1}{c}{*}                       &\multicolumn{1}{c}{*}                        \\
13                    & 1.53                       & 57.08                       & 1571.46                     &                      &\multicolumn{1}{c}{*}                       &\multicolumn{1}{c}{*}                          &\multicolumn{1}{c}{*}                          &                      &\multicolumn{1}{c}{*}                       &\multicolumn{1}{c}{*}                          &\multicolumn{1}{c}{*}                          &                      & \multicolumn{1}{c}{*}      & \multicolumn{1}{c}{*}       & \multicolumn{1}{c}{*}       &  &\multicolumn{1}{c}{*}                           &\multicolumn{1}{c}{*}                       &\multicolumn{1}{c}{*}                        \\
14                    & 1.54                       & 81.10                       & 2922.42                     &                      &\multicolumn{1}{c}{*}                       &\multicolumn{1}{c}{*}                          &\multicolumn{1}{c}{*}                          &                      &\multicolumn{1}{c}{*}                       &\multicolumn{1}{c}{*}                          &\multicolumn{1}{c}{*}                          &                      & \multicolumn{1}{c}{*}      & \multicolumn{1}{c}{*}       & \multicolumn{1}{c}{*}       &  &\multicolumn{1}{c}{*}                           &\multicolumn{1}{c}{*}                       &\multicolumn{1}{c}{*}                        \\
15                    & 1.55                       & 113.98                      & 5315.38                     &                      &\multicolumn{1}{c}{*}                       &\multicolumn{1}{c}{*}                          &\multicolumn{1}{c}{*}                          &                      &\multicolumn{1}{c}{*}                       &\multicolumn{1}{c}{*}                          &\multicolumn{1}{c}{*}                          &                      & \multicolumn{1}{c}{*}      & \multicolumn{1}{c}{*}       & \multicolumn{1}{c}{*}       &  &\multicolumn{1}{c}{*}                           &\multicolumn{1}{c}{*}                       &\multicolumn{1}{c}{*}\\
\hline
\end{tabular}
\end{table}
\end{landscape}

Tables~\ref{tab:compare_(n1,n2)=(20,5)},~\ref{tab:compare_(n1,n2)=(5,20)}, and~\ref{tab:compare_(n1,n2)=(5,25)} show the results of $(n_1,n_2) = (20,5)$, $(5,20)$, and $(5,25)$, respectively.
These tables report the solver and total time because some of the approximation hierarchies spent most of the total time before beginning to solve optimization problems in MOSEK.
Although we used YALMIP for convenience, the total time would be substantially reduced if we implemented these approximation hierarchies directly.

As shown in Tables~\ref{tab:compare_(n1,n2)=(5,20)} and~\ref{tab:compare_(n1,n2)=(5,25)}, the optimal values of the ZVP- and NN-type inner-approximation problems agree with those of the Y{\i}ld{\i}r{\i}m-type outer-approximation problems, which implies that the ZVP-, NN-, and Y{\i}ld{\i}r{\i}m-type approximation hierarchies almost approach the optimal value of the original COPP problem \eqref{prob:COPP_dual} for $(n_1,n_2) = (5,20)$ and $(5,25)$.

Although all Lasserre-type outer-approximation problems were considered unbounded, this would result from the moments \eqref{eq:moment} taking infinitesimal values.
Indeed, the MOSEK solver provided a warning about treating nearly zero elements, and we determined that $y_{\bm{\alpha}}$ is approximately $3.52\times 10^{-37}$ for $(n_1,n_2) = (5,25)$ and $\bm{\alpha} = 2\bm{e}_1\in \mathbb{R}^{30}$, for instance.
Hence, the Lasserre-type outer-approximation hierarchy is numerically unstable, whereas the Y{\i}ld{\i}r{\i}m-type outer-approximation hierarchy is numerically stable.

The results of this numerical experiment support the theoretical comparison mentioned in Sect.~\ref{sec:compare_hierarchy}.
As shown in Tables~\ref{tab:compare_(n1,n2)=(5,20)} and~\ref{tab:compare_(n1,n2)=(5,25)}, the dP- and Y{\i}ld{\i}r{\i}m-type approximation hierarchies are not affected much by the increase in $n_2$.
The solver and total time required to solve the dP- and Y{\i}ld{\i}r{\i}m-type approximation problems with $(n_1,n_2) = (5,25)$ is less than twice as long as those with $(n_1,n_2) = (5,20)$ regardless of $r$, except for $r = 0$.
Conversely, the others are considerably affected by the increase in $n_2$.
For example, the solver and total time required to solve the NN-type inner-approximation problem with $(n_1,n_2) = (5,25)$ of depth 0 is more than ten times as long as those with $(n_1,n_2) = (5,20)$.
In addition, although the increase for $r = 0,1$ is mild, those required to solve the ZVP-type inner-approximation problem with $(n_1,n_2) = (5,25)$ of depth 2 is also approximately ten times as long as those with $(n_1,n_2) = (5,20)$.
\footnote{This is because, as shown in Table~\ref{tab:comparison_approximation_hierarchy}, the maximum size of the semidefinite constraints defining the ZVP-type inner-approximation hierarchy increases from $|\mathbb{I}_{=1}^n| = n$ to $|\mathbb{I}_{=2}^n| = n(n+1)/2$ when $r$ increases from 1 to 2.}

As shown in Tables~\ref{tab:compare_(n1,n2)=(20,5)} and~\ref{tab:compare_(n1,n2)=(5,20)}, the dP- and Y{\i}ld{\i}r{\i}m-type approximation hierarchies are considerably affected by the increase in $n_1$.
The solver and total time required to solve the dP- and Y{\i}ld{\i}r{\i}m-type approximation problems with $(n_1,n_2) = (20,5)$ is longer than those with $(n_1,n_2) = (5,20)$ for each $r$, and the difference rapidly increases with $r$.
Because $n = n_1 + n_2$ is the same for the two pairs and because $n_2$ in the pair $(n_1,n_2) = (20,5)$ is smaller than that in the pair $(n_1,n_2) = (5,20)$, we can conclude that the increase in the required time results from the increase in $n_1$.
Conversely, if $n_1$ is small, we can increase depth parameter $r$ and, in this case, the Y{\i}ld{\i}r{\i}m-type outer-approximation hierarchy may approach a nearly optimal value of the COPP problems.

Finally, the ZVP- and NN-type inner-approximation hierarchies provided much tighter bounds than that of the dP-type in all cases, and, as mentioned above, the two hierarchies are guaranteed to approach nearly optimal values even at a depth of 0 for $(n_1,n_2) = (5,20)$ and $(5,25)$.
Moreover, the time required to solve the ZVP-type inner-approximation problems of depth 0 is much shorter than that for the NN-type ones.
Therefore, from a practical perspective, using the zeroth level of the ZVP-type inner-approximation hierarchy may be preferable if aiming to obtain a reasonable lower bound of the optimal value of the original problem \eqref{prob:COPP_dual}.

\subsection{Effect of concise expressions of dP- and Y{\i}ld{\i}r{\i}m-type approximation hierarchies}\label{sect:effect_easier_expression}
In this subsection, we investigate the numerical effect of the concise expressions of the dP- and Y{\i}ld{\i}r{\i}m-type approximation hierarchies mentioned in Sects.~\ref{subsubsec:easier_expression_dP} and~\ref{subsubsec:easier_expression_Yildirim}.
The dP- and Y{\i}ld{\i}r{\i}m-type approximation hierarchies with the full expressions are provided by \eqref{eq:characterization_C} and \eqref{eq:characterization_O}, respectively.
Except for the differences in the expressions of these approximation hierarchies, the experimental settings were the same as those in Sect.~\ref{subsec:compare_hierarchy}.

\begin{table}[t]
\centering
\caption{Solver time (solt) and total time (tott) required to solve the optimization problems obtained by replacing the COP cone $\COP(\mathbb{R}_+^{5}\times \mathbb{L}^{25})$ in \eqref{prob:COPP_dual} with the dP- and  Y{\i}ld{\i}r{\i}m-type approximation hierarchies with the full expressions.
The ``solt ratio" (resp. ``tott ratio") column lists values obtained by dividing ``solt" (resp. ``tott") in this table (the case of not making the expressions more concise) by that in Table~\ref{tab:compare_(n1,n2)=(5,25)} (the case of making the expressions more concise).
The asterisk * indicates that the total time exceeded 7200~s, and the symbol $+\infty$ in the columns of ``solt ratio" and ``tott ratio" indicates that the total time in this table exceeded 7200~s, whereas that in Table~\ref{tab:compare_(n1,n2)=(5,25)} did not.
The ratios that are more than 1 are in bold.
All values are rounded to the second decimal place.}\label{tab:effect_(n1,n2)=(5,25)}
\footnotesize
\begin{tabular}{rrrrrcrrrr}
\hline
  & \multicolumn{4}{c}{dP}                                                                                                &  & \multicolumn{4}{c}{Y{\i}ld{\i}r{\i}m}                                                                                          \\
\cline{2-5} \cline{7-10}
\multicolumn{1}{c}{$r$} & \multicolumn{1}{c}{solt} & \multicolumn{1}{c}{solt ratio} & \multicolumn{1}{c}{tott} & \multicolumn{1}{c}{tott ratio} &  & \multicolumn{1}{c}{solt} & \multicolumn{1}{c}{solt ratio} & \multicolumn{1}{c}{tott} & \multicolumn{1}{c}{tott ratio} \\
\hline
0                     & 0.03                     & \textbf{1.23}                  & 0.71                     & 0.62                  &  & 0.04                     & \textbf{2.97}                  & 0.27                     & \textbf{1.48}                  \\
1                     & 0.08                     & \textbf{3.63}                  & 0.77                     & \textbf{1.64}                  &  & 0.11                     & \textbf{2.34}                  & 0.78                     & \textbf{2.06}                  \\
2                     & 0.26                     & \textbf{2.95}                  & 1.82                     & \textbf{2.40}                  &  & 0.30                     & \textbf{2.58}                  & 2.03                     & \textbf{2.27}                  \\
3                     & 0.57                     & \textbf{3.34}                  & 4.47                     & \textbf{2.88}                  &  & 0.98                     & \textbf{2.60}                  & 5.65                     & \textbf{2.46}                  \\
4                     & 1.17                     & \textbf{2.57}                  & 7.45                     & \textbf{1.93}                  &  & 1.96                     & \textbf{2.86}                  & 13.37                    & \textbf{2.67}                  \\
5                     & 2.34                     & \textbf{1.89}                  & 15.04                    & \textbf{2.25}                  &  & 4.44                     & \textbf{2.70}                  & 36.01                    & \textbf{2.91}                  \\
6                     & 5.17                     & \textbf{2.02}                  & 33.12                    & \textbf{2.42}                  &  & 8.57                     & \textbf{2.70}                  & 92.24                    & \textbf{3.22}                  \\
7                     & 7.81                     & \textbf{1.94}                  & 65.96                    & \textbf{2.61}                  &  & 15.44                    & \textbf{2.71}                  & 258.37                   & \textbf{3.72}                  \\
8                     & 12.25                    & \textbf{2.14}                  & 143.70                   & \textbf{2.99}                  &  & 32.80                    & \textbf{2.60}                  & 700.01                   & \textbf{4.03}                  \\
9                     & 19.66                    & \textbf{1.71}                  & 304.94                   & \textbf{2.97}                  &  & 57.39                    & \textbf{2.56}                  & 1893.38                  & \textbf{4.39}                  \\
10                    & 34.48                    & \textbf{2.27}                  & 690.06                   & \textbf{3.46}                  &  & 95.71                    & \textbf{2.49}                  & 4710.17                  & \textbf{4.37}                  \\
11                    & 50.02                    & \textbf{2.14}                  & 1465.11                  & \textbf{3.66}                  &  & \multicolumn{1}{c}{*}    & \bm{$+\infty$}                 & \multicolumn{1}{c}{*}    & \bm{$+\infty$}                 \\
12                    & 71.78                    & \textbf{1.81}                  & 2943.20                  & \textbf{3.77}                  &  & \multicolumn{1}{c}{*}    & \bm{$+\infty$}                 & \multicolumn{1}{c}{*}    & \bm{$+\infty$}                 \\
13                    & 108.49                   & \textbf{1.90}                  & 5723.38                  & \textbf{3.64}                  &  & \multicolumn{1}{c}{*}    & \multicolumn{1}{c}{*}          & \multicolumn{1}{c}{*}    & \multicolumn{1}{c}{*}          \\
14                    & \multicolumn{1}{c}{*}    & \bm{$+\infty$}                 & \multicolumn{1}{c}{*}    & \bm{$+\infty$}          &  & \multicolumn{1}{c}{*}    & \multicolumn{1}{c}{*}          & \multicolumn{1}{c}{*}    & \multicolumn{1}{c}{*}          \\
15                    & \multicolumn{1}{c}{*}    & \bm{$+\infty$}                 & \multicolumn{1}{c}{*}    & \bm{$+\infty$}          &  & \multicolumn{1}{c}{*}    & \multicolumn{1}{c}{*}          & \multicolumn{1}{c}{*}    & \multicolumn{1}{c}{*}\\
\hline
\end{tabular}
\end{table}
Table~\ref{tab:effect_(n1,n2)=(5,25)} provides the results of $(n_1,n_2)=(5,25)$.
The information of the optimal values is omitted because those of the dP- and Y{\i}ld{\i}r{\i}m-type approximation problems with the full expressions are the same as those provided theoretically (and numerically).
The effect of the concise expressions is significant, and the solver and total time required to solve the dP- and Y{\i}ld{\i}r{\i}m-type approximation problems with the concise expressions are shorter than those with the full expressions, except for the total time of solving the dP-type inner-approximation problem of depth 0.
Note that the effect was also confirmed for $(n_1,n_2) = (20,5)$.

\section{Conclusion and future works}\label{sec:conclusion}
In this study, we provided approximation hierarchies for the COP cone over a symmetric cone and compared them with existing approximation hierarchies.
We first provided the NN-type inner-approximation hierarchy.
Its strength is that the hierarchy permits an SOS representation for a general symmetric cone.
We then provided the dP- and Y{\i}ld{\i}r{\i}m-type approximation hierarchies for the COP matrix cone over the direct product of a nonnegative orthant and second-order cone by exploiting those for the usual COP cone provided by de Klerk and Pasechnik~\cite{dP2002} and Y{\i}ld{\i}r{\i}m~\cite{Yildirim2012}.
Remarkably, they are not affected much by the increase in the size of the second-order cone, unlike the NN-type and existing approximation hierarchies.
Combining the proposed approximation hierarchies with those existing, we obtained nearly optimal values of COPP problems when the size of the nonnegative orthant is small.

Unfortunately, the infinity of a set of Jordan frames is guaranteed to be solved in Sects.~\ref{subsec:inner_approx} and~\ref{subsec:outer_approx} only when the symmetric cone is the direct product of a nonnegative orthant and second-order cone and $m = 2$.
However, as shown by the following proposition, we obtain an outer-approximation hierarchy implementable on a computer for the COP cone over a general symmetric cone by replacing a set $\mathfrak{F}$ of Jordan frames appearing in \eqref{eq:outer_approx} with its finite subset $\mathfrak{F}_r$ such that the union $\bigcup_r\mathfrak{F}_r$ is dense in $\mathfrak{F}$.
The proof is omitted for brevity.

\begin{proposition}
Let $\{\mathcal{O}_r^{\rk,m}\}_r$ be a sequence such that each $\mathcal{O}_r^{\rk,m}$ is a closed convex cone and $\mathcal{O}_r^{\rk,m} \downarrow \COP^{\rk,m}$.
We fix a set $\mathfrak{F}$ with $\mathfrak{F}_{\rm c}(\mathbb{E}) \subseteq \mathfrak{F} \subseteq \mathfrak{F}(\mathbb{E})$ and let $\{\mathfrak{F}_r\}_r$ be a non-decreasing sequence of finite subsets of $\mathfrak{F}$ such that $\bigcup_{r=0}^{\infty}\mathfrak{F}_r$ is dense in $\mathfrak{F}$.
Let
\begin{equation*}
\widehat{\mathcal{O}}_r^{n,m}(\mathbb{E}_+) \coloneqq \bigcap_{(c_1,\dots,c_{\rk})\in\mathfrak{F}_r}\{\mathcal{A}\in\mathcal{S}^{n,m}(\mathbb{E}) \mid \mathcal{A}(c_1,\dots,c_{\rk})\in\mathcal{O}_r^{\rk,m}\}.
\end{equation*}
Then, each $\widehat{\mathcal{O}}_r^{n,m}(\mathbb{E}_+)$ is a closed convex cone and $\widehat{\mathcal{O}}_r^{n,m}(\mathbb{E}_+) \downarrow \COP^{n,m}(\mathbb{E}_+)$.
\end{proposition}

The characterization of $\COP^{n,m}(\mathbb{E}_+)$ provided by Theorem~\ref{thm:permutation_invariant} might be also useful to investigate its geometric properties.
For example, \textit{faces}\footnote{
Recall that a nonempty convex subcone $\mathcal{F}$ of a closed convex cone $\mathcal{K}$ is a face of $\mathcal{K}$ if the following condition holds: for any $x,y\in\mathcal{K}$, if $x+y\in\mathcal{F}$, then $x,y\in\mathcal{F}$.
}
are geometric objects of a closed convex cone, and the facial structure of the usual COP cone has been investigated so far~\cite{Dickinson2011,DC2020}.
The following proposition enables us to study the facial structure of $\COP^{n,m}(\mathbb{E}_+)$ through that of the usual COP cone appearing in the characterization of $\COP^{n,m}(\mathbb{E}_+)$.

\begin{proposition}
Let $\mathcal{F}$ be a face of $\COP^{\rk,m}$.
Then,
\begin{equation*}
\COP^{n,m}(\mathbb{E}_+;\mathcal{F}) \coloneqq \bigcap_{(c_1,\dots,c_{\rk})\in\mathfrak{F}}\{\mathcal{A}\in\mathcal{S}^{n,m}(\mathbb{E})\mid\mathcal{A}(c_1,\dots,c_{\rk})\in\mathcal{F}\}
\end{equation*}
is a face of $\COP^{n,m}(\mathbb{E}_+)$.
\end{proposition}

\begin{proof}
Since $\mathcal{A}(c_1,\dots,c_{\rk})$ is linear with respect to $\mathcal{A} \in \mathcal{S}^{n,m}(\mathbb{E})$ for each $(c_1,\dots,c_{\rk})\in\mathfrak{F}$, we see that $\COP^{n,m}(\mathbb{E}_+;\mathcal{F})$ is a convex subcone of $\COP^{n,m}(\mathbb{E}_+)$.
For $\mathcal{A},\mathcal{B}\in\COP^{n,m}(\mathbb{E}_+)$, we assume that
\begin{equation}
\mathcal{A} + \mathcal{B} \in \COP^{n,m}(\mathbb{E}_+;\mathcal{F}). \label{eq:face_asm}
\end{equation}
Let $(c_1,\dots,c_{\rk})\in\mathfrak{F}$ be arbitrary.
Then, it follows from \eqref{eq:face_asm} that
\begin{equation*}
\mathcal{A}(c_1,\dots,c_{\rk}) + \mathcal{B}(c_1,\dots,c_{\rk}) = (\mathcal{A} + \mathcal{B})(c_1,\dots,c_{\rk}) \in \mathcal{F}.
\end{equation*}
The two tensors $\mathcal{A}(c_1,\dots,c_{\rk})$ and $\mathcal{B}(c_1,\dots,c_{\rk})$ belong to $\COP^{\rk,m}$ since $\mathcal{A},\mathcal{B}\in\COP^{n,m}(\mathbb{E}_+)$.
Given that $\mathcal{F}$ is a face of $\COP^{\rk,m}$, we see that $\mathcal{A}(c_1,\dots,c_{\rk})$ and $\mathcal{B}(c_1,\dots,c_{\rk})$ belong to $\mathcal{F}$.
As $(c_1,\dots,c_{\rk})\in\mathfrak{F}$ is arbitrary, we obtain $\mathcal{A},\mathcal{B}\in \COP^{n,m}(\mathbb{E}_+;\mathcal{F})$.
\end{proof}

Finally, questions arise concerning the inclusion among the dP-, ZVP-, and NN-type inner-approximation hierarchies.
In the numerical experiment conducted in Sect.~\ref{subsec:compare_hierarchy}, the ZVP- and NN-type inner-approximation hierarchies provided considerably tighter bounds than the dP-type.
In the case in which the symmetric cone is a nonnegative orthant, the dP-type inner-approximation hierarchy is well known~\cite{dP2002} to be included in the ZVP- and NN-type ones, which agree with that provided by Parrilo~\cite{Parrilo2000}.
Investigating whether the inclusion also holds where the symmetric cone is the direct product of a nonnegative orthant and second-order cone would be interesting.

\vspace{0.5cm}
\noindent
{\bf Acknowledgments}
The authors thank anonymous reviewers for their useful comments.
This work was supported by Japan Society for the Promotion of Science Grants-in-Aid for Scientific Research (Grant Numbers JP20H02385 and JP22J20008).
This version of the article has been accepted for publication, after peer review but is not the Version of Record and does not reflect post-acceptance improvements, or any corrections.
The Version of Record is available online at: \url{https://doi.org/10.1007/s10898-023-01319-3}.

\section*{Declarations}
{\bf Conflict of interest}
The authors declare that they have no conflicts of interest.

\begin{appendices}
\section{Calculation of \eqref{eq:moment}} \label{apdx:moment}
Note that we can represent the set $\Delta(\mathbb{K})$ as
\begin{equation*}
\left\{\bm{x} = (\bm{x}_1,x_{21},\dots,x_{2n_2}) \in \mathbb{R}^n \relmiddle|
\begin{aligned}
&(\bm{x}_1,x_{21})\in \Delta_{\le}^{n_1+1},\\
&(x_{22},\dots,x_{2n_2}) \in B^{n_2-1}(x_{21})
\end{aligned}
\right\},
\end{equation*}
where
\begin{align*}
\Delta_{\le}^{n_1+1} &\coloneqq \{\bm{z}\in\mathbb{R}_+^{n_1+1} \mid \bm{z}^\top\bm{1} \le 1\},\\
B^{n_2-1}(x_{21}) &\coloneqq \{\bm{z}\in\mathbb{R}^{n_2-1}\mid \|\bm{z}\|_2 \le x_{21}\}.
\end{align*}
Then, for $\bm{\alpha} = (\bm{\alpha}_1,\alpha_{21},\dots,\alpha_{2n_2})\in \mathbb{N}^n$, it follows that
\begin{align*}
y_{\bm{\alpha}} &= \int_{\mathbb{R}^n}\bm{x}^{\bm{\alpha}}d\nu\\
&= \int_{\Delta(\mathbb{K})}\bm{x}^{\bm{\alpha}}d\bm{x}\\
&= \int_{\Delta_{\le}^{n_1+1}}\bm{x}_1^{\bm{\alpha}_1}x_{21}^{\alpha_{21}}\left(\int_{B^{n_2-1}(x_{21})}x_{22}^{\alpha_{22}}\cdots x_{2n_2}^{\alpha_{2n_2}}dx_{22}\cdots dx_{2n_2}\right)d\bm{x}_1dx_{21}.
\end{align*}
From \cite{Folland2001}, we note that
\begin{align}
&\int_{B^{n_2-1}(x_{21})}x_{22}^{\alpha_{22}}\cdots x_{2n_2}^{\alpha_{2n_2}}dx_{22}\cdots dx_{2n_2} \nonumber\\
&\quad = \begin{cases}
\frac{2\prod_{i=2}^{n_2}\Gamma(\beta_{2i})x_{21}^{\sum_{i=2}^{n_2}\alpha_{2i}+n_2-1}}{\left(\sum_{i=2}^{n_2}\alpha_{2i}+n_2-1\right)\Gamma(\sum_{i=2}^{n_2}\beta_{2i})} & \text{(if all of $\alpha_{22},\dots,\alpha_{2n_n}$ are even)},\\
0 & \text{(if some of $\alpha_{22},\dots,\alpha_{2n_n}$ are odd)}.
\end{cases} \label{eq:integral_ball}
\end{align}
As \eqref{eq:integral_ball} implies that $y_{\bm{\alpha}} = 0$ where some of $\alpha_{22},\dots,\alpha_{2n_n}$ are odd, we consider the case in which all $\alpha_{22},\dots,\alpha_{2n_n}$ are even.
In this case, we have
\begin{align*}
y_{\bm{\alpha}} &= \frac{2\prod_{i=2}^{n_2}\Gamma(\beta_{2i})}{\left(\sum_{i=2}^{n_2}\alpha_{2i}+n_2-1\right)\Gamma(\sum_{i=2}^{n_2}\beta_{2i})}\int_{\Delta_{\le}^{n_1+1}}\bm{x}_1^{\bm{\alpha}_1}x_{21}^{\sum_{i=1}^{n_2}\alpha_{2i}+n_2-1}d\bm{x}_1dx_{21}\\
&= \frac{2\bm{\alpha}_1!\left(\sum_{i=1}^{n_2}\alpha_{2i}+n_2-1\right)!\prod_{i=2}^{n_2}\Gamma(\beta_{2i})}{\left(\sum_{i=2}^{n_2}\alpha_{2i}+n_2-1\right)(n+|\bm{\alpha}|)!\Gamma(\sum_{i=2}^{n_2}\beta_{2i})}.
\end{align*}
See \cite{GM1978}, for example, to obtain the last equation.
\end{appendices}


\begin{thebibliography}{99}
\bibitem{AM2019}Ahmadi, A.A., Majumdar, A.: DSOS and SDSOS optimization: more tractable alternatives to sum of squares and semidefinite optimization. SIAM J. Appl. Algebra Geom. {\bf 3}(2), 193--230 (2019). \url{https://doi.org/10.1137/18M118935X}
\bibitem{Alizadeh2012}Alizadeh, F.: An introduction to formally real Jordan algebras and their applications in optimization. In: Anjos, M.F., Lasserre, J.B. (eds.) Handbook on Semidefinite, Conic and Polynomial Optimization, pp. 297--337. Springer, Boston, MA (2012). \url{https://doi.org/10.1007/978-1-4614-0769-0_11}
\bibitem{BMP2016}Bai, L., Mitchell, J.E., Pang, J.-S.: On conic QPCCs, conic QCQPs and completely positive programs. Math. Program. {\bf 159}, 109--136 (2016). \url{https://doi.org/10.1007/s10107-015-0951-9}
\bibitem{Bertsekas1999}Bertsekas, D.P.: Nonlinear Programming, 2nd edn. Athena Scientific, Belmont, MA (1999)
%
\bibitem{BDde2000}Bomze, I.M., D\"{u}r, M., de Klerk, E., Roos, C., Quist, A.J., Terlaky, T.: On copositive programming and standard quadratic optimization problems. J. Glob. Optim. {\bf 18}, 301--320 (2000). \url{https://doi.org/10.1023/A:1026583532263}
%
\bibitem{BV2004}Boyd, S., Vandenberghe, L.: Convex Optimization. Cambridge University Press, Cambridge (2004)
%
\bibitem{BD2009}Bundfuss, S., D\"{u}r, M.: An adaptive linear approximation algorithm for copositive programs. SIAM J. Optim. {\bf 20}(1), 30--53 (2009). \url{https://doi.org/10.1137/070711815}
%
\bibitem{Burer2009}Burer, S.: On the copositive representation of binary and continuous nonconvex quadratic programs. Math. Program. {\bf 120}, 479--495 (2009). \url{https://doi.org/10.1007/s10107-008-0223-z}
\bibitem{CLQ2016}Chen, H., Li, G., Qi, L.: SOS tensor decomposition: theory and applications. Commun. Math. Sci. {\bf 14}(8), 2073--2100 (2016). \url{https://dx.doi.org/10.4310/CMS.2016.v14.n8.a1}
\bibitem{dP2002}de Klerk, E., Pasechnik, D.V.: Approximation of the stability number of a graph via copositive programming. SIAM J. Optim. {\bf 12}(4), 875--892 (2002). \url{https://doi.org/10.1137/S1052623401383248}
\bibitem{Dickinson2011}Dickinson, P.J.C.: Geometry of the copositive and completely positive cones. J. Math. Anal. Appl. {\bf 380}(1), 377--395 (2011). \url{https://doi.org/10.1016/j.jmaa.2011.03.005}
\bibitem{DG2014}Dickinson, P.J.C., Gijben, L.: On the computational complexity of membership problems for the completely positive cone and its dual. Comput. Optim. Appl. {\bf 57}(2), 403--415 (2014). \url{https://doi.org/10.1007/s10589-013-9594-z}
\bibitem{DP2013}Dickinson, P.J.C., Povh, J.: Moment approximations for set-semidefinite polynomials. J. Optim. Theory Appl. {\bf 159}, 57--68 (2013). \url{https://doi.org/10.1007/s10957-013-0279-7}
\bibitem{Dong2013}Dong, H.: Symmetric tensor approximation hierarchies for the completely positive cone. SIAM J. Optim. {\bf 23}(3), 1850--1866 (2013). \url{https://doi.org/10.1137/100813816}
\bibitem{DC2020}Dong, M., Chen, H.: Geometry of the copositive tensor cone and its dual. Asia-Pac. J. Oper. Res. {\bf 37}(4), 2040008 (2020). \url{https://doi.org/10.1142/S0217595920400084}
\bibitem{FK1994}Faraut, J., Kor\'{a}nyi, A.: Analysis on Symmetric Cones. Clarendon Press, Oxford (1994)
\bibitem{Faybusovich1997}Faybusovich, L.: Linear systems in Jordan algebras and primal-dual interior-point algorithms. J. Comput. Appl. Math. {\bf 86}(1), 149--175 (1997). \url{https://doi.org/10.1016/S0377-0427(97)00153-2}
\bibitem{Folland2001}Folland, G.B.: How to integrate a polynomial over a sphere. Am. Math. Mon. {\bf 108}(5), 446--448 (2001). \url{https://doi.org/10.1080/00029890.2001.11919774}
\bibitem{FKN1997}Fujisawa, K., Kojima, M., Nakata, K.: Exploiting sparsity in primal-dual interior-point methods for semidefinite programming. Math. Program. {\bf 79}, 235--253 (1997). \url{https://doi.org/10.1007/BF02614319}
\bibitem{GPS2020}Gouveia, J., Pong, T.K., Saee, M.: Inner approximating the completely positive cone via the cone of scaled diagonally dominant matrices. J. Glob. Optim. {\bf 76}, 383--405 (2020). \url{https://doi.org/10.1007/s10898-019-00861-3}
\bibitem{Gowda2019}Gowda, M.S.: Weighted LCPs and interior point systems for copositive linear transformations on Euclidean Jordan algebras. J. Glob. Optim. {\bf 74}, 285--295 (2019). \url{https://doi.org/10.1007/s10898-019-00760-7}
\bibitem{GS2013}Gowda, M.S., Sznajder, R.: On the irreducibility, self-duality, and non-homogeneity of completely positive cones. Electron. J. Linear Algebra {\bf 26}, 177--191 (2013). \url{https://doi.org/10.13001/1081-3810.1648}
\bibitem{GM1978}Grundmann, A., M\"{o}ller, H.M.: Invariant integration formulas for the $n$-simplex by combinatorial methods. SIAM J. Numer. Anal. {\bf 15}(2), 282--290 (1978). \url{https://doi.org/10.1137/0715019}
%
\bibitem{GDFe2014}Guo, X., Deng, Z., Fang, S.-C., Xing, W.: Quadratic optimization over one first-order cone. J. Ind. Manag. Optim. {\bf 10}(3), 945--963 (2014). \url{https://doi.org/10.3934/jimo.2014.10.945}
%
\bibitem{GL2007}Gvozdenovi\'{c}, N., Laurent, M.: Semidefinite bounds for the stability number of a graph via sums of squares of polynomials. Math. Program. {\bf 110}, 145--173 (2007). \url{https://doi.org/10.1007/s10107-006-0062-8}
\bibitem{HL2006}Henrion, D., Lasserre, J.B.: Convergent relaxations of polynomial matrix inequalities and static output feedback. IEEE Trans. Autom. Control {\bf 51}(2), 192--202 (2006). \url{https://doi.org/10.1109/TAC.2005.863494}
\bibitem{HS2004}Hol, C.W.J., Scherer, C.W.: Sum of squares relaxations for polynomial semi-definite programming. In: Proceedings of the 16th International Symposium on Mathematical Theory of Networks and Systems, pp. 1--10 (2004)
\bibitem{HS2005}Hol, C.W.J., Scherer, C.W.: A sum-of-squares approach to fixed-order $H_{\infty}$-synthesis. In: Henrion, D., Garulli, A. (eds.) Positive Polynomials in Control, pp. 45--71. Springer, Berlin, Heidelberg (2005). \url{https://doi.org/10.1007/10997703_3}
\bibitem{IA2022}Iqbal, M.F., Ahmed, F.: Approximation hierarchies for the copositive tensor cone and their application to the polynomial optimization over the simplex. Mathematics {\bf 10}(10), 1683 (2022). \url{https://doi.org/10.3390/math10101683}
\bibitem{KK2010}Kim, S., Kojima, M.: Solving polynomial least squares problems via semidefinite programming relaxations. J. Glob. Optim. {\bf 46}, 1--23 (2010). \url{https://doi.org/10.1007/s10898-009-9405-3}
\bibitem{KKT2020}Kim, S., Kojima, M., Toh, K.-C.: A geometrical analysis on convex conic reformulations of quadratic and polynomial optimization problems. SIAM J. Optim. {\bf 30}(2), 1251--1273 (2020). \url{https://doi.org/10.1137/19M1237715}
\bibitem{Kojima2003}Kojima, M.: Sums of squares relaxations of polynomial semidefinite programs. Research report B-397, Department of Mathematical and Computing Sciences, Tokyo Institute of Technology, Tokyo (2003)
\bibitem{Lasserre2014}Lasserre, J.B.: New approximations for the cone of copositive matrices and its dual. Math. Program. {\bf 144}, 265--276 (2014). \url{https://doi.org/10.1007/s10107-013-0632-5}
\bibitem{LMNe2018}Li, G., Mordukhovich, B.S., Nghia, T.T.A., Ph\d{a}m, T.S.: Error bounds for parametric polynomial systems with applications to higher-order stability analysis and convergence rates. Math. Program. {\bf 168}, 313--346 (2018). \url{https://doi.org/10.1007/s10107-016-1014-6}
\bibitem{Lofberg2004}L\"{o}fberg, J.: YALMIP: a toolbox for modeling and optimization in MATLAB. In: Proceedings of the 2004 IEEE International Symposium on Computer Aided Control Systems Design, pp. 284--289 (2004). \url{https://doi.org/10.1109/CACSD.2004.1393890}
%
\bibitem{MT2015}Miyashiro, R., Takano, Y.: Mixed integer second-order cone programming formulations for variable selection in linear regression, Eur. J. Oper. Res. {\bf 247}(3), 721--731 (2015). \url{https://doi.org/10.1016/j.ejor.2015.06.081}
%
\bibitem{MOSEK}Mosek: MOSEK Optimization Toolbox for MATLAB. \url{https://www.mosek.com/} (2023). Accessed 18 April 2023
%
\bibitem{NN2022}Nishijima, M., Nakata, K.: A block coordinate descent method for sensor network localization. Optim. Lett. {\bf 16}, 1051--1071 (2022). \url{https://doi.org/10.1007/s11590-021-01762-9}
\bibitem{PA2013}Papp, D., Alizadeh, F.: Semidefinite characterization of sum-of-squares cones in algebras. SIAM J. Optim. {\bf 23}(3), 1398--1423 (2013). \url{https://doi.org/10.1137/110843265}
%
\bibitem{Parrilo2000}Parrilo, P.A.: Structured Semidefinite Programs and Semialgebraic Geometry Methods in Robustness and Optimization. Ph.D. thesis, California Institute of Technology, Pasadena, CA (2000)
%
\bibitem{PVZ2007}Pe\~{n}a, J., Vera, J., Zuluaga, L.F.: Computing the stability number of a graph via linear and semidefinite programming. SIAM J. Optim. {\bf 18}(1), 87--105 (2007). \url{https://doi.org/10.1137/05064401X}
\bibitem{PVZ2015}Pe\~{n}a, J., Vera, J.C., Zuluaga, L.F: Completely positive reformulations for polynomial optimization. Math. Program. {\bf 151}, 405--431 (2015). \url{https://doi.org/10.1007/s10107-014-0822-9}
\bibitem{PR2009}Povh, J., Rendl, F.: Copositive and semidefinite relaxations of the quadratic assignment problem. Discrete Optim. {\bf 6}(3), 231--241 (2009). \url{https://doi.org/10.1016/j.disopt.2009.01.002}
\bibitem{PR2001}Powers, V., Reznick, B.: A new bound for P\'{o}lya's theorem with applications to polynomials positive on polyhedra. J. Pure Appl. Algebra {\bf 164}(1--2), 221--229 (2001). \url{https://doi.org/10.1016/S0022-4049(00)00155-9}
\bibitem{Qi2013}Qi, L.: Symmetric nonnegative tensors and copositive tensors. Linear Algebra Appl. {\bf 439}(1), 228--238 (2013). \url{https://doi.org/10.1016/j.laa.2013.03.015}
\bibitem{QL2017}Qi, L., Luo, Z.: Tensor Analysis: Spectral Theory and Special Tensors. SIAM, Philadelphia, PA (2017). \url{https://doi.org/10.1137/1.9781611974751}
\bibitem{QXX2014}Qi, L., Xu, C., Xu, Y.: Nonnegative tensor factorization, completely positive tensors, and a hierarchical elimination algorithm. SIAM J. Matrix Anal. Appl. {\bf 35}(4), 1227--1241 (2014). \url{https://doi.org/10.1137/13092232X}
\bibitem{Reznick1992}Reznick, B.: Sums of Even Powers of Real Linear Forms. American Mathematical Society, Providence, RI (1992)
\bibitem{Reznick1995}Reznick, B.: Uniform denominators in Hilbert's seventeenth problem. Math Z {\bf 220}, 75--97 (1995). \url{https://doi.org/10.1007/BF02572604}
\bibitem{SB2021}Shaked-Monderer, N., Berman, A.: Copositive and Completely Positive Matrices. World Scientific, Singapore (2021). \url{https://doi.org/10.1142/11386}
\bibitem{SBD2012}Sponsel, J., Bundfuss, S., D\"{u}r, M.: An improved algorithm to test copositivity. J. Glob. Optim. {\bf 52}, 537--551 (2012). \url{https://doi.org/10.1007/s10898-011-9766-2}
%
\bibitem{SZ2003}Sturm, J.F., Zhang, S.: On cones of nonnegative quadratic functions. Math. Oper. Res. {\bf 28}(2), 246--267 (2003). \url{https://doi.org/10.1287/moor.28.2.246.14485}
%
\bibitem{Yildirim2012}Y{\i}ld{\i}r{\i}m, E.A.: On the accuracy of uniform polyhedral approximations of the copositive cone. Optim. Methods Softw. {\bf 27}(1), 155--173 (2012). \url{https://doi.org/10.1080/10556788.2010.540014}
\bibitem{ZF2019}Zhou, A., Fan, J.: A hierarchy of semidefinite relaxations for completely positive tensor optimization problems. J. Glob. Optim. {\bf 75}, 417--437 (2019). \url{https://doi.org/10.1007/s10898-019-00751-8}
\bibitem{ZVP2006}Zuluaga, L.F., Vera, J., Pe\~{n}a, J.: LMI approximations for cones of positive semidefinite forms. SIAM J. Optim. {\bf 16}(4), 1076--1091 (2006). \url{https://doi.org/10.1137/03060151X}
\end{thebibliography}
\end{document}